\documentclass[pdflatex,sn-mathphys-num]{sn-jnl}

\usepackage{graphicx}%
\usepackage{multirow}%
\usepackage{amsmath,amssymb,amsfonts}%
\usepackage{amsthm}%
\usepackage{mathrsfs}%
\usepackage[title]{appendix}%
\usepackage{xcolor}%
\usepackage{textcomp}%
\usepackage{manyfoot}%
\usepackage{booktabs}%
\usepackage{algorithmicx}%
\usepackage{algpseudocode}%
\usepackage{listings}%

\usepackage{subfigure}

\theoremstyle{thmstyleone}%
\newtheorem{theorem}{Theorem}
\newtheorem{proposition}[theorem]{Proposition}%

\theoremstyle{thmstyletwo}%
\newtheorem{remark}{Remark}%

\theoremstyle{thmstylethree}%
\newtheorem{definition}{Definition}%

\newtheorem{lemma}[theorem]{Lemma}

\newtheorem{corollary}{Corollary}

\newtheorem{assumption}{Assumption}

\begin{document}

\iffalse
\newcommand{\comments}[1]{\footnote{\textcolor{blue}{\textit{#1}}}}
\newcommand{\ZH}[1]{\textcolor{green}{#1}}
\newcommand{\AL}[1]{\textcolor{blue}{#1}}
\newcommand{\ZHnew}[1]{\textcolor{red}{#1}}
\else
\newcommand{\comments}[1]{}
\newcommand{\ZH}[1]{#1}
\newcommand{\AL}[1]{#1}
\newcommand{\ZHnew}[1]{#1}
\fi

\title[Article Title]{Convexity of chance constraints for elliptical and skewed distributions with copula structures dependent on decision variables}


\author*[1]{\fnm{Heng} \sur{Zhang}}\email{hengzhang@centralesupelec.fr}

\author[1]{\fnm{Abdel} \sur{Lisser}}\email{abdel.lisser@centralesupelec.fr}

\affil[1]{\orgname{Universit\'e Paris-Saclay, CNRS, CentraleSupelec, Laboratory of signals and systems}, \orgaddress{\street{3, Rue Joliot-Curie}, \postcode{91192}, \city{Gif sur Yvette}, \country{France}}}



\abstract{Chance constraints describe a set of given random inequalities depending on the decision vector satisfied with a large enough probability. They are widely used in decision making under uncertain data in many engineering problems. This paper aims to derive the convexity of chance constraints with row dependent elliptical and skewed random variables via a copula depending on decision vectors. We obtain best thresholds of the $r$-concavity for any real number $r$ and improve probability thresholds of the eventual convexity. We prove the eventual convexity with elliptical distributions and a Gumbel-Hougaard copula despite the copula's singularity near the origin. We determine the $\alpha$-decreasing densities of generalized hyperbolic distributions by estimating the modified Bessel functions. By applying the $\alpha$-decreasing property and a radial decomposition, we achieve the eventual convexity for three types of skewed distributions. Finally, we provide an example to illustrate the eventual convexity of a feasible set containing the origin.}




\maketitle


\section{Introduction}
We consider the following linear optimization problem subject to joint chance constraints
\begin{equation}\label{CCP}
\min c^{T}x \ \text{ s.t. } \ \mathbb{P}\left( V x \leq D \right) \geq p, \ x\in X,  
\end{equation}
where $X \subseteq \mathbb{R}^{N}(N \geq 1)$ is a closed convex set; $D := [D_{1}, ... , D_{N}]^{T}$ is a deterministic vector; $\lbrace v_{i}\rbrace^{K}_{i = 1} \subseteq \mathbb{R}^{N}$ \ZHnew{are} random vectors; $V := [v_{1}, ... , v_{K}]^{T}$ is a random matrix with size $K \times N$. \AL{We call problem (\ref{CCP}) hereafter chance constrained optimization problem (CCO).} The \ZH{feasible set} of \eqref{CCP} is defined as
\begin{equation}\label{Fea-1}
S(p) := \lbrace x\in X : \mathbb{P}(V x \leq D) \geq p \rbrace.
\end{equation}

\ZH{Chance constraints,} \ZHnew{initially introduced} \ZH{by Charnes and Cooper \cite{charnes1963deterministic} and Charnes et al. \cite{charnes1958cost}, describe the probability of events occurring} \ZHnew{beyond} \ZH{a given probability value. Prékopa \cite{prekopa1995stochastic, prekopa2003probabilistic} proposed a general form of chance constraints defined as follows:}
\begin{equation}\label{Introduction 1}
\mathbb{P}\left( h(x, \xi) \leq 0 \right) \geq p,
\end{equation}
where $p \in [0, 1]$ is \ZH{the probability threshold value}, $\xi\in \mathbb{R}^{m}$ is a random vector, $x \in \mathbb{R}^{n}$ is a decision vector, and $h: \mathbb{R}^{n} \times \mathbb{R}^{m} \to \mathbb{R}^{k}$ is a vector-valued mapping. Prékopa \cite[Theorem 10.2.1]{prekopa1995stochastic} proved a \ZH{classical } result: Suppose the density of $\xi$ is log-concave and the components of $-h$ are quasi-concave, then the \ZH{feasible set} of \eqref{Introduction 1} is convex for all $p \in [0, 1]$. However, the function $-h$ can not always be quasi-concave. A counterexample in \cite{kataoka1963stochastic} shows that when $h(x, \xi)= x^{T} \xi$ and $\xi$ subject to a multivariate Gaussian distribution, then the \ZH{feasible set} of \eqref{Introduction 1} is convex if and only if $p \in [1/2, 1]$. Thus, the natural question arises whether the \ZH{feasible set} of \eqref{Introduction 1} is convex when probability $p$ is large enough? We call this convexity property of \ZH{feasible sets} as \textit{eventual convexity} \ZHnew{\cite{henrion2011convexity}}. Observe that the sum of two quasi-concave functions might not be a quasi-concave function. Thus, Henrion and Strugarek \cite{henrion2008convexity} considered the case with a separable constraint, \ZHnew{i.e.} the random vectors appear only on the left-hand side of inequalities:
\begin{equation}\label{Introduction 2}
\mathbb{P}\left( \xi \leq g(x) \right) \geq p.
\end{equation} 
We can see \ZHnew{that} \eqref{Introduction 2} is a specific case of \eqref{Introduction 1} with $h(x, \xi) := \xi - g(x)$. To consider the eventual convexity of chance constraints, Henrion and Strugarek \cite{henrion2008convexity} provided a \ZH{classical } method \ZHnew{to solve \eqref{Introduction 2}}. Let $r > 0$. Suppose that $g$ is a $-r$-concave function and $F$ is the distribution function of $\xi$ with an $(r + 1)$-decreasing density. Then, they proved that the convexity of $F\circ g$ implies the eventual convexity of \eqref{Introduction 2}.
In \cite{van2019eventual}, Van Achooij and Malick considered the case with non-separable inequalities, where $\xi$ as an elliptical random vector was reformulated by using radial density functions. 

In the above results, the inequalities of random vectors are described by single or independent rows. In this way, it is convenient to separate the form of chance constraint from a joint form into several independent \ZHnew{individual constraints}. As \ZHnew{for the dependent case,} copulas are powerful tools to \ZHnew{deal with} joint distributions with arbitrary margins and dependence structures. Copulas are widely used in finance \cite{cherubini2004copula, dewick2022copula} and hydrology \cite{genest2007everything, salvadori2007use}. We refer to \cite{nelsen2006introduction} for an introductory monograph on copula. In \cite{henrion2011convexity}, Henrion and Strugarek first proposed to use copulas to present the \ZHnew{dependence} among \ZHnew{the} rows of a chance \ZH{constrained} random matrix. In \cite{van2016convexity}, Van Ackooij and Oliveira considered the eventual convexity problem under Archimedean copulas (a special family of copulas). Using the supporting hyperplane method, they provided an algorithm to solve eventual concavity problems. Observe that a key idea in \cite{henrion2008convexity} is to guarantee \ZHnew{that} the mapping $z \mapsto F(z^{1/r})$ is concave on a certain interval. Then, Van Ackooij \cite{van2015eventual} extended the $r$-decreasing into a general concept called \textit{$r$-revealed-concavity}. Assuming $h$ is separable and copulas are independent of $x$, Van Ackooij \cite{van2015eventual} provided a sharper probability threshold for \eqref{Introduction 1}. In \ZH{\cite{laguel2022convexity}}, Laguel et al. dealt with the eventual convexity problem with nonlinear mappings $h$ and copulas. By generalizing the $r$-convexity into $G$-convexity, they extended the results of \cite{van2016convexity, van2015eventual} into the case with $x$ \ZHnew{defined in} a Banach space.

Regarding the convexity of chance constraints involving copulas, most attention was paid to the cases \ZHnew{where} copulas \ZHnew{are} independent of decision vectors. Few papers discussed the case \ZHnew{where} copulas depend on decision vectors. Recently, Nguyen, Lisser and Liu \cite{nguyen2023convexity} considered an optimization problem subject to joint chance constraints under a \textit{Gumbel-Hougaard copula} (a special Archimedean copula), where they assumed \ZHnew{that} copula \ZHnew{depends} on \ZHnew{the} decision vector $x$. In order to ensure the existence of copulas, two assumptions about the decision vector $x$ were added: \ZHnew{First}, there exists a neighborhood of the origin not contained in $X$; \ZHnew{Second}, the exponential part of the Gumbel-Hougaard copulas has a positive upper boundary on $X$. This paper extends these two assumptions: \ZHnew{First,} the origin can be contained in \ZHnew{the feasible set of the decision vector $X$}; \ZHnew{Second}, the exponential part of the Gumbel-Hougaard copulas can \ZHnew{extend forward} positive infinity. These two extensions lead to the copula's unboundedness, \ZHnew{which makes} the eventual convexity of chance constraints more \ZHnew{complicated to solve}.

About the eventual convexity problem, there are few papers discussing the boundary of probability threshold $p^{*}$. In \ZH{\cite{minoux2016convexity, minoux2017global}}, several necessary and sufficient conditions were presented for the convexity of probability functions. The forms of these necessary and sufficient conditions \ZH{are non-linear}, so the threshold $p^{*}$ can not be easily induced by these conditions. Thus, in \ZH{\cite{minoux2016convexity, minoux2017global}}, the conditions were strengthened into sufficient conditions to obtain the $p^{*}$. It is of interest to know whether there exists \ZH{a} best threshold $p^{*}$, where the best \ZHnew{value} means that a feasible set of chance constraints is convex for $p \geq p^{*}$. Our results answer this question to some extent. Define $g(x) := (b-\mu^{T} x)/ \sqrt{x^{T}\Sigma x}$. Let $F$ be a distribution function. One part of our work follows the \ZH{classical } method about eventual convexity in \cite{henrion2008convexity}, \ZHnew{i.e.}, we consider the concavity of $F\circ g$. Inspired by \cite{minoux2016convexity}, we derive a necessary and sufficient condition for the $r$-concavity of $g$. Comparing to the $r$-concavity results with $r = 1$ in \cite{van2016convexity} and $r < -1 $ in \cite{cheng2014second}, we extend the range of $r$ to all real numbers. Then, we provide a new geometric approach to compute the best threshold $\theta^{*}$ for the $r$-concavity of $g$, which are all less than the thresholds in \cite{henrion2008convexity, nguyen2023convexity, cheng2014second}. Further, if $F$ is Gaussian distribution, then we can \ZHnew{find} the best threshold $p^{*}$ for the concavity of $F\circ g$, which \ZHnew{improves} the thresholds in \cite{van2016convexity}.

The above results are all based on elliptical distributions, which exhibit strong forms of elliptical symmetry. Thus, using the elliptical distributions for precise descriptions of reality, such as insurance risks of actuarial science \cite{lane2000pricing} and hedge fund returns of finance \cite{davies2009fund}, may not be suitable. In particular, when Gaussian distributions are used, the lack of heavy tails property is also a challenge in data \ZHnew{set} \ZH{analysis}. Thus, the skewness is considered to be incorporated into elliptical distributions via two main approaches. One of the methods is to mix normal distributions with a positive scale-valued random variable, called a normal mean-variance mixture (NMVM) distribution. By adding randomness to the variance, the distributions can have heavy tails. Similarly, introducing randomness into the mean enables the addition of skewness to the distributions. A type of random vector with NMVM distributions can be defined by
\begin{equation}\label{normal mean-variance mixture introduction}
\xi \overset{\textbf{d}}{=} \mu + W \gamma + \sqrt{W}AZ,
\end{equation}
where $Z$ is a multivariate normal distribution, $W$ is a scale-valued random variable, $\gamma \in \mathbb{R}^{N}$ is a skewness parameter, $\mu \in \mathbb{R}^{N}$, $A \in \mathbb{R}^{N \times N}$. In particular, the above NMVM distribution is a Gaussian mixture distribution when $W$ is \ZHnew{either} a discrete random variable, \ZHnew{or} a Gaussian mixture distribution with infinite summation if $W$ is a continuous random variable. When $W$ has a generalized inverse Gaussian distribution, \ZHnew{the random vector} \eqref{normal mean-variance mixture introduction} is well known as the definition of Generalized hyperbolic (GH) distribution. The GH distribution and its extensions have widespread use in portfolio optimization problems \cite{birge2021portfolio, wang2022portfolio}.

Another method is to multiply density functions by some distributions with skewness. This method can better describe asymmetry or heavy tails derived from the sub-normal distributions as an extension of the normal scale mixture distribution \cite[Chapter 1.7]{peel2000finite}. An important distribution of this kind is the skew normal (SN) distribution, with its density defined by 
\begin{equation*}
f(z) := 2 \phi_{N}\left(z; \mu, \Sigma \right) \Phi \left( \lambda^{T} \Sigma^{-1/2}  \left( z - \mu \right)\right),
\end{equation*}
where $\lambda$ is a skewness parameter, $\mu$ is a mean vector, $\Sigma \in \mathbb{R}^{N \times N}$ is a variance matrix, $\phi_{N}$ is a N-variate normal density, $\Phi$ is the cumulative distribution function of standard normal distribution. Further, to construct heavy tails based on SN distributions, some extensions such as scale mixtures of skew normal, skew-t, skew-contaminated normal, and skew-slash distributions are \ZHnew{obtained} by using variance scale mixture with gamma, discrete, and beta distributions respectively \cite{da2011skew}. The SN distribution and its extensions are widely used in risk management \cite{vernic2006multivariate}, asset pricing \cite{carmichael2013asset}, and portfolio selection \cite{adcock2014mean}. We refer the reader to \cite{azzalini2013skew, davila2018finite} for more details about skewed distributions and the references therein.

In recent years, the CCO with skewed distributions has attracted widespread attention. Among all the skew distributions in CCO, Gaussian mixture distributions have gained the most interest and are extensively applied to power and energy dispatching problems \cite{ke2015novel, yang2019analytical, shi2022day}. Only a few papers explore the CCO with skewed distributions. Peng et al., \cite{peng2021chance} considered a chance constrained game (CCG) with a finite mixture of elliptical distributions. They proved the existence of a Nash equilibrium of the CCG by using a fixed-point theorem. Further, Nguyen, Lisser and Singh \cite{nguyen2024random} extended the result into the case with a NMVM distribution. Considering a CCO with Gaussian mixture distributions, Lasserre and Welsser \cite{lasserre2021distributionally} provided a sequence of inner approximations about its \ZH{feasible set} and proved the strong asymptotic convergence of the sequence by using a Moment-Sums-of-Squares method. Tong, Subramanyam and Rao \cite{tong2022optimization} provided a sampling-free method based on large deviation theory to solve a rare CCO with Gaussian mixture distributions. Liu et al., \cite{liu2022chance} presented the first result about the eventual convexity of chance constrained problems with Gaussian mixture distributions. They provided a sequential convex approximation method to solve the CCO and its extension \ZHnew{to} GH distributions, which can be regarded as a  Gaussian mixture distribution with infinite summation.

To the best of our knowledge, no attempt was made to derive the eventual convexity of chance constraints with skewed distributions except in the cases with Gaussian mixture distributions \cite{liu2022chance}. The key challenge in studying general skew distributions \ZHnew{lies in that} complexity, the intricate form of Bessel functions \ZHnew{embedded} within the definition of generalized hyperbolic distributions. \ZHnew{In this paper, we} prove the eventual convexity of chance constraints with general NMVM distributions and its special form GH distributions under mild assumptions. We prove the $\alpha$-decreasing property of GH distributions' densities, extending the eventual convexity of the elliptical cases to the GH distribution situation with a skewness parameter orthogonal to the domain of decision vectors. Using a radial decomposition method, we \ZHnew{show} the eventual convexity of the NMVM distribution's case with bounded scale-valued mixture random vectors. Finally, by combining the above two eventual convexity results about skewed distribution cases, we \ZHnew{obtain a} more general eventual convexity of GH distributions' case with skewness parameter vector non-orthogonal to the domain of decision vectors.

The paper is organized as follows: Section \ref{Preliminaries} presents some preliminaries about symmetric and asymmetric distributions, copulas, and generalized concavity. In Section \ref{Reformulation of chance constraints}, a reformulation of chance constraints is indicated. In Section \ref{Concavity of H_{i}(x)}, the concavity of an auxiliary function $F\circ g$ is proved, and the existence of best thresholds to the concavity is stated. Section \ref{Convexity of Ui(x, yi)} is devoted to show the convexity of \ZHnew{an additional} auxiliary function. In Section \ref{Eventual convexity with elliptical distributions}, we derive the convexity of feasible sets with elliptical symmetric distributions. In Section \ref{Eventual convexity with skewed distributions}, we present the results about the eventual convexity with several different kinds of skewed distributions. In Section \ref{Examples}, we provide an approach to create a strictly convex, positive, and second differentiable function. It can be seen that a specific copula dependent on the decision vector $x$ is available. Then, we present a simulation about the eventual convexity under elliptical distributions. Section \ref{Examples2} provides some examples to show the threshold $\sqrt{\theta^{*}}$ obtained \ZHnew{by our main results.}

\section{Eventual convexity with elliptical distributions and copulas}\label{Main results about eventual convexity with elliptical distributions}
In this section, we presents our main results about eventual convexity under elliptical distributions. In Section \ref{Preliminaries}, we briefly introduce some basic concepts \ZHnew{related to elliptical distributions and copula}. Section \ref{Reformulation of chance constraints} contains a brief discussion of the densities of GH distributions and provides a reformulation of the chance constraints problem \eqref{CCP}. Sections \ref{Concavity of H_{i}(x)} and \ref{Convexity of Ui(x, yi)} \ZHnew{show} the concavity and convexity of two auxiliary functions, respectively. \ZHnew{The latter} are used to \ZHnew{prove} the convexity of the feasible set of the copula \ZHnew{based} chance constraint with elliptical distributions in Section \ref{Eventual convexity with elliptical distributions}.

\subsection{Preliminaries}\label{Preliminaries}
We start by recalling the concept of elliptical distributions. For more details, we refer the reader to \cite[Chapter 2]{fang1990symmetric} as a standard reference.
\begin{definition} An $N$-dimensional random vector $\varsigma$ has a spherical distribution \ZHnew{if and only if} its characteristic function $\Psi(t)$ satisfies the following condition:
\item There exists a function $\Phi(\cdot) : \mathbb{R} \rightarrow \mathbb{R}$, called the \textbf{characteristic generator} of the spherical distribution, such that
\begin{equation*}
\Psi(t) = E(e^{it^{T}\varsigma}) = \Phi(\| t \|^{2}).
\end{equation*}
\end{definition}

\begin{definition}
An $N$-dimensional random vector $\varsigma$ follows an \textbf{elliptical distribution} with location parameter $\mu$, positive definite scale matrix $\Sigma$ and characteristic generator, \ZHnew{if and only if} we have the following representation
\begin{equation*}
\varsigma \overset{\textbf{d}}{=}\mu + A Z,
\end{equation*}
where $Z$ follows a spherical distribution with a characteristic generator $\Phi$, $A \in \mathbb{R}^{N \times N}$ such that $A A^{T} = \Sigma$, and $\mu \in \mathbb{R}^{N}$; $\overset{\textbf{d}}{=}$ means that both sides have the same distribution. We denote it by $\varsigma \sim EC_{N}(\mu, \Sigma, \Phi)$.
\end{definition} 
The family of elliptical distributions \ZHnew{is comprised} of many classical distribution functions, \ZHnew{e.g.,} Gaussian distributions, Student distributions, and Laplace distributions. To describe the dependence between rows of random vectors, we introduce some basic concepts about copulas. For a deeper discussion of copulas, we refer the reader to the monograph \cite{nelsen2006introduction}.
\begin{definition}
A $K$-dimensional \textbf{copula} is a distribution function $C : [0, 1]^{K} \rightarrow [0, 1]$ of some random vector whose marginals are uniformly distributed on $[0, 1]$.
\end{definition}
An important property of copulas is that each copula is dominated by the minimum component of the variable, where the minimum component is itself a copula \ZHnew{called} maximum copula, described \ZHnew{by} the following proposition.
\begin{proposition}(Frechet-Hoeffding upper bound)\cite[Page 11]{nelsen2006introduction}\label{Frechet-Hoeffding upper bound} For each copula $C$ and $u = (u_{1}, ..., u_{K})$, \ZHnew{we have}
\begin{equation*}
C(u) \leq \min_{1 \leq i \leq K} u_{i}.
\end{equation*}
\end{proposition}
\begin{proposition}(Sklar's theorem)\cite[Theorem 2.3.3]{nelsen2006introduction}\label{Sklar's Theorem.}
For each distribution function $F : \mathbb{R}^{K} \rightarrow [0, 1]$ with marginals $\lbrace F_{i} \rbrace_{1 \leq i \leq K}$, there exists a $K$-dimensional copula $C$ satisfying
\begin{equation*}
F(x) = C( F_{1}(x_{1}), ... , F_{K}(x_{K})  ), \ \ \forall x \in \mathbb{R}^{K}.
\end{equation*}
Moreover, if $F_{i}$, $i = 1, ... K$ is continuous, then the copula $C$ is uniquely given by
\begin{equation*}
C(u) = F ( F^{-1}_{1}(u_{1}), ... , F^{-1}_{K}(u_{K}) ), \ \ F^{-1}_{i}(t) := \underset{ F_{i}(r) \geq r }{\inf} r.
\end{equation*}
\end{proposition}
The Sklar's theorem builds a bridge between distribution functions and copulas, meaning we can define a distribution function of a given copula or a copula by providing a distribution function. In the following, we present a special family of copulas.
\begin{definition}
A copula $C$ is called \textbf{Archimedean} \ZHnew{if and only if} there exists a continuous strictly decreasing function $\Phi : [0, 1] \rightarrow [0, +\infty)$ such that $\Phi(1) = 0$ and
\begin{equation*}
C(u) = \Phi^{-1}\left( \sum\limits^{K}_{i=1} \Phi \left( u_{i} \right) \right),
\end{equation*}
$\Phi$ is called the generator of the Archimedean copula $C$. If $\lim_{u \rightarrow 0} \Phi(u) = +\infty$, then $C$ is called a strict Archimedean copula.
\end{definition}
The Archimedean copula is an important \ZHnew{copula} \ZHnew{thanks to its} components \ZHnew{which} can be separated. \ZHnew{We} list some examples of Archimedean copulas with \ZHnew{their} corresponding generators in Appendix \ref{list of Archimedean copulas}.

\subsection{Reformulation of chance constraints}\label{Reformulation of chance constraints}
In order to get the convexity of $S(p)$, we \ZHnew{reformulate} \eqref{Fea-1} as follows. \ZHnew{Let} $v_{i} \thicksim EC_{N}({\mu}_{i}, {\Sigma}_{i}, {\Phi}_{i}) $ and $x \in X$. If $x \neq 0$, then we can set
\begin{equation}\label{Fea-2}
\xi_{i}(x) := \frac{v_{i}^{T}x - \mu_{i}^{T}x}{\sqrt{x^{T} \Sigma_{i} x}}, \ \ g_{i}(x):=\frac{D_{i} - \mu_{i}^{T}x}{\sqrt{x^{T}\Sigma_{i} x}}.
\end{equation}
Thus, the chance constraint in \eqref{CCP} can be rewritten as follows
\begin{equation*}
\mathbb{P}\left(\xi_{i} \leq g_{i}(x), \ \ i = 1, \cdots, K \right) \geq p.
\end{equation*}
It can be shown that $\xi_{i}(x)$ follows 1-dimensional spherical distribution with characteristic generator $\Phi_{i}$ \cite[Section 2.1]{fang1990symmetric}. Further, by using Proposition \ref{Sklar's Theorem.} \ZHnew{of} Sklar's Theorem, we can obtain a $K-$dimensional copula $C_{x}$ of $\xi(x) = [\xi_{1}(x), ... , \xi_{K}(x)]^{T}$, \ZHnew{i.e.,}
\begin{equation}\label{Fea-3}
C_{x}[F_{1}(g_{1}(x)), ... , F_{K}(g_{K}(x))] \geq p,
\end{equation}
where $F_{i}$ is the cumulative distribution function of $\xi_{i}(x), i = 1, ... , K$. Assume that $C_{x}$ is a strictly Archimedean copula with generator $\psi_{x}$. Then, the definition of Archimedean copulas  implies a reformulation of \eqref{Fea-3} as follows
\begin{equation}\label{Fea-4}
\psi^{(-1)}_{x}\left( \sum\limits^{K}_{i=1} \psi_{x}(F_{i}(g_{i}(x))) \right) \geq p,
\end{equation}
Since $\psi_{x}$ is non-negative and decreasing, we can further reformulate \eqref{Fea-4} \ZHnew{as the following system} \cite{cheng2014second}
\begin{equation}\label{Fea-6}
\left\{
        \begin{array}{ll}
        &\psi_{x}(F_{i}(g_{i}(x))) \leq y_{i}\psi_{x}(p),\\
        &\sum\limits^{K}_{i=1} y_{i} = 1, \ y_{i} \geq 0, \  i= 1, ..., K.
        \end{array}
\right.
\end{equation}
From \eqref{Fea-6} we can see that $x^{*} \in S(p)$ if and only if there exists $y^{*} = [y^{*}_{1}, ..., y^{*}_{K}]^{T} \in \mathbb{R}^{K}$ and $x^{*} \in X$ such that $(x^{*}, y^{*})$ is a feasible solution of \eqref{Fea-6}. To satisfy the constraints of \eqref{Fea-6}, we will restrict ourselves here to $K$ constants $y_{i} \in (0, 1)$ for the sake of simplicity. Applying the decreasing monotonicity of the generator $\psi_{x}$ again, we can rewrite \eqref{Fea-6} as follows \ZH{\cite{nguyen2023convexity}}
\begin{equation}\label{Fea-7}
\left\{
        \begin{array}{ll}
        &F_{i}(g_{i}(x)) \geq  \psi^{(-1)}_{x} (y_{i}\psi_{x}(p)),\\
        &\sum\limits^{K}_{i=1} y_{i} = 1, \ y_{i} \geq 0, \  i= 1, ..., K.
        \end{array}
\right.
\end{equation}
In the remainder of this paper, we assume $C_{x}$ to be a Gumbel-Hougaard copula defined according to the following assumption.
\begin{assumption}
$C_{x}$ is a Gumbel-Hougaard copula on $X$ if the generator $\psi_{x}$ is given by
\begin{equation*}
\psi_{x}(t) = \left( -\ln t \right)^{\frac{1}{\kappa(x)}}, \ \forall (x, t) \in X \times (0, 1],
\end{equation*}
where $\kappa(x): X \rightarrow (0, 1]$ is a strictly positive function. 
\end{assumption}

Define $H_{i}(x) := F_{i}(g_{i})(x)$ and $U(x, y_{i}) := \psi^{(-1)}_{x} (y_{i} \psi_{x}(p))$. Using \eqref{Fea-7}, we can see that the convexity of \eqref{Fea-1} can be obtained once we prove the concavity of $H_{i}$ and \ZH{the} convexity of $U$. Thus, in the next two sections, we will provide the proofs for the concavity of $H_{i}$ and convexity of $U$, respectively.

\subsection{Concavity of $H_{i}(x)$}\label{Concavity of H_{i}(x)}
In this section, we prove the convexity of $H_{i}(x) := F_{i}(g_{i})(x)$. We assume $F_{i}$ is a $r$-revealed-concavity function and show the $r$-concavity \cite{shapiro2009lectures} of $g_{i}$ with some real number $r$. We \ZHnew{start} with a necessary and sufficient condition for local $r$-concavity of $g_{i}$ defined by \eqref{Fea-2}, which is presented in Appendix \ref{proof of Lemma about Hi-0}. \ZHnew{Note} that the convexity of $\text{sign}(-r) \cdot g_{i}$ is just the $r$-concavity of $g_{i}$ in the \ZHnew{remaining} of this paper. 
\begin{lemma}\label{Lemma about Hi-1}
Let $r \in \mathbb{R}$, $b \in \mathbb{R}$, $\mu \in \mathbb{R}^{N}$. Assume $\Sigma$ is a positive definite covariance matrix. Suppose $f$ is a function defined on the set $E := \lbrace x \in \mathbb{R}^{N} \setminus \lbrace o \rbrace : b - \mu^{T}x > 0 \rbrace$ as follows
\begin{equation*}
f(x) := {\left(\frac{b - \mu^{T}x}{\sqrt{x^{T}\Sigma x}}\right)}^{r}.
\end{equation*}
If $r = 0$, then $f$ is locally concave on $E$; If $r \neq 0$, then $\text{sign(-r)} \cdot f$ is locally convex on $E$ if and only if :
\begin{equation}\label{sufficient and necessary condition r<2 and r neq 0}
\mu^{T}\Sigma^{-1}\mu \leq (2-r)\frac{(\mu^{T}x)^{2}}{x^{T}\Sigma x} - 2r\sqrt{\theta} \frac{\mu^{T}x}{\sqrt{x^{T}\Sigma x}} - (r + 1)\theta,
\end{equation}
where $\theta := \frac{\left( b - \mu^{T}x \right)^{2}}{x^{T} \Sigma x}$.
The function $\ln ((b - \mu^{T}x)/ \sqrt{x^{T} \Sigma x})$ is locally concave on $E$ if and only if:
\begin{equation}\label{sufficient and necessary condition ln}
\mu^{T} \Sigma^{-1} \mu \leq 2 \cdot \frac{\left( \mu^{T} x \right)^{2} }{x^{T} \Sigma x} - \theta.
\end{equation}
\end{lemma}

\begin{remark}
We can see that Lemma 3 in \cite{cheng2014second} is a \ZHnew{special} case of Lemma \ref{Lemma about Hi-1} with $\mu = 0$ and $r = -1$ since each locally convex function is convex on convex sets. Thus, the result of Lemma \ref{Lemma about Hi-1} implies  $\text{sign}(-r) \cdot f$ is convex on any convex subset of $E$.
\end{remark}

It can be noticed that the form of these necessary and sufficient conditions in Lemma \ref{Lemma about Hi-1} \ZH{are non-linear}. Thus, the threshold $p^{*}$ to the convexity of $\text{sign}(-r)\cdot f$ can not be induced by these conditions easily. In the following, we will provide a geometric approach to find the \ZH{best lower bound} of $\theta$ such that $\text{sign}(-r)\cdot f$ is convex. First, we provide a new reformulation for \eqref{sufficient and necessary condition r<2 and r neq 0} from a geometric perspective.

\begin{lemma}\label{reformualtion of the necessary and sufficient conditions again}
Assume the assumptions of Lemma \ref{Lemma about Hi-1} hold. Then, we can reformulate \eqref{sufficient and necessary condition r<2 and r neq 0} as follows
\begin{equation}\label{2 transformation of key sufficient and necessary conndition inequality}
\mu^{T} \Sigma^{-1} \mu \leq \frac{\| \mu \|^{2} \cdot \text{cos}^{2} \beta}{\lambda(x)} + \frac{2b \cdot \| \mu \| \cdot \text{cos}\beta}{\lambda(x)} \cdot \frac{1}{\| x \|} + \frac{\left( -1 - r \right)b^{2}}{\lambda(x)} \cdot \frac{1}{\| x \|^{2}}.
\end{equation}
Assume $\mu \neq 0$. Define $g(x) := (b - \mu^{T}x)/\sqrt{x^{T}\Sigma x}$. Define $\lambda_{\mu, \text{min}}$ as follows
\begin{equation}\label{the definition of lamda_mu_min}
\frac{1}{\sqrt{\lambda_{\mu, \text{min}}}} = \max_{x \in \mathbb{S}^{N-1}} \frac{\mu \cdot x}{\| \mu \|} \cdot \frac{1}{\sqrt{x^{T} \Sigma x}},
\end{equation}
where $\mathbb{S}^{N-1}$ is the $N-1$-dimensional unit sphere in $\mathbb{R}^{N}$. We denote $x_{\text{min}} \in \mathbb{S}^{N-1}$ the eigenvector of $\Sigma$ corresponding to $\lambda_{\text{min}}$. Then, $\lambda_{\text{min}} \leq \lambda_{\mu, \text{min}} \leq \lambda_{\text{max}}$. In particular, when $\mu/\| \mu \| = x_{\text{min}}$, we can see $\lambda_{\mu, \text{min}} = \lambda_{\text{min}}$. Let $h$ be the right hand side of \eqref{2 transformation of key sufficient and necessary conndition inequality}. Then, the $g$ and $h$ can be reformulated as follows
\begin{equation}\label{h function's transformation with overline t and overline s}
h(\overline{t}, \overline{s}) = \| \mu \|^{2} \cdot \overline{s}^{2} + 2b \cdot \| \mu \| \cdot \overline{t} \cdot \overline{s} + (-1 - r) \cdot b^{2} \cdot \overline{t}^{2},
\end{equation}
\begin{equation}\label{g function's transformation with overline t and overline s}
g(\overline{t}, \overline{s}) = b \cdot \overline{t} - \| \mu \| \cdot \overline{s},
\end{equation}
where \ZH{$\overline{t}:=1/\sqrt{x^{T}\Sigma x}$, $\overline{s}:= cos(\beta) \cdot \| x \| / \sqrt{x^{T} \Sigma x}$}, $\overline{t} \in (0, +\infty)$ and $\overline{s} \in [-1/ \sqrt{\lambda_{\mu, \text{min}}}, 1 / \sqrt{\lambda_{\mu, \text{min}}}]$. 
\end{lemma}
The main idea of the proof for Lemma \ref{reformualtion of the necessary and sufficient conditions again} is to transform the decision vector $x$ into a polar coordinate form such that the interior link among $b$, $\mu$, $x$, and $\Sigma$ can be presented. The complete proof of Lemma \ref{reformualtion of the necessary and sufficient conditions again} is provided in Appendix \ref{proof of Lemma about Hi-0}.

Observe that for any $c \in \mathbb{R}$, we can reformulate $g(\overline{t}, \overline{s}) = c$ as follows
\begin{equation*}
\overline{s}_{1}(\overline{t}; c) : = \overline{s} = \frac{b}{\| \mu \|} \cdot \overline{t} - \frac{c}{\| \mu \|}.
\end{equation*}
An implicit function $\overline{s}_{2}(\overline{t})$ can also be defined by $h(\overline{t}, \overline{s}) = \mu^{T} \Sigma^{-1} \mu$, where $h$ is defined by \eqref{h function's transformation with overline t and overline s}. In the following lemma, we discuss the relative position between $\overline{s}_{1}$ and $\overline{s}_{2}$.
\begin{lemma}\label{tangent relationship between s1 and s2}
Let $b>0$, $\| \mu \| > 0$, $r<-1$, $c \in \mathbb{R}$. Assume $\Sigma$ is a positive definite covariance matrix. Define $\overline{s}_{1}$, induced by \eqref{g function's transformation with overline t and overline s}, as follows
\begin{equation*}
\overline{s}_{1}(\overline{t}; c) : = \overline{s} = \frac{b}{\| \mu \|} \cdot \overline{t} - \frac{c}{\| \mu \|}.
\end{equation*}
Then, at least one of the following cases can be satisfied:\\
\textbf{Case (1)} {There \ZH{exists} $(\overline{t}^{*}, \overline{s}^{*}) \in (0, +\infty) \times [-1/\sqrt{\lambda_{\mu, \text{min}}}, 1 / \sqrt{\lambda_{\mu, \text{min}}}]$ and $c^{*} \geq 0$ such that $\overline{s}_{1}(\overline{t}; c^{*})$ is tangent with $\overline{s}_{2}(\overline{t})$ at $\overline{t}^{*}$} and $\overline{s}_{1}(\overline{t}^{*}; c^{*}) =  \overline{s}^{*} = \overline{s}_{2}(\overline{t}^{*})$;\\ 
\textbf{Case (2)} {In $(0, +\infty) \times [-1/\sqrt{\lambda_{\mu, \text{min}}}, 1 / \sqrt{\lambda_{\mu, \text{min}}}]$, there exist unique $\overline{t}^{*} \geq 0$ and $c^{*} \geq 0$ such that $\overline{s}_{1}(\overline{t}; c^{*})$ and $\overline{s}_{2}(\overline{t})$ are intersected at $(\overline{t}^{*}, -1/ \sqrt{\lambda_{\mu, \text{min}}})$};\\
\textbf{Case (3)} {There is no interaction point of $\overline{s}_{1}(\overline{t}; c^{*})$ and $\overline{s}_{2}(\overline{t})$ in $(0, +\infty) \times [-1/\sqrt{\lambda_{\mu, \text{min}}}, 1 / \sqrt{\lambda_{\mu, \text{min}}}]$. Take $c^{*} = \|\mu \| / \sqrt{\lambda_{\mu, \text{min}}}$. }\\
Further, only the Case (1) and Case (2) can occur. Moreover, $c^{*}$ can be defined as follows:
\begin{equation}\label{case 2 and case 3 equation 9 & case 2 and case 3 equation 4}
c^{*} = 
\left\{
        \begin{array}{ll}
        \sqrt{\frac{-r + 2}{-r - 2}} \cdot \sqrt{\mu^{T} \Sigma^{-1} \mu}, & \text{if} \ \ r \leq r^{*}(\mu, \Sigma),\ r^{o}(\mu, \Sigma) > 1,\\      
        \ &\ \\        
        \frac{\left( -r \right) \frac{\| \mu \|}{\sqrt{\lambda_{\mu, \text{min}}}} + \sqrt{\left( 2 + r \right) \frac{\| \mu \|^{2}}{\lambda_{\mu, \text{min}}} + \left( -1 - r \right) \mu^{T} \Sigma^{-1} \mu}}{-1 - r}, & \text{otherwise},
        \end{array}
\right.
\end{equation}
where $r^{o}(\mu, \Sigma)$ and $r^{*}(\mu, \Sigma)$ are defined by
\begin{equation}\label{case 2 and case 3 equation 11}
r^{o}(\mu, \Sigma) := \frac{\| \mu \|^{2}}{\mu^{T} \Sigma^{-1} \mu}\cdot \frac{1}{\lambda_{\mu, \text{min}}}, \ \ r^{*}(\mu, \Sigma): = -2 \cdot \sqrt{\frac{1}{r^{o}(\mu, \Sigma) - 1 } + 1},
\end{equation}
In particular, given $a>0$, the case with $\Sigma = a \cdot I$ belongs to the Case (2).
\end{lemma}
\ZH{Lemma} \ref{tangent relationship between s1 and s2} is the key result in this section, which provides a new geometric approach to \ZHnew{find} best thresholds $\theta^{*}$. The proof of Lemma \ref{tangent relationship between s1 and s2} is provided in Appendix \ref{proof of Lemma about Hi-0}. We are now in a position to show how to find the \ZH{best lower bound} of $\theta$ such that $\text{sign}(-r)\cdot f$ is convex. \ZH{Define $g(x) := (b - \mu^{T}x)/\sqrt{x^{T}\Sigma x}$ .We want to consider whether there exists $\theta^{*}$ such that for any $\theta^{**} \geq \theta^{*}$ the function $\text{sign(-r)}\cdot g^{r}$ is locally convex on 
\begin{equation}\label{"locally convex" set}
G(\theta^{**}) : = \left\lbrace x \in E \cap X :  b - \mu^{T}x  \geq \sqrt{\theta^{**}} \cdot \sqrt{x^{T} \Sigma x}
 \right\rbrace.
\end{equation}}
We can see that the following theorem \ZHnew{generalizes} \ZH{Lemma} 2 of \cite{cheng2014second} from the case with $r < -1$ into the case for any $r \in \mathbb{R}$.

\begin{theorem}\label{Corollary of Hi-1}
Assume $o \notin X$. Under the assumptions of Lemma \ref{Lemma about Hi-1}, define $g(x) := (b - \mu^{T}x)/\sqrt{x^{T}\Sigma x}$. 
    \begin{enumerate}
        \item If $\| \mu \| = 0$, then there exists $\theta^{*}$ \ZHnew{if and only if} $r < -1$. And $\theta^{*} = 0$.
        \item If $\| \mu \| \neq 0$, $b = 0$, then there exists an {\textbf{best}} threshold $\theta^{*}=\mu^{T} \Sigma^{-1} \mu$ for any $r \neq 0$.
        \item If $\| \mu \| \neq 0$, $b < 0$ and $r \leq -2$, then there exists
        \begin{equation*}
\theta^{*} = \frac{1}{\left( -1 - r \right)} \cdot \left[ \mu^{T} \Sigma^{-1} \mu - \left( 2 + r \right) \cdot \frac{\| \mu \|^{2}}{\lambda_{\text{min}}} \right].
\end{equation*}
        \item If $\| \mu \| \neq 0$, $b < 0$ and $-2 < r < -1$, then there exists $\theta^{*} = \mu^{T} \Sigma^{-1} \mu / (-1 - r)$. 
        \item If $\| \mu \| \neq 0$, $b < 0$ and $r \geq -1$ and $r \neq 0$, then there exists
        \begin{equation*}
\theta^{*} = \mu^{T} \Sigma^{-1} \mu + (2 + r) \cdot \frac{\| \mu \|^{2}}{\lambda_{\text{min}}}.
\end{equation*}
        \item If $\| \mu \| \neq 0$, $b > 0$, $r \geq -1$ and $r \neq 0$, then there is no $\theta^{*}$.
        \item If $\| \mu \| \neq 0$, $b > 0$, and $r < -1$, then there exists \ZHnew{an} unique {\textbf{best}} $\theta^{*}$
        \scriptsize{
        \begin{equation*}
\sqrt{\theta^{*}} = 
\left\{
        \begin{array}{ll}
        \sqrt{\frac{-r + 2}{-r - 2}} \cdot \sqrt{\mu^{T} \Sigma^{-1} \mu}, & \text{if} \ \ r \leq r^{*}(\mu, \Sigma),\ r^{o}(\mu, \Sigma) > 1,\\      
        \ &\ \\        
        \frac{\left( -r \right) \frac{\| \mu \|}{\sqrt{\lambda_{\mu, \text{min}}}} + \sqrt{\left( 2 + r \right) \frac{\| \mu \|^{2}}{\lambda_{\mu, \text{min}}} + \left( -1 - r \right) \mu^{T} \Sigma^{-1} \mu}}{-1 - r}, & \text{otherwise},
        \end{array}
\right.
\end{equation*}}
        \item \normalsize{If there exists a constant $\lambda_{0} > 0$ such that $\Sigma = \lambda_{0} \cdot I$ and $\| \mu \| \neq 0$, then the best threshold is defined by $\sqrt{\theta^{*}} = \lambda^{-\frac{1}{2}}_{0} \cdot \| \mu \| \cdot \left( -r + 1 \right) / \left( -r - 1 \right)$.
}
        \item \normalsize{If $b = 0$, then $\sqrt{\theta^{*}} = \sqrt{\mu^{T} \Sigma^{-1} \mu}$ is {\textbf{best}} \ZH{such that $\ln g$ is locally concave on \eqref{"locally convex" set} for any $\theta^{**} \geq \theta^{*}$}.}
        \item If $b < 0$, then there exists $\theta^{*} = \mu^{T} \Sigma^{-1} \mu + 2\| \mu \|^{2} / \lambda_{\text{min}}$ \ZH{such that $\ln g$ is locally concave on \eqref{"locally convex" set} for any $\theta^{**} \geq \theta^{*}$}.
        \item If $b > 0$, then there is no \ZH{$\theta^{*}, \theta^{**}$ such that $\ln g$ is locally concave on \eqref{"locally convex" set}}.
    \end{enumerate}
\end{theorem}

\begin{proof}
Based on \eqref{sufficient and necessary condition r<2 and r neq 0} and \eqref{sufficient and necessary condition ln}, it can be seen that when $\| \mu \| = 0$, we have $\text{sign}(-r) \cdot g^{r}$ is convex on $E$ if and only if $r < -1$, which implies $\theta^{*} = 0$. In the \ZHnew{remaining} of this theorem's proof, we assume $\| \mu \| \neq 0$. Using Lemma \ref{reformualtion of the necessary and sufficient conditions again}, we can see that \eqref{sufficient and necessary condition r<2 and r neq 0} can be reformulated as \eqref{2 transformation of key sufficient and necessary conndition inequality}. Due to limited space, the $b \leq 0$ \ZHnew{cases are} considered by Lemma \ref{proof of part of Th. 3.4} in Appendix \ref{proof of Lemma about Hi-0}. Next, we only provide the proof under $b > 0$. Since $b>0$ and $\| \mu \| > 0$, we know that $g(\overline{t}, \overline{s})$ is a bilinear function, increasing with respect to $\overline{t}$, and decreasing with respect to $\overline{s}$. Consequently, if we take some $c_{0} \in \mathbb{R}$, then $g(\overline{t}, \overline{s}) = c_{0}$  is a straight line in $\overline{t}o\overline{s}$-plane, and based on \eqref{g function's transformation with overline t and overline s} \ZHnew{the set}
\begin{equation*}
\left\lbrace x \in \mathbb{R}^N: \frac{b - \mu^{T}x }{\sqrt{x^{T} \Sigma x}} \geq c_{0} \right\rbrace
\end{equation*}
is equivalent to
\begin{equation}\label{tangent_G set}
G(c_{0}) := \left\lbrace (\overline{t}, \overline{s}) \in (0, +\infty) \times [-1/\sqrt{\lambda_{\mu, \text{min}}}, 1 / \sqrt{\lambda_{\mu, \text{min}}}]: g(\overline{t}, \overline{s}) \geq c_{0} \right\rbrace,
\end{equation}
which corresponds to the right-hand side area of the line $g(\overline{t}, \overline{s}) = c_{0}$ between two horizontal lines $\overline{s} = 1/ \sqrt{\lambda_{\mu, \text{min}}}$ \ZH{and} $\overline{s} = -1 / \sqrt{\lambda_{\mu, \text{min}}}$. In particular, the set $E$ is a subset of $G(0)$ and $G(c_{0}) \subset G(0)$ for all $c_{0} \geq 0$. 
We set
\begin{equation}\label{tangent_Q set}
Q := \left\lbrace (\overline{t}, \overline{s}) \in (0, +\infty) \times [-1/\sqrt{\lambda_{\mu, \text{min}}}, 1 / \sqrt{\lambda_{\mu, \text{min}}}]: h(\overline{t}, \overline{s}) \geq \mu^{T} \Sigma^{-1} \mu \right\rbrace.
\end{equation}
Thus, we have
\begin{equation}\label{key description of theta* in the form of set}
\sqrt{\theta^{*}} = \inf \left\lbrace c_{0} \geq 0: G(c_{0}) \subset Q \right\rbrace.
\end{equation}
Suppose $r = -1$. Using \eqref{h function's transformation with overline t and overline s}, then we \ZH{have}
\begin{equation*}
h(\overline{t}, \overline{s}) = \| \mu \|^{2} \cdot \overline{s}^{2} + 2b \cdot \| \mu \| \cdot \overline{t} \cdot \overline{s}.
\end{equation*}
Therefore, if we take $\overline{s}_{0} < 0$, then we obtain
\begin{equation*}
h(\overline{t}, \overline{s}_{0}) < 0 < \mu^{T} \Sigma^{-1} \mu \ \ \text{as} \ \ \overline{t} \rightarrow +\infty,
\end{equation*}
which implies that there is no $c_{0} \geq 0$ such that $G(c_{0}) \subset Q$. 
Suppose $r > -1$. \ZHnew{Then,} $h(\overline{t}, \overline{s})$ defined by \eqref{h function's transformation with overline t and overline s} is a hyperbolic paraboloid. Thus, there is no $c_{0}$ satisfying $G(c_{0}) \subset Q$.
As a result, we proved that if $b>0$, $r \in [-1, +\infty) \setminus \lbrace 0 \rbrace$, then there does not exist $\theta^{*} \in \mathbb{R}$ such that $\text{sign(-r)}\cdot g^{r}$ is locally convex on \eqref{"locally convex" set}. Suppose $r < -1$. Using \eqref{h function's transformation with overline t and overline s}, then we can see that $h(\overline{t}, \overline{s})$ is an elliptic paraboloid. 
It follows that there \ZH{exists} $\sqrt{\theta^{*}}$ defined by \eqref{key description of theta* in the form of set}. 

Applying Lemma \ref{tangent relationship between s1 and s2}, we know that at least one of the \ZHnew{assertions 1, 2 and 3} can be satisfied. 
According to the definitions of $\overline{s}_{1}$, $g$ and $c^{*}$, the best threshold is $\sqrt{\theta^{*}} = c^{*}$. Thus, when $b > 0$, the $\sqrt{\theta^{*}}$ \ZHnew{can} be defined by Lemma \ref{tangent relationship between s1 and s2} as follows:
\begin{equation*}
\sqrt{\theta^{*}} = 
\left\{
        \begin{array}{ll}
        \sqrt{\frac{-r + 2}{-r - 2}} \cdot \sqrt{\mu^{T} \Sigma^{-1} \mu}, & \text{if} \ \ r \leq r^{*}(\mu, \Sigma),\ r^{o}(\mu, \Sigma) > 1,\\      
        \ &\ \\        
        \frac{\left( -r \right) \frac{\| \mu \|}{\sqrt{\lambda_{\mu, \text{min}}}} + \sqrt{\left( 2 + r \right) \frac{\| \mu \|^{2}}{\lambda_{\mu, \text{min}}} + \left( -1 - r \right) \mu^{T} \Sigma^{-1} \mu}}{-1 - r}, & \text{otherwise},
        \end{array}
\right.
\end{equation*}
where $r^{*}(\mu, \Sigma)$ and $r^{o}(\mu, \Sigma)$ are defined by \eqref{case 2 and case 3 equation 11}. 

Finally, it remains to consider $\ln \left( \left( b - \mu^{T}x \right) / \sqrt{x^{T} \Sigma x} \right)$. \ZHnew{Notice} that \eqref{sufficient and necessary condition ln} is a specific case of \eqref{sufficient and necessary condition r<2 and r neq 0} with $r = 0$. Thus, it suffices to repeat the proof process of the case with $r \geq -1$ and $r \neq 0$ literally.
\end{proof}

\begin{remark}
(1). According to \ZH{Lemma} 2 of \cite{cheng2014second} and Lemma 1 of \cite{nguyen2023convexity}, the threshold with $r < -1$ \ZHnew{is given by}:
\begin{equation*}
\sqrt{\theta^{*}_{0}} := \| \mu \| \cdot \lambda^{-\frac{1}{2}}_{\text{min}} \cdot \frac{-r + 1}{-r - 1}.
\end{equation*}
Consider the case with $b < 0$ and $-2 < r < -1$. Observe that $r < -1$ implies $1/\sqrt{-1 - r} < (-r + 1)/(-r - 1)$. Thus, according to Theorem \ref{Corollary of Hi-1}, our $\sqrt{\theta^{*}} := \| \mu \| \cdot \lambda^{-\frac{1}{2}}_{\text{min}} / \sqrt{-1 - r}$ is less than $\sqrt{\theta^{*}_{0}}$. Consider the case with $b < 0$ and $r \leq -2$. Since $r < -1$ implies $(-r + 1)/(-r - 1) > 1$, we can see that our $\theta^{*}$ is also less than $\theta^{*}_{0}$. Suppose $b > 0$ and $r < -1$. Notice that $\lambda_{\mu, \text{min}} \geq \lambda_{\text{min}}$. Then, we can see \ZH{that} $\sqrt{\theta^{*}} = c^{*}$ defined by the second line of \eqref{case 2 and case 3 equation 9 & case 2 and case 3 equation 4} is less than $\sqrt{\theta^{*}_{0}}$. Since $r < r^{*}(\mu, \Sigma) \leq -2$, we \ZH{have}
\begin{equation*}
\sqrt{\frac{\left( -\frac{r}{2} \right)^{2}}{-\frac{r}{2} + 1}} \geq \sqrt{-\frac{r}{2} + 1} \cdot \frac{\left( -r - 1 \right)}{\left( -r + 1 \right)},
\end{equation*}
that is
\begin{equation}\label{case 2 and case 3 equation 13}
\sqrt{1 + \frac{1}{\frac{1}{4}r^{2} - 1}} \geq \left( \sqrt{ \frac{-r + 2}{-r - 2}} \right) \cdot \left( \frac{-r + 1}{-r - 1} \right)^{-1}.
\end{equation}
Observe that when the Case (1) of \ZHnew{Lemma \ref{tangent relationship between s1 and s2}} holds, then $\overline{s}^{*} \geq -1/ \sqrt{\lambda_{\mu, \text{min}}}$ is equivalent to
\begin{equation}\label{case 2 and case 3 equation 10}
\sqrt{1 + \frac{1}{\frac{1}{4}r^{2} - 1}} \cdot \frac{\sqrt{\mu^{T} \Sigma^{-1} \mu}}{\| \mu \|} \leq \frac{1}{\sqrt{\lambda_{\mu, \text{min}}}}.
\end{equation}
Thus, combining \eqref{case 2 and case 3 equation 10} and \eqref{case 2 and case 3 equation 13} we can see \ZH{that} $\sqrt{\theta^{*}} = c^{*}$ defined by the first line of \eqref{case 2 and case 3 equation 9 & case 2 and case 3 equation 4} is also less than $\sqrt{\theta^{*}_{0}}$.

(2). Let us explain the meaning of \ZH{\textbf{"best"}} in Theorem \ref{Corollary of Hi-1}. Suppose $F$ is a probability function of chance constraints. In this paper, one of the key ideas for the proof follows the \ZH{classical } method established in \cite{henrion2008convexity}: Suppose $g$ is a $-\alpha$-concave function and $F$ is the distribution function of $\xi$ with an $(\alpha + 1)$-decreasing density, then the convexity of $F\circ g$ implies the eventual convexity of the chance constraint $S(p)$. Thus, proving the convexity of $F\circ g$ is the main task in \cite{henrion2008convexity}. Regarding the eventual convexity problem, most previous results followed the same idea in \cite{henrion2008convexity} to obtain a better probability threshold to prove certain compound function $F\circ g$ is convex. In \ZH{this paper,  \textbf{"best"}} means that if \ZHnew{we} want to use the above composition of a $-\alpha$-concave function and an $(\alpha + 1)$-decreasing density to obtain the eventual convexity, the threshold $\sqrt{\theta^{**}}$ has to be greater than the best threshold $\sqrt{\theta^{*}}$, otherwise the above method is invalid. \ZH{More precisely}, if a threshold $\sqrt{\theta^{**}}$ is less \ZH{than the} \textbf{best} threshold $\theta^{*}$, then the $-\alpha$-concavity of $g$ can not be guaranteed  on $S(p)$. It is worth remarking that even though $F\circ g$ is not a concave function, we might also be able to \ZHnew{show that} $S(p)$ is convex. As a result, the convexity of $S(p)$ can not imply the concavity of function $F\circ g$. This might explain why the probability threshold $p^{*}$ in most papers is close to 1 and even greater than 0.9, but the convexity of the \ZH{feasible set} can be observed starting from the probability values close to 0.5. 

(3). In this paper, we obtain a better value of $\theta^{*}$ compared to \cite{cheng2014second, minoux2016convexity, minoux2017global, nguyen2023convexity} for the following reasons: Firstly, we considered the relationship among $b$, $\mu$, and the feasible point $x$ contained in $b > \mu^{T}x$. Secondly, we transform the variable $x$ into \ZH{a} polar coordinate form $(\| x \|, \text{cos} \beta)$, which \ZHnew{uses} a geometric method to obtain \ZHnew{the} best threshold $\theta^{*}$. \ZHnew{Due} to the space limitation, we only provided the best thresholds $\sqrt{\theta^{*}}$ for the cases with $b \geq 0$. We can also use the same method to obtain \ZH{the best value of} $\sqrt{\theta^{*}}$ with $b < 0$. In the following corollary, we see that the geometric method can even be used to \ZHnew{find} the best probability threshold $p^{*}$ under some special distributions, such as the Gaussian distributions in \ZH{\cite{minoux2016convexity}}. We provide the complete proof in Appendix \ref{proof of Lemma about Hi-0}.
\end{remark}

\begin{corollary}\label{Corollary of Hi-1-1}
Under the definitions of $g$ and $G(\theta^{**})$ in Theorem \ref{Corollary of Hi-1}, assume $F$ is a Gaussian distribution function and $o \notin X$. If $b \geq 0$, then there \ZH{exists} an best threshold $\theta^{*}\geq 0$ such that $F(g)(x)$ is locally concave on $G(\theta^{**})$ if and only if $\theta^{**} \geq \theta^{*}$. If $b < 0$, then there is no threshold $\theta^{**}$. Further, if $p > F(\sqrt{\theta^{*}})$, then the \ZH{feasible set} $S(p)$ is a convex set and $F(\sqrt{\theta^{*}})$ is an \ZH{best lower bound} of the probability $p$. In particular, if $b = 0$, then the best probability threshold is $p^{*}:= F(\sqrt{\mu^{T} \Sigma^{-1} \mu})$.
\end{corollary}

It is worth remarking that the set $\lbrace x: \theta(x) \geq \theta^{**} \rbrace$ is convex for every $\theta^{**} > 0$, where $\theta(x):= (b - \mu^{T} x)/ \sqrt{x^{T} \Sigma x}$, $\Sigma$ is a positive definite matrix, which implies $G(\theta^{**})$ defined by \eqref{"locally convex" set} is a convex set. Thus, the results of \ZH{Theorem} \ref{Corollary of Hi-1} means that, with suitable $r$, the function $\text{sign}(-r) \cdot g^{r}$ is convex on the set $G(\theta^{**})$ since a locally convex function is \ZHnew{global} convex on a convex set.

In order to apply the above results, we add some assumptions. First, we provide two definitions about $r$-decreasing and $r$-revealed-concavity.
\begin{definition}\label{def of r-decreasing}
A function $f : \mathbb{R} \rightarrow \mathbb{R}$ is \textbf{$r$-decreasing} for some $r \in \mathbb{R}$ \ZHnew{if and only if} $f$ is continuous on $(0, +\infty)$ and there exists some $t^{*}(r) > 0$ such that the mapping $t \mapsto t^{r} f(t)$ is strictly decreasing for all $t > t^{*}(r)$.
\end{definition}
When $r = 0$, the $r$-decreasing is simply \ZHnew{a} decreasing \ZHnew{function}. If $f$ is a non-negative function, then $r$-decreasing implies $r'$-decreasing for any $r' \leq r$. This \textit{decreasing order} property will be significant for our \ZHnew{results}. The aim of illustrating $r$-decreasing is to \ZHnew{find} a concave mapping defined by $t \mapsto f(t^{\frac{1}{r}})$. Consequently, a much \textit{weaker} family of functions than $r$-decreasing functions is proposed as follows.

\begin{definition}\label{def of r-revealed-concavity}
A function $f: \mathbb{R} \to \mathbb{R} \cup \lbrace +\infty \rbrace$ is \textbf{$r$-revealed-concavity} \ZHnew{if and only if} there exists $t^{**}(r) > 0$ such that $f$ satisfies \ZHnew{anyone of the following}:
\begin{itemize}
\item[(1)] $r < 0$ and $t \mapsto f(t^{\frac{1}{r}})$ is concave on $(0, t^{**}(r)]$;
\item[(2)] $r = 0$ and $t \mapsto f(e^{t})$ is concave on $[t^{**}(r), +\infty)$;
\item[(3)] $r > 0$ and $t \mapsto f(t^{\frac{1}{r}})$ is concave on $[t^{**}(r), +\infty)$.
\end{itemize}
\end{definition}
Many \ZH{classical } distributions are $r$-revealed-concavity, \ZHnew{e.g.,} Gaussian and Student distributions. Let $r \in \mathbb{R}\setminus \lbrace 0 \rbrace$. Suppose $F$ is a distribution with $(-r+1)$-decreasing density. Then, in the following of this paper, we will see that $F$ is $r$-revealed-concavity and $t^{**}(r) = (t^{*}(-r + 1))^{r}$. Similar to $r$-decreasing functions, the family of $r$-revealed-concavity functions follows an \textit{increasing order}, \ZHnew{i.e.,} a $r$-revealed-concavity function implies $r'$-revealed-concavity for any $r' \geq r$ \cite[Lemma 3.5]{van2015eventual}.

\begin{assumption}\label{Assumption 2}

(1) The cumulative distribution function $F_{i}$ is $r_{i}$-revealed-concavity with the thresholds $t^{**}_{i}(r_{i})$, $i = 1, ... , K$, where $r_{i} \neq 0$. For $r_{i} = 0$, \ZH{there exists} $0 < \varepsilon_{0} < 1$ such that $F_{i}$ is $(-\varepsilon_{0})$-revealed-concavity with the thresholds $t^{**}_{i}(-\varepsilon_{0})$.

(2) $p > p^{*}$, where
\begin{equation*}
p^{*} = \max \left\lbrace \frac{1}{2}, \underset{i\in \lbrace 1, ..., K \rbrace}{\max} \ F_{i}\left( \sqrt{\theta^{*}_{i}} \right), \underset{i\in \lbrace 1, ..., K \rbrace}{\max} \ F_{i}[t^{**}_{i}(r_{i})^{1/r_{i}}] \right\rbrace,
\end{equation*}
where $\theta^{*}_{i}$ is defined in Theorem \ref{Corollary of Hi-1} when $r_{i} \neq 0$, and if $r_{i} = 0$, $\theta^{*}_{i}$ is redefined equal to the value corresponding to the case with $r_{i} = -\varepsilon_{0}$ in Theorem \ref{Corollary of Hi-1} and $F_{i}[t^{**}_{i}(r_{i})^{1/r_{i}}]$ is replaced by $F_{i}[t^{**}_{i}(-\varepsilon_{0})^{-1/\varepsilon_{0}}]$.
\end{assumption}

The following Lemma shows that $(-r+1)$-decreasing of a density implies $r$-revealed-concavity of the distribution function for each $r \in \mathbb{R} \setminus \lbrace 0 \rbrace$. It is an extension of \cite[Lemma 3.1]{henrion2008convexity} from $r > 0$ into $r \in \mathbb{R} \setminus \lbrace 0 \rbrace$. We provide the complete proof in Appendix \ref{proof of Lemma about Hi-0}.

\begin{lemma}\label{Lemma about Hi-2}
Assume $F$ is a distribution function with $(-r + 1)$-decreasing density $f$ for some $r \in \mathbb{R}$. If $r < 0$, then the mapping $z \mapsto F(z^{\frac{1}{r}})$ is concave on $(0, (t^{*}(-r + 1))^{r})$, where $t^{*}$ is defined in Definition \ref{def of r-decreasing}; If $r > 0$, then $z \mapsto F(z^{\frac{1}{r}})$ is concave on $((t^{*}(-r + 1))^{r}, +\infty)$. In particular, we can see $t^{**}(r) = (t^{*}(-r + 1))^{r}$, where $t^{**}$ is defined in Definition \ref{def of r-revealed-concavity}.
\end{lemma}

\begin{remark}
Due to \ZH{Lemma} \ref{Lemma about Hi-2}, we can see that \ZH{assertion (1)} of Assumption \ref{Assumption 2} can be strengthened \ZHnew{by}
\begin{itemize}
\item[(1)*] The cumulative distribution function $F_{i}$ has $(-r_{i} + 1)$-decreasing densities with the thresholds $t^{*}_{i}(-r_{i} + 1)$, $i = 1, ... , K$, where $r_{i} \neq 0$. For $r_{i} = 0$, there \ZH{exists} $0 < \varepsilon_{0} < 1$ such that $F_{i}$ has $(1 + \varepsilon_{0})$-decreasing densities with the thresholds $t^{*}_{i}(1 + \varepsilon_{0})$.
\end{itemize}
Using the relationship between $t^{*}$ and $t^{**}$, we can rewrite \ZH{Assumption \ref{Assumption 2} (2)} as follows
\begin{itemize}
\item[(2)*] $p > p^{*}$, where
\begin{equation*}
p^{*} = \max \left\lbrace \frac{1}{2}, \underset{i\in \lbrace 1, ..., K \rbrace}{\max} \ F_{i}\left( \sqrt{\theta^{*}_{i}} \right), \underset{i\in \lbrace 1, ..., K \rbrace}{\max} \ F_{i}[t^{*}_{i}(-r_{i} + 1)] \right\rbrace,
\end{equation*}
where $\theta^{*}_{i}$ is defined in Theorem \ref{Corollary of Hi-1} when $r_{i} \neq 0$. \ZHnew{If} $r_{i} = 0$, we \ZHnew{set} $\theta^{*}_{i}$ to the value corresponding to the case with $r_{i} = \varepsilon_{0}$ in Theorem \ref{Corollary of Hi-1}, and $F_{i}[t^{*}_{i}(-r_{i} + 1)]$ is replaced by $F_{i}[t^{*}_{i}(1 + \varepsilon_{0})]$.
\end{itemize}
In the remainder, for the sake of simplicity, \ZH{we will use $t^{*}$ instead of $t^{**}$}.
\end{remark}

We can now show the concavity of $H_{i}$ on the convex hull of $S(p)$.
\begin{lemma}\label{Proposition about concave Hi}
Assume $o \notin X$. If \ZH{Assumption} \ref{Assumption 2} \ZH{holds}, then $H_{i}$ is concave on $\text{conv}S(p)$.
\end{lemma}

\begin{proof}
Define $g_{i}$ as follows:
\begin{equation*}
g_{i}(x) := \frac{D_{i} - \mu^{T}_{i}x}{\sqrt{x^{T}\Sigma_{i}x}}.
\end{equation*}
Combining \ZH{assertion $(2)$ of} Assumption \ref{Assumption 2} and \ZH{Theorem} \ref{Corollary of Hi-1}, we can see that $\text{sign}(-r_{i})\cdot g_{i}^{r_{i}}$ is convex on $\text{conv}S(p) \subset \cap^{K}_{i=1}G(\theta^{*}_{i})$. Thus, for any $x_{1}, x_{2} \in S(p)$, and $\lambda \in [0, 1]$, we get
\begin{equation*}
g_{i}(\lambda x_{1} + (1 - \lambda)x_{2}) \geq \left( \lambda g^{r_{i}}_{i}(x_{1}) + (1 - \lambda)g^{r_{i}}_{i}(x_{2}) \right)^{\frac{1}{r_{i}}}, \ \ \forall r_{i} \in \mathbb{R} \setminus \lbrace 0 \rbrace,
\end{equation*} 
and then we have
\begin{equation}\label{Proposition about concave Hi first part 1}
F_{i}\left( g_{i}(\lambda x_{1} + (1 - \lambda)x_{2}) \right) \geq F_{i}\left( \left( \lambda g^{r_{i}}_{i}(x_{1}) + (1 - \lambda)g^{r_{i}}_{i}(x_{2}) \right)^{\frac{1}{r_{i}}} \right).
\end{equation}

Next, we prove 
\begin{equation}\label{Proposition about concave Hi first part 2}
F_{i}\left( \left( \lambda g^{r_{i}}_{i}(x_{1}) + (1 - \lambda)g^{r_{i}}_{i}(x_{2}) \right)^{\frac{1}{r_{i}}} \right) \geq \lambda F_{i}\left( g_{i}\left( x_{1} \right) \right) + (1 - \lambda)F_{i}\left( g_{i}(x_{2}) \right).
\end{equation}
Due to \ZH{Lemma} \ref{Lemma about Hi-2}, it suffices to prove that if $r_{i} > 0$, then $g_{i}(x_{j})^{r_{i}} \in ((t^{*}_{i}\left( -r_{i} + 1 \right))^{r_{i}}, +\infty)$ for $j = 1, 2$. If $r_{i} < 0$, then $g_{i}(x_{j})^{r_{i}} \in (0, (t^{*}_{i}\left( -r_{i} + 1 \right))^{r_{i}})$ for $j = 1, 2$. Since $x_{1}, x_{2} \in S(p)$ and $p > p^{*}$, together with \eqref{Fea-3}, Proposition \ref{Frechet-Hoeffding upper bound} and \ZH{assertion $(2)$} of Assumption \ref{Assumption 2}, we \ZHnew{have}
\begin{equation*}
F_{i}\left( g_{i}\left( x_{j} \right) \right) > p^{*} \geq F_{i}\left( t^{*}_{i}\left( -r_{i} + 1 \right) \right), j = 1, 2,
\end{equation*}
thus we have
\begin{equation}\label{Proposition about concave Hi first part 3}
g_{i}(x_{j}) > t^{*}_{i}(-r_{i} + 1) > 0, j = 1, 2.
\end{equation}
Then, the proof is complete by taking both sides of \eqref{Proposition about concave Hi first part 3} to the power of $r_{i}$. Finally, \ZH{combining \eqref{Proposition about concave Hi first part 1} and \eqref{Proposition about concave Hi first part 2} leads to}
\begin{equation*}
F_{i}\left( g_{i}(\lambda x_{1} + (1 - \lambda)x_{2}) \right) \geq \lambda F_{i}\left( g_{i}\left( x_{1} \right) \right) + (1 - \lambda)F_{i}\left( g_{i}(x_{2}) \right),
\end{equation*}
which implies $F_{i}(g_{i}(x))$ is concave on $S(p)$.

It remains to consider the case with $r_{i} = 0$. \ZH{Combining \ZH{Theorem} \ref{Corollary of Hi-1},  Assumption \ref{Assumption 2} and \ZH{Lemma} \ref{Lemma about Hi-1}, we proved that given the values of $\theta^{*}$ presented in this paper, $\ln g_{i}$ is locally concave on the set $S(p)$, i.e.,}
\begin{equation*}
\ln g_{i} \left( \lambda x_{1} + \left( 1 - \lambda \right) x_{2} \right) \geq \lambda \ln g_{i} \left( x_{1} \right) + \left( 1 - \lambda \right) \ln g_{i} \left( x_{2} \right),
\end{equation*}
which can also be rewritten as follows
\begin{equation*}
g_{i} \left( \lambda x_{1} + \left( 1 - \lambda \right) x_{2} \right) \geq g_{i} \left( x_{1} \right)^{\lambda} \cdot g_{i} \left( x_{2} \right)^{\left( 1 - \lambda \right)}.
\end{equation*}
\ZHnew{Due} to the monotonicity of distribution $F_{i}$, we have
\begin{equation}\label{Proposition about concave Hi first part 4}
F_{i} \left( g_{i} \left( \lambda x_{1} + \left( 1 - \lambda \right) x_{2} \right) \right) \geq F_{i} \left( g_{i} \left( x_{1} \right)^{\lambda}  \cdot g_{i} \left( x_{2} \right)^{\left( 1 - \lambda \right)}\right).
\end{equation}
Notice \ZHnew{that} $r$-decreasing functions follow a "decreasing" order. This "decreasing" property leads to the conclusion that a $(1 + \varepsilon_{0})$-decreasing function $F_{i}$ is also a $r$-decreasing function for all $r \leq 1 + \varepsilon_{0}$. \ZHnew{For} $D_{i} < 0$ and $r > -1$, \ZHnew{by} Theorem \ref{Corollary of Hi-1}, we take \ZHnew{either}
\begin{equation}\label{theta* of ln g}
\theta^{*} = \mu^{T}_{i} \Sigma^{-1}_{i} \mu_{i} + (2 + r) \cdot \frac{\| \mu_{i} \|^{2}}{\lambda_{\text{min, i}}}
\end{equation}
or the greater one
\begin{equation*}
\theta^{*} = (3 + r) \cdot \| \mu_{i} \|^{2} \cdot \lambda^{-1}_{\text{min, i}},
\end{equation*}
which are both increasing w. r. t. $r$. Let $r \in [-\varepsilon_{0}, \varepsilon_{0}] \setminus \lbrace 0 \rbrace$. Since $\varepsilon_{0} \in (0, 1)$, we know that $2 + r \in (1, 3)$, \ZH{which implies $\theta^{*}$ defined by \eqref{theta* of ln g} is non-negative}. As a result, the function $\text{sign}(-r) \cdot g_{i}^{r}$ is locally convex on the following set
\begin{equation*}
\left\lbrace x \in E: g_{i}(x) \geq \sqrt{\mu^{T}_{i} \Sigma^{-1}_{i} \mu_{i} + (2 + \varepsilon_{0}) \cdot \frac{\| \mu_{i} \|^{2}}{\lambda_{\text{min, i}}} } \right\rbrace.
\end{equation*}
Using the local convexity property of $\text{sign}(-r) \cdot g_{i}^{r}$, together with \ZH{assertion (2)} of Assumption \ref{Assumption 2}, we \ZHnew{get}
\begin{equation}\label{Proposition about concave Hi first part 5}
F_{i}\left( \left( \lambda g^{r}_{i}(x_{1}) + (1 - \lambda)g^{r}_{i}(x_{2}) \right)^{\frac{1}{r}} \right) \geq \lambda F_{i}\left( g_{i}\left( x_{1} \right) \right) + (1 - \lambda)F_{i}\left( g_{i}(x_{2}) \right).
\end{equation}
Take the limitation of \eqref{Proposition about concave Hi first part 5} as $r$ tends to $0$, then we have
\begin{equation}\label{Proposition about concave Hi first part 6}
\begin{aligned}
&F_{i}\left( g_{i} \left( x_{1} \right)^{\lambda} \cdot g_{i} \left( x_{2} \right)^{\left( 1 - \lambda \right)} \right)\\
=& \lim_{r \to 0} F_{i} \left( \left( \lambda g^{r}_{i}\left( x_{1} \right) + \left( 1 - \lambda \right) g^{r}_{i}\left( x_{2} \right) \right)^{\frac{1}{r}} \right)\\
\geq & \lambda F_{i} \left( g_{i}\left( x_{1} \right) \right) + \left( 1 - \lambda \right) F_{i} \left( g_{i}\left( x_{2} \right) \right).
\end{aligned}
\end{equation}
When $D_{i} = 0$, due to Theorem \ref{Corollary of Hi-1}, we take $\theta^{*}_{i} = \mu^{T}_{i} \Sigma^{-1}_{i} \mu_{i}$ which is a constant independent on $r$. Thus, we can get the same result as \eqref{Proposition about concave Hi first part 6}. Combining \eqref{Proposition about concave Hi first part 4} and \eqref{Proposition about concave Hi first part 6} we \ZHnew{obtain}
\begin{equation*}
F_{i}\left( g_{i}(\lambda x_{1} + (1 - \lambda)x_{2}) \right) \geq \lambda F_{i}\left( g_{i}\left( x_{1} \right) \right) + (1 - \lambda)F_{i}\left( g_{i}(x_{2}) \right),
\end{equation*}
which completes the proof.
\end{proof}
\begin{remark}
Let $\theta^{*}_{max}$ be the maximum of $\lbrace \theta^{*}_{i} \rbrace^{K}_{i=1}$. It is worth remarking that if $\theta^{*}_{max}$ is the best threshold obtained from Theorem \ref{Corollary of Hi-1} and satisfies
\begin{equation}\label{Proposition about concave Hi first part 7}
\sqrt{ \theta^{*}_{max} } > \underset{i\in \lbrace 1, ..., K \rbrace}{\max} \left\lbrace F^{-1}_{i} \left( \frac{1}{2} \right),  \ t^{*}_{i}(-r_{i} + 1) \right\rbrace,
\end{equation}
then $H_{i}$ is concave on $S(p)$ for any $p \geq F_{i}\left(  \sqrt{\theta^{*}_{max}}\right)$. Moreover, if the distribution functions $\lbrace F_{i} \rbrace^{K}_{i=1}$ are all symmetric, then \eqref{Proposition about concave Hi first part 7} can be simplified as
\begin{equation*}
\sqrt{ \theta^{*}_{max} } > \underset{i\in \lbrace 1, ..., K \rbrace}{\max} \left\lbrace t^{*}_{i}(-r_{i} + 1) \right\rbrace.
\end{equation*}
\ZHnew{Notice} that the $\theta^{*}$ obtained \ZHnew{by} Theorem \ref{Corollary of Hi-1} is just the best value for the convexity of $\text{sign}(-r) \cdot g^{r}$ on $G(\theta^{*})$ defined by \eqref{"locally convex" set}. 
\end{remark}

\subsection{Convexity of $U(x, y_{i})$}\label{Convexity of Ui(x, yi)}
In this section, \ZH{we prove} the mapping $(x, y_{i}) \mapsto U(x, y_{i})$ is joint convex on $X \times (0, 1)$. At first, we define $\varphi_{1}$, \ZH{$\varphi_{2}$}, and $\omega$ as follows
\begin{equation}
\left\{
        \begin{array}{ll}
        \varphi_{1} =& \kappa(x)\cdot\ln y_{i} \cdot [ \kappa(x) - 1 + \kappa(x) \cdot  \ln p \cdot y^{\kappa(x)}_{i} ],\\
        \varphi_{2} =& \kappa(x) \cdot \left( \ln y_{i} \right)^{2} \cdot \left( 1 + \ln p \cdot y^{\kappa(x)}_{i} \right) \cdot [ 1 - \kappa(x) - \kappa(x) \cdot  \ln p \cdot y^{\kappa(x)}_{i} ]\\
        & + \left( 1 + \kappa(x)\cdot \ln y_{i} + \ln p \cdot \ln y_{i} \cdot y^{\kappa(x)}_{i} \cdot \kappa(x) \right)^{2},\\
        \omega =&\varphi_{2}/\varphi_{1}.
        \end{array}
\right.
\end{equation}
Let $H_{x}\kappa$ denote the Hessian matrix of $\kappa$, and $\nabla_{x}\kappa$ denote the gradient vector of $\kappa$. In order to prove $(x, y_{i}) \mapsto U(x, y_{i})$ is convex on $X \times (0, 1)$, we make the following assumptions. 
\begin{assumption}\label{Assumption 3}
(1) $p \geq e^{-1}$; (2) $0< \kappa(x) \leq 1$, $\forall x \in X$.
\end{assumption}

In \cite{nguyen2023convexity}, \ZH{$\omega$ is} defined \ZHnew{as} a constant and there exists $c > 0$ such that $\kappa \geq c$. In this paper, \ZH{$\omega$ is} extended to a function \ZHnew{of} the decision vector $x$, and the lower \ZHnew{bound} of $\kappa$ is extended to $\kappa > 0$. These two extensions about $\omega$ and $\kappa$ lead to $\omega$ approaching \ZHnew{infinity} as $\kappa$ \ZHnew{is} close to 0, which is the main difficulty in proving the convexity of $U$. The following lemma is the key result of this section. We show that even if the function $\omega$ is unbounded on $X$, a constant multiplier of the function $1/\kappa$ can still be the upper bound of $\omega$.
\begin{lemma}\label{Lemma end 2}
\ZH{Assume $o \notin X$ and \ZH{Assumption} \ref{Assumption 3}  holds}, then $\varphi_{1}, \varphi_{2} > 0$, and there exists $d>0$ such that $1/\omega(x, y_{i}) \geq d \cdot \kappa(x)$ for any $(x, y_{i}) \in X \times (0, 1)$.
\end{lemma}
The proof of Lemma \ref{Lemma end 2} is presented in Appendix \ref{proof of Lemma about Hi-0}. The formulation of $H_{x} U(x, y_{i})$ is so complex that it is hard to verify its positive semi-definiteness. The following assumption provides an auxiliary matrix to justify the positive semi-definiteness of $H_{x} U(x, y_{i})$.
\ZH{
\begin{assumption}\label{Assumption 4}
$d \cdot \kappa(x) H_{x}\kappa(x) - \nabla_{x} \kappa(x) (\nabla_{x}\kappa(x))^{T}$ is a positive semi-definite matrix for any $x \in X$.
\end{assumption}
}

\begin{proposition}\label{Property of Schur complement}
There are some properties of the Schur complement of $C$ in $G$ :

(1) $G \succ 0$ if and only if $C \succ 0$ and $A - B C^{-1}B^{T} \succ 0$;

(2) If $C \succ 0$, then $G \succeq 0$ if and only if $A - B C^{-1}B^{T} \succeq 0$.
\end{proposition}

To prove the convexity of $U$, it suffices to \ZH{show that the} Hessian matrix of $U$ is positive semi-definite. We define
\begin{equation*}
\nabla_{x} := \left( \frac{\partial}{\partial x_{1}}, ... , \frac{\partial}{\partial x_{N}} \right)^{T},
\end{equation*}
\begin{equation*}
H_{x}U(x, y_{i}) :=
\begin{pmatrix}
\frac{\partial^{2}U}{\partial x^{2}_{1}} & \frac{\partial^{2}U}{\partial x_{1}\partial x_{2}} & \cdots & \frac{\partial^{2}U}{\partial x_{1}\partial x_{N}}\\
\frac{\partial^{2}U}{\partial x_{2}\partial x_{1}} & \frac{\partial^{2}U}{\partial x^{2}_{2}} & \cdots & \frac{\partial^{2}U}{\partial x_{2}\partial x_{N}}\\
\vdots & \vdots & \ddots & \vdots\\
\frac{\partial^{2}U}{\partial x_{N}\partial x_{1}} & \frac{\partial^{2}U}{\partial x_{N}\partial x_{2}} & \cdots & \frac{\partial^{2}U}{\partial x^{2}_{N}}
\end{pmatrix},
\end{equation*}
then the Hessian matrix of $U$ with respect to $(x, y_{i})$ can be written as follows
\begin{equation*}
H_{(x, y_{i})} U(x, y_{i}) = 
\begin{pmatrix}
H_{x}U(x, y_{i}) & \nabla_{x} \frac{\partial}{\partial y_{i}} U(x, y_{i})\\
\left( \nabla_{x} \frac{\partial}{\partial y_{i}} U(x, y_{i}) \right)^{T} & \frac{\partial^{2}}{\partial y^{2}_{i}} U(x,  y_{1})
\end{pmatrix}.
\end{equation*}

Using the auxiliary positive semi-definite matrix proposed in Assumption \ref{Assumption 4} and the Schur complement, we will provide an equivalent form of the Hessian matrix of $U$. The following proposition is the result in \cite{nguyen2023convexity}. For the sake of completeness and self-containment, \ZH{we provide} the proof of the following proposition in Appendix \ref{proof of Lemma about Hi-0}.

\begin{proposition}\cite{nguyen2023convexity}\label{The detail of the second part of N-0--1}
Let $(x, y_{i}) \in X \times (0, 1)$. Define a matrix $M(x, y_{i})$ on $X \times (0, 1)$ as follows
\begin{equation*}
M(x, y_{i}) := \varphi_{1}(x, y_{i}) \cdot H_{x}\kappa(x) - \varphi_{2}(x, y_{i}) \cdot \left( \nabla_{x} \kappa(x) \right) \cdot \left( \nabla_{x} \kappa(x) \right)^{T}.
\end{equation*}
Then, the matrix $H_{(x, y_{i})} U$ is positive semi-definite if and only if $M(x, y_{i})$ is positive semi-definite.
\end{proposition}

We can now state the main result of this section.
\begin{lemma}\label{Lemma end 1}
Assume $o \notin X$. If Assumption \ref{Assumption 3} and \ref{Assumption 4} hold, then $U(x, y_{i})$ is jointly convex for any $(x, y_{i}) \in X \times (0, 1)$. In particular, if $d \cdot \kappa(x) H_{x}\kappa(x) - \nabla_{x} \kappa(x) (\nabla_{x}\kappa(x))^{T}$ is a positive definite matrix, then $U(x, y_{i})$ is strictly convex.
\end{lemma}
\begin{proof}
\ZH{Using Proposition \ref{The detail of the second part of N-0--1}, it suffices to prove} $M(x, y_{i})$ is a positive semi-definite matrix. By using Lemma \ref{Lemma end 2} we have $\omega := \varphi_{2}/\varphi_{1} > 0$. Observe that $\left( \nabla_{x} \kappa(x) \right) \cdot \left( \nabla_{x} \kappa(x) \right)^{T}$ is a positive semi-definite matrix and $\kappa(x)$ satisfies \ZH{Assumption} \ref{Assumption 4}.  It follows that $d \cdot \kappa(x) H_{x}\kappa(x)$ is also a positive semi-definite matrix. Because $d > 0$ and $0 < \kappa(x) \leq 1$, we can get the conclusion that $H_{x}\kappa(x)$ is a positive semi-definite matrix. Then, applying \ZH{Lemma} \ref{Lemma end 2} again, we have
\begin{equation}\label{The detail of the second part of N---1}
\begin{array}{ll}
M(x, y_{i}) &= \varphi_{2} \cdot \left[ \frac{1}{\omega(x, y_{i})} H_{x}\kappa(x) - \nabla_{x} \kappa(x) (\nabla_{x}\kappa(x))^{T}  \right]\\ 
& \ \\
&\succeq \varphi_{2} \cdot \left[ d \cdot \kappa(x) H_{x}\kappa(x) - \nabla_{x} \kappa(x) (\nabla_{x}\kappa(x))^{T} \right] \succeq 0,
\end{array}
\end{equation}

Suppose $d \cdot \kappa(x) H_{x}\kappa(x) - \nabla_{x} \kappa(x) (\nabla_{x}\kappa(x))^{T}$ is positive definite. By using the \ZHnew{assertion} (1) of Proposition \ref{Property of Schur complement} and \eqref{The detail of the second part of N---1}, together with the fact that $\varphi_{2} > 0$,  we \ZHnew{have} $M(x, y_{i}) \succ 0$. As a result, the Hessian matrix of $U$ is a positive definite matrix, which implies $U$ is strictly convex on $X \times (0, 1)$.
\end{proof}

\subsection{Convexity of $S(p)$ under elliptical distributions}\label{Eventual convexity with elliptical distributions}
In this section, we provide the eventual convexity of chance constraints with elliptical distributions and copulas. In the previous sections, we assume the origin point $o \notin X$. In fact, we will prove that the results in Section \ref{Concavity of H_{i}(x)} and \ref{Convexity of Ui(x, yi)} are also valid for the domain $X$ containing $o$. To this end, we recall some results about "star-shaped".
\begin{definition}
A set $S \subset \mathbb{R}^{K}$ is called star-shaped with respect to o (or briefly \textbf{'star-shaped'}) \ZHnew{if and only if} $S \neq \emptyset$ and $[o, x] \subset S$ for all $x \in S$ (thus, a star-shaped set, in our terminology, contains the origin).
\end{definition}

Similar to Proposition 5.1 in \cite{henrion2008convexity}, we can obtain the following result. 
\begin{lemma}\label{Lemma end 0}
If $o \in S(p)$ and $p \in (0, 1]$, then the \ZH{feasible set} $S(p)$ is star-shaped w. r. t. the origin.
\end{lemma}
\ZH{The complete proof is given in Appendix \ref{proof of Lemma about Hi-0}.} We can now present the main result of this paper. Combining the concavity of $H_{i}$ and the joint convexity of $U$, we can get the convexity of the \ZH{feasible set} $S(p)$.
\begin{theorem}\label{Theorem of convex feasible set}
Assume the \ZH{feasible set} $S(p) \neq \emptyset$. If \ZH{Assumptions} \ref{Assumption 2}, \ref{Assumption 3}, and \ref{Assumption 4} \ZH{hold}, then $S(p)$ is convex. Moreover, if $d \cdot \kappa(x) H_{x}\kappa(x) - \nabla_{x} \kappa(x) (\nabla_{x}\kappa(x))^{T}$ is a positive definite matrix, then $S(p)$ is strictly convex.
\end{theorem}
\begin{proof}
We show that, for any $x_{1}, x_{2} \in S(p)$, $\lambda \in [0, 1]$, it follows that $x_{\lambda} := \lambda x_{1} + (1 - \lambda)x_{2} \in S(p)$. Suppose $o \notin X$. \ZHnew{Let $y^{1} := (y^{1}_{1}, ..., y^{1}_{K})$ and $y^{2} := (y^{2}_{1}, ..., y^{2}_{K})$ such that $y^{j}_{i} \neq 0$, $j = 1, 2$ and $i = 1, ..., K$,} which satisfy \eqref{Fea-7} and  correspond to $x_{1}$ and $x_{2}$ respectively. Applying \ZH{Lemma} \ref{Proposition about concave Hi} and Lemma \ref{Lemma end 1}, \ZH{we have}
\begin{equation*}
\begin{array}{ll}
F_{i}\left( g_{i}(x_{\lambda}) \right) &\geq \lambda F_{i}\left( g_{i}(x_{1}) \right) + (1-\lambda)F_{i}\left( g_{i}(x_{2}) \right)\\
&=\lambda H_{i}(x_{1}) + (1 - \lambda) H_{i}(x_{2})\\
&\geq \lambda U(x_{1}, y^{1}_{i}) + (1 - \lambda) U(x_{2}, y^{2}_{i})\\
&\geq U(x_{\lambda}, \lambda y^{1}_{i} + (1 - \lambda) y^{2}_{i}),
\end{array}
\end{equation*} 
which implies $(x_{\lambda}, \lambda y^{1} + (1 - \lambda) y^{2})$ satisfies \eqref{Fea-7}. It follows that $x_{\lambda} \in S(p)$.

Suppose $o \in X$. Then, we consider the following cases.

Case 1: If $ x_{\lambda} = o = x_{1} = x_{2}$, then $x_{\lambda} \in S(p)$.

Case 2: If $\lbrace x_{\lambda} \neq o, x_{1} = o, x_{2} \neq o \rbrace$ or $\lbrace x_{\lambda} \neq o, x_{1} \neq o, x_{2} = o \rbrace$, then by applying \ZH{Lemma} \ref{Lemma end 0} we get $x_{\lambda} \in S(p)$.

Case 3: If $x_{\lambda} \neq o, x_{1} \neq o, x_{2} \neq o$, then we can use the result of the case of $o \notin X$ to obtain $x_{\lambda} \in S(p)$. 

Precisely, we can divide Case 3 into two cases: (1) $o \in X \cap S(p)$; (2) $o \in X \setminus S(p)$.  Suppose $o \in X \cap S(p)$. Since $o \in S(p)$ if and only if $\min_{i} D_{i} \geq 0$. Then, by using \ZH{Theorem} \ref{Corollary of Hi-1}, we can see the locally $r$-concavity of $g_{i}$ can be satisfied only in the case with $r < -1$. In addition, we can see \ZHnew{that} the function $\text{sign(-r)}\cdot g^{r}_{i}$ is well-defined on $x = 0$ with $r < -1$ and $D_{i} \geq 0$. As a result, we can follow the proof of the $o \notin X$ case to \ZHnew{get} $x_{\lambda} \in S(p)$. Suppose $o \in X \setminus S(p)$. Using \ZHnew{the} proof by contradiction, we can see that $x_{1} \neq o$ and $x_{2} \neq o$ means $x_{\lambda} \neq o$. Thus, there is a convex subset $\hat{X} \subset X \setminus \left\lbrace o \right\rbrace$ such that $S(p) \subseteq \hat{X}$ due to the continuity of the probability function w.r.t. decision variable $x$. Then, we can use the result of the case of $o \notin \hat{X}$ to obtain $x_{\lambda} \in S(p)$. Thus, the proof of the theorem is complete.
\end{proof}
\begin{remark}
    In the proof of Theorem \ref{Theorem of convex feasible set}, we used the proof by contradiction to prove that $x_{1} \neq o \neq x_{2}$ implies $x_{\lambda} \neq o$. A similar technique is used in \cite[Page 275]{henrion2008convexity}.
\end{remark}

\section{Extension of eventual convexity under skewed distributions}\label{Eventual convexity with skewed distributions}
In this section, we extend the results about eventual convexity of the chance constraints from the cases with elliptical symmetric distributions into the cases with skewed distributions. First, we will extend the elliptical results in the previous sections of this paper into the case under a generalized hyperbolic (GH) distribution. The modified Bessel functions will play critical roles in this part. Next, we will consider a case with more general skewed distributions than GH distributions, i.e. \ZHnew{Normal Mean-Variance Mixture (NMVM for short)} distributions, where the linear mean mixture in GH distributions is substituted by a nonlinear mixture form. Finally, we discuss a GH case with stronger influence of skewness than our first skewed eventual convexity result. Based on a radial decomposition method, we will see this stronger skewed problem can be solved by means of the former two skewed eventual convexity results.

\subsection{Fundamental about skewed distributions}

Consider an elliptical distribution with a spherical distribution taken by a normal distribution. By adding randomness into the covariance matrix, we can generalize this kind of elliptical distribution into normal mixture distributions.
\begin{definition}
The random vector $\varsigma$ \ZHnew{follows} (multivariate) \textbf{normal variance mixture} distribution if
\begin{equation*}
\varsigma \overset{d}{=} \mu + \sqrt{W} A Z, 
\end{equation*}
where

(1) $Z \sim \mathcal{N}_{N}(0, I_{N})$;

(2) $W \geq 0$ is a non-negative, scalar-valued random variable that is independent of $Z$;

(3) $A \in \mathbb{R}^{N \times N}$ and $\mu \in \mathbb{R}^{N}$ are \ZH{deterministic matrix and vector}, respectively;

(4) $\overset{\textbf{d}}{=}$ means that both sides have the same distribution.

We denote it briefly by $\varsigma \sim M_{N}(\mu, \Sigma, \hat{H})$, where $\Sigma = A A^{T}$ and $\hat{H}(\theta) = \mathbb{E}(e^{-\theta W})$ is the \textbf{Laplace-Stieltjes transform} of the distribution function $H$ of $W$.
\end{definition}

In this paper, we always assume that the covariance matrix $\Sigma$ is positive definite and the distribution of $W$ has no point mass at zero. The normal variance mixture distribution can be viewed as a composite function consisting of multiple normal distributions with the same mean vector and the same covariance matrix up to a multiplicative constant. \ZHnew{The NMVM distribution is derived by adding randomness to \ZHnew{its} mean vector, which introduces asymmetry into the random vectors.}
\begin{definition}\label{definition of normal mean variance mixture distribution}
The random vector $\varsigma$ \ZHnew{follows} (multivariate) \textbf{NMVM} distribution if
\begin{equation}\label{normal mean variance mixture distribution}
\varsigma \overset{d}{=} m(W) + \sqrt{W} A Z,
\end{equation}
where

(1) $Z \sim \mathcal{N}_{N}(0, I_{N})$;

(2) $W \geq 0$ is a non-negative, scaler-valued random vector which is independent of $Z$; 

(3) $A \in \mathbb{R}^{N \times N}$ is a matrix;

(4) $m : [0, +\infty) \to \mathbb{R}^{N}$ is a measurable function.
\end{definition}
Next, we introduce an important subclass of normal mean-variance mixture distributions.

\begin{definition}\label{definition of Generalized hyperbolic distributions}
The random vector $\varsigma$ is said to have a (multivariate) \textbf{generalized hyperbolic (GH)} distribution if $\varsigma$ is defined by \eqref{normal mean variance mixture distribution} with $W$ following a generalized inverse Gaussian distribution $W \sim N^{-}(\lambda, \chi, \psi)$, and $m(W)$ defined as follows
\begin{equation*}
m(W) = \mu + W \gamma,
\end{equation*}
where $\mu$ and $\gamma$ are vectors in $\mathbb{R}^{N}$, and \ZHnew{the} other parameters satisfy
\begin{equation*}
\left\{
        \begin{array}{ll}
        \chi > 0, \psi \geq 0, & \text{if} \ \ \lambda < 0,\\         
        \chi > 0, \psi > 0, & \text{if} \ \ \lambda = 0,\\
        \chi \geq 0, \psi > 0, & \text{if} \ \ \lambda > 0.
        \end{array}
\right.
\end{equation*}
\end{definition}
We present the definition of generalized inverse Gaussian (GIG) distributions in Appendix \ref{definition of Generalized inverse Gaussian distributions}. We will use $\varsigma \sim GH_{N}(\lambda, \chi, \psi, \mu, \Sigma, \gamma)$ to denote $\varsigma$ following a GH distribution.
Suppose $Z \sim GH_{N}(\lambda, \chi, \psi, \mu, \Sigma, \gamma)$. Define $h(\omega)$ as the density of the random variable $W$. Then, we can formulate the densities of $Z$ as follows:
\begin{equation}\label{density of GH_d random vector with W density form}
f_{Z}(t) = \int^{+\infty}_{0} \frac{e^{ \left( t - \mu \right)^{T} \Sigma^{-1} \gamma }}{\left( 2\pi \right)^{N/2} \left| \Sigma \right|^{1/2} \omega^{N/2} } exp \left\lbrace - \frac{\left( t - \mu \right)^{T} \Sigma^{-1} \left( t - \mu \right)}{2 \omega} - \frac{\gamma^{T} \Sigma^{-1} \gamma}{2/\omega} \right\rbrace h(\omega)d \omega.
\end{equation}
Further, the evaluation of \eqref{density of GH_d random vector with W density form} provides the generalized hyperbolic density defined as follows
\begin{equation}\label{non-integral form _ density of GH_d random vector with W density form}
f_{Z}(t) = c \frac{K_{\lambda - (N/2)} \left( \sqrt{\left( \chi + \left( t - \mu \right)^{T} \Sigma^{-1} \left( t- \mu \right) \right) \left( \psi + \gamma^{T} \Sigma^{-1} \gamma \right) } \right) e^{ \left( t - \mu \right)^{T} \Sigma^{-1} \gamma  }}{\left( \sqrt{\left( \chi + \left( t - \mu \right)^{T}\Sigma^{-1} \left( t - \mu \right) \right) \left( \psi + \gamma^{T} \Sigma^{-1} \gamma \right) } \right)^{\left( N/2 \right) - \lambda}},
\end{equation}
where the constant $c$ is defined by
\begin{equation*}
c = \frac{\left( \sqrt{\chi \psi} \right)^{-\lambda} \psi^{\lambda} \left( \psi + \gamma^{T} \Sigma \gamma \right)^{(N/2) - \lambda} }{\left( 2\pi \right)^{N/2} \left|\Sigma  \right|^{1/2} K_{\lambda}\left( \sqrt{\chi \psi} \right)}.
\end{equation*} 
The GH distributions contain a rich family of important distributions. Here, we collect the relevant examples in Appendix \ref{list of GH distributions}. For more details about normal mean-variance mixture distributions, we refer the reader to \cite[Chapter 3]{mcneil2005quantitative} for a comprehensive and
general overview of the theory and applications. More discussion about Bessel functions can be found in \cite{spanier1987atlas, gil2002evaluation, arfken2013mathematical} and references therein.

Suppose $v \sim GH_{N}(\lambda, \chi, \psi, \mu, \Sigma, \gamma)$ and $x \in \mathbb{R} \backslash \lbrace 0 \rbrace$. Define $\xi(x):= (v^{T}x - \mu^{T}x)/\sqrt{x^{T} \Sigma x}$. We will see $\xi(x)$ follows 1-dimensional generalized hyperbolic distribution. To this end, we provide the following \ZH{propositions}.

\begin{proposition}\cite[Proposition 3.9.]{mcneil2005quantitative}\label{scalarization of normal variance mixture distributions}
If $Y \sim M_{N}(\mu, \Sigma, \hat{H})$ and $Z = BY + b$, where $B \in \mathbb{R}^{k \times d}$ and $b \in \mathbb{R}^{k}$, then $Z \sim M_{k}(B\mu + b, B \Sigma B^{T}, \hat{H})$.
\end{proposition}

\begin{proposition}\cite[Proposition 3.13.]{mcneil2005quantitative}\label{scalarization of generalized hyperbolic distributions}
If $Y \sim GH_{N}(\lambda, \chi, \psi, \mu, \Sigma, \gamma)$ and $Z = B Y + b$, where $B \in \mathbb{R}^{k \times d}$ and $b \in \mathbb{R}^{k}$, then $Z \sim GH_{k}(\lambda, \chi, \psi, B\mu + b, B\Sigma B^{T}, B\gamma)$.
\end{proposition}

\begin{proposition}\label{1-dimensional generalized hyperbolic distributions}
Assume $x \in \mathbb{R}^{N} \backslash \lbrace 0 \rbrace $. Suppose $v \sim GH_{N}(\lambda, \chi, \psi, \mu, \Sigma, \gamma)$. Then, we have $\xi \sim GH_{1}(\lambda, \chi, \psi, 0, 1, \frac{x^{T}\gamma}{\sqrt{x^{T}\Sigma x}})$. In particular, if $v \sim M_{N}(\mu, \Sigma, \hat{H})$, then $\xi \sim M_{1}(0, 1, \hat{H})$.
\end{proposition}
\begin{proof}
\ZH{Let}
\begin{equation*}
B := \frac{x^{T}}{\sqrt{x^{T} \Sigma x}} \in \mathbb{R}^{1\times d}, \ \ b := \frac{-x^{T}\mu}{\sqrt{x^{T} \Sigma x}} \in \mathbb{R}.
\end{equation*}
The proof is completed by using \ZH{Propositions} \ref{scalarization of normal variance mixture distributions} and \ref{scalarization of generalized hyperbolic distributions} with the above $B$ and $b$.
\end{proof}
\ZHnew{Notice} that $\xi \sim M_{1}(0, 1, \hat{H})$ is equivalent to $\xi \sim GH_{1}(\lambda, \chi, \psi, 0, 1, 0)$. Suppose $\xi(x) \sim GH_{1}(\lambda, \chi, \psi, 0, 1, \frac{x^{T}\gamma}{\sqrt{x^{T}\Sigma x}})$. Using \eqref{non-integral form _ density of GH_d random vector with W density form}, we can formulate the density of $\xi(x)$ as follows:
\begin{equation}\label{1-dimensional non-integral form _ density of GH_d random vector with W density form}
f_{\xi(x)}(t) = c \frac{K_{\lambda - \left( 1/2 \right)}\left( \sqrt{ \left( \chi + t^{2} \right) \left( \psi + \frac{\left( x^{T}\gamma \right)^{2}}{x^{T}\Sigma x} \right) } \right) e^{\frac{t\cdot x^{T}\gamma}{\sqrt{x^{T}\Sigma x}}}}{\left( \sqrt{\left( \chi + t^{2} \right) \left( \psi + \frac{\left( x^{T} \gamma \right)^{2}}{x^{T} \Sigma x} \right)} \right)^{\left( \frac{1}{2} - \lambda \right)}},
\end{equation}
where $c$ is defined by
\begin{equation*}
c = \frac{\left( \sqrt{\chi \psi} \right)^{-\lambda} \psi^{\lambda} \left( \psi + \frac{\left( x^{T} \gamma \right)^{2}}{x^{T} \Sigma x} \right)^{\left(\frac{1}{2} - \lambda \right)}}{\sqrt{2\pi}K_{\lambda}(\sqrt{\chi \psi})}.
\end{equation*}

\ZH{The following} corollary provides a special case with $\psi = 0$ and $x^{T} \gamma / \sqrt{x^{T} \Sigma x} = 0$. From the Definition \ref{definition of Generalized hyperbolic distributions}, we can see $\psi = 0$ can be taken only when $\lambda < 0$ and $\chi > 0$. Thus, in the following corollary, we only need to consider the case $\psi = 0$ under $\lambda < 0$ and $\chi > 0$.
\begin{corollary}\label{density of special GH_1}
Assume $\lambda < 0$ and $\chi > 0$. Suppose $Y \sim GH_{1}\left( \lambda, \chi, 0, 0, 1, 0 \right)$. Then, the density of the random variable $Y$ can be written \ZH{as}
\begin{equation*}
f_{Y}(t) = \hat{c} \cdot \left( \chi + t^{2} \right)^{\lambda - \frac{1}{2}},
\end{equation*}
where $\hat{c}$ is defined by
\begin{equation*}
\hat{c}: = \frac{\Gamma \left( -\lambda + \frac{1}{2} \right)}{\Gamma \left( -\lambda \right)} \frac{1}{\sqrt{\pi}} \cdot \chi^{-\lambda}.
\end{equation*}
\end{corollary}
The proof of Corollary \ref{density of special GH_1} is shown in Appendix \ref{Appendix H}.

Using \ZH{Proposition} \ref{1-dimensional generalized hyperbolic distributions}, we know that $\xi(x):= (v^{T}x - \mu^{T}x)/\sqrt{x^{T} \Sigma x}$ is 1-dimensional generalized hyperbolic distribution \ZH{such that}
\begin{equation*}
\xi(x) \sim GH_{1}(\lambda, \chi, \psi, 0, 1, \frac{x^{T}\gamma}{\sqrt{x^{T}\Sigma x}}).
\end{equation*}
Under mild assumptions, we will see \ZH{that} the method used to prove the elliptical case in Section \ref{Eventual convexity with elliptical distributions} can be extended to prove the skewed case. More precisely, the significant difficulty of this extension is \ZH{to prove} that the density function of $\xi$ is $\alpha$-decreasing with some suitable $\alpha \in \mathbb{R}$.  Since the density of GH distribution is defined by the modified Bessel functions of the third kind, which is a kind of extreme non-linear function \cite{arfken2013mathematical, gil2002evaluation}. Thus, it is more difficult to prove the $\alpha$-decreasing property than the case with Gaussian, t-Student, and other elliptical distributions \cite[Proposition 4.1]{henrion2008convexity}, \cite[Proposition 5]{cheng2014second}. We will provide an approximation method to \ZH{get} the $\alpha$-decreasing of densities with the help of some properties of the modified Bessel functions. Here, we collect some properties about the modified Bessel functions of the third kind in Appendix \ref{Properties of the modified Bessel functions of the third kind} Lemma \ref{The properties of modified Bessel functions of the third kind}, which will play an essential role in the following results.

\subsection{Convexity under less skewed GH distributions}

\ZH{The} following lemma provides the main result of this section about the $\alpha$-decreasing property of $f_{\xi(x)}$. In contrast to the classical elliptical distributions, \ZH{it is more challenging to discuss} the $\alpha$-decreasing property of skewed distributions due to \ZH{the complexity} of the Bessel functions. A set of estimates relevant to the modified Bessel function of the third kind plays a crucial role in proving the following lemma. We provide the complete proof in Appendix \ref{Appendix H}.

\begin{lemma}\label{alpha-decreasing of density function}
Assume $\xi(x) \sim GH_{1}(\lambda, \chi, \psi, 0, 1, \frac{x^{T}\gamma}{\sqrt{x^{T}\Sigma x}})$. Let $\psi > 0$. Given $\alpha \in \mathbb{R}$, there exists $t^{*}(\alpha) > 0$ such that the density function $f_{\xi(x)}$ is $\alpha$-decreasing on $(t^{*}(\alpha), +\infty)$. Suppose $\psi = 0$ and $\phi := \left( x^{T}\gamma \right)/ \left( \sqrt{x^{T}\Sigma x} \right) \neq 0$. If $\phi < 0$, then $f_{\xi(x)}$ is $\alpha$-decreasing with any $\alpha \in \mathbb{R}$. If $\phi > 0$, then $f_{\xi(x)}$ is $\alpha$-decreasing with $\alpha < 1 - \lambda$. Suppose $\psi = 0$ and $\phi = 0$. Then, $f_{\xi(x)}$ is $\alpha$-decreasing if and only if $\alpha < 1 - 2 \lambda$ with $\lambda < 0$.
\end{lemma}

\begin{remark}
Due to the complexity of the Bessel functions, we have not provided the precise computation of the $t^{*}\left( \alpha \right)$ in Lemma \ref{alpha-decreasing of density function}. However, it is convenient to \ZH{come up with} $t^{*}\left( \alpha \right)$ by applying \ZHnew{a dichotomy algorithm}. Furthermore, we can use some approximation results about the lower and upper \ZHnew{bound} of $K_{\nu - 1} / K_{\nu}$ to simplify the computation \cite{segura2011bounds, ruiz2016new}.
\end{remark}

\ZHnew{We are now in a position to show the eventual convexity under less skewed GH distributions.}
\begin{theorem}\label{eventual convextiy of GH case 1}
Suppose $v_{i} \sim GH_{N}(\lambda_{i}, \chi_{i}, \psi_{i}, \mu_{i}, \Sigma_{i}, \gamma_{i})$, $i = 1, ..., K$. Assume $x^{T}\gamma_{i} = 0$ for any $x \in X$, $i = 1, ..., K$. \ZH{Assume Assumptions \ref{Assumption 2}, \ref{Assumption 3}, and \ref{Assumption 4} hold and $S(p) \neq \emptyset$.} Then, the \ZH{feasible set $S(p)$ defined by \eqref{Fea-1} is eventual convex.}
\end{theorem}
\begin{proof}
\ZHnew{Assume $\xi(x) \sim GH_{1}(\lambda, \chi, \psi, 0, 1, \frac{x^{T}\gamma}{\sqrt{x^{T}\Sigma x}})$. Let $F_{\xi(x)}$ be the distribution function of $\xi(x)$. Using Proposition \ref{1-dimensional generalized hyperbolic distributions}, we know that $v_{i}$ can be reformulated in the form of $\xi(x)$, where for convenience we suppress the subscript. Suppose $x^{T} \gamma = 0$ for any $x \in X$. Then, we can see $F_{\xi(x)}$ is independent of the decision variable $x$. Using Lemma \ref{alpha-decreasing of density function}, we can \ZHnew{show that} the density $f_{\xi(x)}$ is $\alpha$-decreasing with some $\alpha \in \mathbb{R}$. Thus, we can follow the proof of Lemma \ref{Proposition about concave Hi} to \ZHnew{come up with} the concavity of $F_{\xi(x)}\left(g \left( x \right) \right)$. Under row \ZHnew{based copula of} chance constraints, the convexity of $U(x, y_{i}) := \psi^{\left( -1 \right)}_{x} \left( y_{i} \psi_{x}\left( p \right) \right)$ shown in Lemma \ref{Lemma end 1} can be used literately because the right-hand side term of \eqref{Fea-7} has nothing to do with $\xi(x)$. \ZHnew{At the end}, following the proof of Theorem \ref{Theorem of convex feasible set}, we complete the proof.}
\end{proof}

It is worth noting that we assume $x^{T}\gamma = 0$ for any $x \in X$ in Theorem \ref{eventual convextiy of GH case 1}, which can be separated into the following cases: (1) $\gamma \neq 0$ and $\gamma$ is orthogonal to set $X$; (2) $\gamma = 0$. Case (1) means there is a hyperplane containing the set $X$, and $\gamma$ is a normal vector of the hyperplane. The case (1) corresponds to a chance constrained problem with asymmetric skewed distributions i.e. asymmetric $GH$ distributions. The case (2) corresponds to a particular case with normal variance mixture distributions, in which $W$ has a $GIG$ distribution. \ZHnew{Notice that, under rows independent chance constraint, the convexity of $U$ is trivial due to $U(x, y_{i})\equiv p$. Thus, in Theorem \ref{eventual convextiy of GH case 1} and in the \ZHnew{remaining} of Section \ref{Eventual convexity with skewed distributions}, we only provide the eventual convexity results about feasible sets of rows \ZHnew{based copula of} chance constraints.}

\subsection{Convexity with NMVM distributions}

\ZHnew{We} consider the case with NMVM distributions, which include the GH distribution as a special case. Under NMVM distributions, the main difference between the GH one's is that the method of incorporating randomness and skewness may \ZHnew{follow} a nonlinear form $m(W)$, not just a linear form $\mu + W \gamma$. Thus, we will suppose $m(W)$ is bounded w.r.t. random variable $W$. Then, in the following theorem, we will see the eventual convexity of the \ZH{feasible set} $S(p)$ can be satisfied, where \ZH{all} the random variables $\left\lbrace v_{i} \right\rbrace^{K}_{i = 1}$ follow NMVM distributions.
\begin{theorem}\label{eventual convexity of normal mean-variance mixture with bounded m(W)}
Suppose random vectors $\left\lbrace v_{i} \right\rbrace^{K}_{i = 1}$ \ZHnew{follow} normal mean-variance mixture distributions \ZHnew{as stated by} Definition \ref{definition of normal mean variance mixture distribution} i.e.
\begin{equation}\label{roll vectors v_i normal mean variance mixture distribution}
v_{i} = m_{i}(W_{i}) + \sqrt{W_{i}} A_{i} Z_{i}.
\end{equation}
Assume $\| m_{i}(W_{i}) \|$ is bounded w.r.t. random variable $W_{i}$. \ZH{Assume Assumptions \ref{Assumption 2}, \ref{Assumption 3}, and \ref{Assumption 4} hold and $S(p) \neq \emptyset$.} Then, the \ZH{feasible set} $S(p)$ \ZHnew{is} \ZH{eventual convex}.
\end{theorem}
\begin{proof}
\ZH{Let} $x \neq 0$. Fix $W_{i}$. Define $\xi_{i}$ and $g_{i}$ as follows
\begin{equation}\label{normalization of roll vectors v_i normal mean variance mixture distribution}
\xi_{i}(x) := \frac{v^{T}_{i} x - m_{i}(W_{i})^{T} x}{\sqrt{x^{T} \Sigma_{i} x}}, \ \ g_{i}(x) := \frac{D_{i} - m_{i}(W_{i})^{T} x}{\sqrt{x^{T} \Sigma_{i} x}}.
\end{equation}
Since $Z_{i} \sim N_{N}\left(0, I_{N} \right)$ implies $\sqrt{W_{i}} A_{i} Z_{i} \sim N_{N}\left(0, W_{i}\Sigma_{i} \right)$, where $\Sigma_{i} = A_{i} A^{T}_{i}$. As a result, by combining \eqref{roll vectors v_i normal mean variance mixture distribution} and \eqref{normalization of roll vectors v_i normal mean variance mixture distribution}, we can see
\begin{equation*}
\xi_{i}(x) \overset{d}{=} \frac{\left( \sqrt{W_{i}} A_{i} Z_{i} \right)^{T} x}{\sqrt{x^{T} \Sigma_{i} x}} \sim N_{1}\left(0, W_{i} \right).
\end{equation*}
which implies the distribution function of $\xi_{i}(x)$ is a normal distribution independent of \ZHnew{the} decision variable $x$. Using the Law of total expectation, we \ZH{have}
\begin{equation*}
\mathbb{P}\left( v_{i}^{T}x \leq D_{i} \right) = \mathbb{P}\left( \xi_{i}\left( x \right) \leq g_{i}\left( x \right) \right) = E_{W}\left[ F_{\xi_{i}} \left( g_{i}\left( x, m_{i}\left(\omega_{i}\right) \right) | W_{i} = \omega_{i} \right) \right].
\end{equation*}
It is well-known that normal densities are $\alpha$-decreasing with some $\alpha > 0$ \cite[Proposition 4.1]{henrion2008convexity}. Applying Theorem \ref{Corollary of Hi-1}, Lemma \ref{Proposition about concave Hi} and Theorem \ref{Theorem of convex feasible set}, we \ZH{get} the following results:
\begin{itemize}
\item[(1)] Fix $W_{i} = \omega_{i} \geq 0$. There is $r_{i} \in \mathbb{R}$ such that $g_{i}(x, m_{i}\left(\omega_{i}\right) )$ is locally $r_{i}$-concave w.r.t $x$ on $G(\theta^{*}_{i}(m_{i}\left(\omega_{i}\right)))$ defined by \eqref{"locally convex" set}, with corresponding threshold $\theta^{*}_{i}(m_{i}\left(\omega_{i}\right))$.
\item[(2)] The \ZHnew{composition} function $F_{\xi_{i}}\circ g_{i} \left( \cdot , m_{i}\left(\omega_{i}\right) \right)$ is concave w.r.t $x$ due to the $r$-concavity of $g_{i}\left( \cdot , m_{i}\left(\omega_{i}\right) \right)$ and $\left( -r_{i} + 1 \right)$-decreasing of density $f_{\xi_{i}}$ with suitable $r_{i} \in \mathbb{R}$ and $t^{**}_{i}\left( r_{i}, \omega_{i} \right)$.
\item[(3)] Due to the boundedness of $\| m_{i}(W_{i}) \|$ w.r.t. $W_{i}$, there exists $p*$ such that $F_{\xi_{i}}\circ g_{i} \left( \cdot , m_{i}\left(\omega_{i}\right) \right)$ are uniformly concave \ZH{on convS(p)} when $p > p^{*}$.
\item[(4)] The function $E_{W}\left[ F_{\xi_{i}} \left( g_{i}\left( x, m_{i}\left(\omega_{i}\right) \right) | W_{i} = \omega_{i} \right) \right]$ is concave \ZH{on convS(p)} since the positive linear combination of concave functions preserves concavity.
\end{itemize}
\ZH{As a result, we have proved the concavity of $H_{i} := F_{\xi_{i}}(g_{i}) = \mathbb{P}(v^{T}_{i} x \leq D_{i})$, ($i= 1, ..., K$). Then, using Lemma \ref{Lemma end 1} and following the proof of Theorem \ref{Theorem of convex feasible set}, we complete the proof.}
\end{proof}

\subsection{Convexity under strong skewed GH distributions}

\ZH{We} considered the eventual convexity problem under GH distributions with $\gamma \perp X$ in Theorem \ref{eventual convextiy of GH case 1}. Then, we discussed a more general case under NMVM distributions with bounded $\| m\left( W \right) \|$ in Theorem \ref{eventual convexity of normal mean-variance mixture with bounded m(W)}. In these two cases, we weakened the influence of the parameter $\gamma$, which reflects skewness. In the following, we consider a more general case regarding the skewness $\gamma$. Suppose $u \in \lbrace -1, 1 \rbrace$, $W \geq 0$, $x \in \mathbb{R}^{N}\setminus \lbrace o \rbrace$, $\gamma \in \mathbb{R}^{N}$. Assume $\Sigma$ is a positive definite matrix and $M \subset \mathbb{R}$ is a measurable set. Inspired by the radial decomposition method under elliptical distributions in \cite{van2019eventual}, we define the ray function of $M$ in a skewed version as follows
\begin{equation*}
\rho_{M}\left( x, u \right) :=
\left\{
        \begin{array}{l}
        \sup\limits_{t \geq 0} \ t\\            
        s.t. \ \ W \cdot \frac{x^{T} \gamma}{\sqrt{x^{T} \Sigma x}} + \sqrt{W} t u \in M.
        \end{array}
\right.
\end{equation*}

\begin{lemma}\label{lemma of general case with gamma times x greater than 0}
Assume $D \in \mathbb{R}$ and $v \sim GH_{N}(\lambda, \chi, \psi, \mu, \Sigma, \gamma)$. Suppose $x \in \mathbb{R}^{N} \setminus \left\lbrace 0 \right\rbrace$ and $x^{T}\gamma \geq 0$. Define the probability function as follows
\begin{equation*}
\Phi(x) := \mathbb{P} \left( v^{T} x \leq D \right).
\end{equation*}
Then, the probability function $\Phi$ can be rewritten \ZH{as}
\begin{equation*}
\Phi(x) = \int^{+\infty}_{0}\frac{1}{2} F_{\mathcal{R}} \left( \frac{D - \left( \mu + \omega \cdot \gamma \right)^{T}x}{\sqrt{\omega} \cdot \sqrt{x^{T} \Sigma x}} \right) d\mu_{W}(\omega) + \frac{1}{2},
\end{equation*}
where $W \sim N^{-}\left( \lambda, \chi, \psi \right)$, $F_{\mathcal{R}}$ is a distribution function of a radial random variable and $\mu_{W}$ is the law of the random variable $W$.
\end{lemma}
The proof of Lemma \ref{lemma of general case with gamma times x greater than 0} is presented in Appendix \ref{Appendix H}. In Lemma \ref{lemma of general case with gamma times x greater than 0}, we \ZH{reformulate} the probability function into a \ZH{composition} form \ZHnew{of} a radial distribution and a mixture random variable $W$. In this way, we only need to consider the \ZHnew{finite range of the integral} of the mixture random variable. Thus, the chance constraint with unbounded mixture random variable $W$ can be transformed into the bounded case.

\begin{theorem}
Assume $X \subseteq \mathbb{R}^{N}$ is a convex closed bounded set and $o \notin X$. Suppose random vectors $\left\lbrace v_{i} \right\rbrace^{K}_{i = 1}$ following generalized hyperbolic distributions defined by Definition \ref{definition of Generalized hyperbolic distributions} i.e. $v_{i} \sim GH_{N}(\lambda_{i}, \chi_{i}, \psi_{i}, \mu_{i}, \Sigma_{i}, \gamma_{i})$, $i=1, \cdots, K$. \ZH{Assume Assumptions \ref{Assumption 2}, \ref{Assumption 3}, and \ref{Assumption 4} hold and $S(p) \neq \emptyset$.} If $\gamma^{T}_{i} x \geq 0$ for any $x \in X$ and any $i=1, \cdots, K$, then the \ZH{feasible set} $S(p)$ is \ZH{eventual convex}.
\end{theorem}
\begin{proof}
Define the probability function by
\begin{equation*}
\Phi_{i}(x) := \mathbb{P} \left( v^{T}_{i} x \leq D_{i} \right), \ \ i=1, \cdots, K.
\end{equation*}
Using the reformulation \eqref{Fea-2} and \eqref{Fea-7}, we only need to consider the concavity of $\Phi_{i}$, regardless of whether there is a copula relationship between random vector rows. If $\gamma^{T}_{i} x = 0$, then the eventual convexity of $S(p)$ can be obtained by using Theorem \ref{eventual convextiy of GH case 1}. It remains to consider the case with $\gamma^{T}_{i} x > 0$. Using \ZH{Lemma} \ref{lemma of general case with gamma times x greater than 0}, we \ZHnew{have}
\begin{equation}\label{integral under radial decomposition}
\Phi_{i}(x) = \int^{\frac{D_{i} - \mu_{i}^{T}x}{\gamma_{i}^{T}x}}_{0}\frac{1}{2} F_{\mathcal{R}_{i}} \left( \frac{D_{i} - \left( \mu_{i} + \omega \cdot \gamma_{i} \right)^{T}x}{\sqrt{\omega} \cdot \sqrt{x^{T} \Sigma_{i} x}} \right) d\mu_{W_{i}}(\omega) + \frac{1}{2},
\end{equation}
where we use the fact the radial distribution $F_{\mathcal{R}_{i}}$ is non-zero only on $(0, +\infty)$ to induce the upper limit of integral. Since \ZHnew{the} domain $X \subset \mathbb{R}^{N} \setminus \lbrace 0 \rbrace$ is closed bounded. We can see \ZHnew{that} the upper limit of integral of \eqref{integral under radial decomposition} is bounded, which allows to repeat the result of Theorem \ref{eventual convexity of normal mean-variance mixture with bounded m(W)} with bounded $W_{i}$ to obtain the eventual convexity of $S(p)$.  
\end{proof}

\section{Simulations about eventual convexity with elliptical distributions}\label{Convexity of the feasible set $S(p)$}
\ZH{In this section, we first present a simulation about the eventual convexity under elliptical distributions. At the end, we provide some examples of the threshold $\sqrt{\theta^{*}}$ obtained in Theorem \ref{Corollary of Hi-1}.}

\subsection{Numerical simulation of eventual convexity.}\label{Examples}
This section is devoted to providing a precise $\kappa$ satisfying \ZH{Assumption} \ref{Assumption 3} and \ref{Assumption 4}. Let $(x, y_{i}) \in X \times (0, 1)$. \ZHnew{More precisely}, we take $y_{i} := 1/K$ \ZHnew{in} the following examples. Suppose $d \cdot \kappa(x) H_{x}\kappa(x) - \nabla_{x} \kappa(x) (\nabla_{x}\kappa(x))^{T}$ and $H_{x}\kappa(x)$ are positive semi-definite matrices. Applying \ZH{Lemma} \ref{Lemma end 2} \ZHnew{leads to}
\begin{equation}\label{example 1}
1/\omega(x, y_{i}) \cdot H_{x}\kappa(x) -  \nabla_{x} \kappa(x) (\nabla_{x}\kappa(x))^{T} \succeq d \cdot \kappa(x) H_{x}\kappa(x) - \nabla_{x} \kappa(x) (\nabla_{x}\kappa(x))^{T} \succeq 0.
\end{equation}
It follows that $H_{x}\kappa(x) - \omega(x, y_{i}) \nabla_{x} \kappa(x) (\nabla_{x}\kappa(x))^{T}$ is a positive semi-definite matrix. Using Lemma \ref{Lemma end 1},  we prove \ZHnew{that} $S(p)$ is a convex set. Thus, we provide an approach to define certain $\kappa$ such that $d \cdot \kappa(x) H_{x}\kappa(x) - \nabla_{x} \kappa(x) (\nabla_{x}\kappa(x))^{T}$ and $H_{x}\kappa(x)$ are positive semi-definite on $X$.

Suppose $d \geq 1$. Let $f$ be a convex function defined on $X$ satisfying $0 \geq f(x) > -\infty$. Define $\kappa(x) := e^{f(x)}$. It follows that
\begin{equation*}
H_{x}\left( ln \left( \kappa(x) \right) \right) = \frac{\kappa(x) H_{x}(\kappa(x)) - \nabla_{x} \kappa(x) \left( \nabla_{x} \kappa(x) \right)^{T}}{\kappa^{2}(x)} \succeq 0.
\end{equation*}
\ZHnew{In this case, we have} a function $\kappa(x)$ \ZHnew{that satisfies} \eqref{example 1}. However, when $d \in (0, 1)$, the above form of $\kappa$ is invalid. Therefore, we provide a new approach to define the function $\kappa$ in the following lemma. The proof is shown in Appendix \ref{Appendix HH}. In the following, we define $\triangle$ as follows
\begin{equation*}
\triangle := 
\begin{pmatrix}
H_{x} \kappa(x) & \nabla_{x}\kappa(x)\\
\left( \nabla_{x} \kappa(x) \right)^{T} & d \cdot \kappa(x)
\end{pmatrix}
\end{equation*}
and denote by $| \triangle |$ the determinant of the matrix $\triangle$.

\begin{lemma}\label{the construction of copula}
Suppose $X \subset \mathbb{R}^{K}$ is bounded. Define the minimum value of the vectors' components in $X$ as
\begin{equation*}
X_{min} := \underset{ i = 1, ... , K }{\inf} \left\lbrace x_{i} \left|  \forall x = [x_{1}, x_{2}, ..., x_{K}]^{T} \in X \right\rbrace\right.
\end{equation*}
Then, there exists a $\kappa(x) := \Sigma^{K}_{i = 1} f(x_{i})$ satisfying Assumption \ref{Assumption 3} and \ref{Assumption 4}, where $f$ is defined by
\begin{equation}\label{solution of differential equaiton 2}
f(x) :=
\left\{
        \begin{array}{ll}
        \left( \left( \frac{2}{d} - 1 \right)\cdot \left( C_{2} + C_{1} x \right) \right)^{\frac{d}{d - 2}}, & \text{if} \ \ d \in (0, 1),\\         
        \left( C_{2} + C_{1}x \right)^{-1}, & \text{if} \ \ d \in [1, +\infty),
        \end{array}
\right.
\end{equation}
and $C_{1}$, $C_{2}$ are \ZH{given constant coefficients.}
\end{lemma}

Notice that it is convenient to choose suitable $C_{1}$, $C_{2}$, \ZHnew{and} $X$. Suppose $d \in (0, 1)$. Let $C_{1} = 1$, $C_{2} = 10$, $K = 2$, $r = 9 - d/(4-2d) $. Define $X := B(0, r)$ as a ball with origin center and radius $r$. Then, $X_{min} = -r$ and for any $x=[x_{1},...,x_{K}]^{T} \in X$ we have
\begin{equation*}
f(x_{i}) = \left( \left( \frac{2}{d} - 1 \right) \cdot \left( x_{i} + 10 \right) \right)^{\frac{d}{d - 2}} > 0, 
\end{equation*}
\begin{equation*}
f''(x_{i}) = \frac{2d}{(d - 2)^{2}} \cdot \left( \frac{2}{d} - 1 \right)^{2} \cdot \left( \left( \frac{2}{d} - 1 \right) \cdot \left( x_{i} + 10 \right) \right)^{\frac{d}{d - 2} - 2} > 0,
\end{equation*}
\begin{equation*}
0 < \kappa(x) = \sum\limits^{K}_{i = 1} f(x_{i}) < 1,
\end{equation*}
where $X_{min} = -r > -9$ is used to prove the right inequality in the last equation. Suppose $d \in [1, +\infty)$. Take $r = 7$, $X = B(0, r)$. Substitute $C_{1} = 1$, $C_{2} = 10$ into \eqref{solution of differential equaiton 2}. Then, we \ZH{have}
\begin{equation*}
f(x_{i})=\frac{1}{x_{i} + 10} > 0, \ \ f'(x_{i})=-\frac{1}{(x_{i} + 10)^{2}} < 0, \ \ f''(x_{i}) = \frac{2}{(x_{i} + 10)^{3}} > 0,
\end{equation*}
\begin{equation*}
0 < \kappa(x) := \sum\limits^{2}_{i = 1}\frac{1}{x_{i} + 10} \leq \frac{1}{3} + \frac{1}{3} < 1,
\end{equation*}
\begin{equation*}
\begin{vmatrix}
\triangle
\end{vmatrix}
= \left( \prod\limits^{2}_{i = 1} \frac{2}{(x_{i} + 10)^{3}} \right) \cdot \left( \sum\limits^{2}_{i = 1} \frac{1}{x_{i} + 10} \right) \cdot (d - \frac{1}{2}) >0,
\end{equation*}
which satisfy \ZH{Assumption} \ref{Assumption 3} and \ref{Assumption 4}.

Finally, we give an example. We take $N = 2$, $K = 3$, 
\begin{equation*}
D :=
\begin{bmatrix}
22 \\
27 \\
2 
\end{bmatrix}, \ \ 
\mu :=
\begin{bmatrix}
3 & 0 & 2\\
-4 & 1 & -1
\end{bmatrix}, \ \ 
r :=
\begin{bmatrix}
-2 \\
-2 \\
-2
\end{bmatrix}, \ \ 
\end{equation*}
and
\begin{equation*}
\Sigma_{1} :=
\begin{bmatrix}
96 & -11 \\
-11 & 98
\end{bmatrix}, \ \ 
\Sigma_{2} :=
\begin{bmatrix}
44 & 21 \\
21 & 92
\end{bmatrix}, \ \ 
\Sigma_{3} :=
\begin{bmatrix}
90 & -2 \\
-2 & 24
\end{bmatrix}. \ \ 
\end{equation*}
We can compute the $\theta^{*} : = (\theta^{*}_{1}, \theta^{*}_{2}, \theta^{*}_{3})$ defined in Theorem \ref{Corollary of Hi-1} \ZHnew{which} is equal to $(2.4343, 0.1965, 1.4433)$, which corresponds to $\text{max}_{i \in \lbrace  1,2,3 \rbrace} F_{i} ( \sqrt{ \theta^{*}_{i} } ) = 0.9031$. We consider the Student's $t$ with 4-degree, so we have $F_{i}(t^{*}(-r_{i} + 1)) = 0.9648$. Thus, based on \ZH{Assumption} \ref{Assumption 2} and \ref{Assumption 3} we \ZH{get} $p^{*} = 0.9648$. Here, we take $p = 0.97$. The Figure \ref{The surface of a probability function with Student distribution} and \ref{The contour line of a probability function with Student distribution} show the surface and contour lines of the probability function defined by the left-hand side of \eqref{Fea-4}. It can be seen that the origin is contained in $S(p)$. Figure \ref{The contour line of a probability function with Student distribution} shows that $p^{*}$ is a conservative estimate because the convexity of $S(p)$ begins around $p = 0.5$.
\begin{figure}[ht]
\begin{center}
\includegraphics[width=2.0 in]{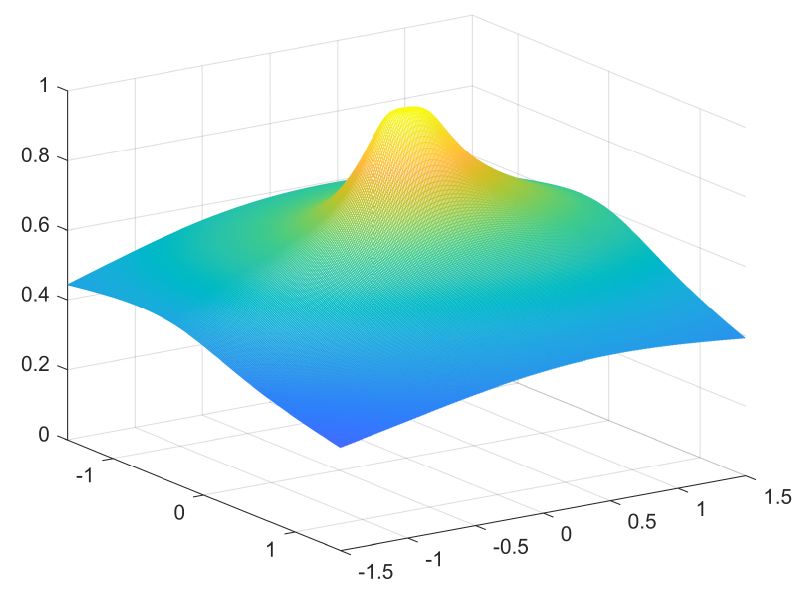}
\end{center}
\caption{The surface of a probability function with Student distribution.}
\label{The surface of a probability function with Student distribution}
\end{figure}
\begin{figure}[ht]
\begin{center}
\includegraphics[width=2.0 in]{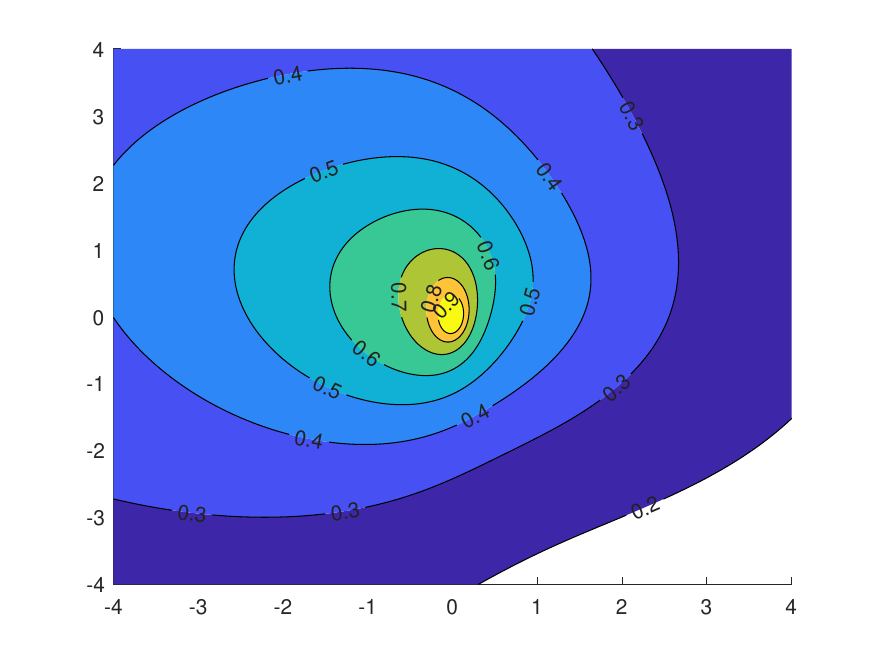}
\end{center}
\caption{The contour line of a probability function with Student distribution.}
\label{The contour line of a probability function with Student distribution}
\end{figure}

\subsection{Examples about the threhold $\sqrt{\theta^{*}}$.}\label{Examples2}
In this section, we \ZHnew{perform} numerical experiments to discuss the threshold $\sqrt{\theta^{*}}$ defined in Theorem \ref{Corollary of Hi-1}. First, we present examples of non-existence of $\sqrt{\theta^{*}}$ with $b > 0$ and $r \geq -1$ shown in Figures \ref{Fig_r=-1} and \ref{Fig_r<-1}, in which we take
\begin{equation}\label{choice of the mean and variance for numerical}
b := 4, \ \ 
\mu :=
\begin{bmatrix}
0 \\
28 \\
-1 
\end{bmatrix}, \ \ 
\Sigma :=
\begin{bmatrix}
32 & 20 & 3\\
20 & 26 & 23\\
3 & 23 & 38
\end{bmatrix}.
\end{equation}
\begin{figure}[ht]
\begin{center}
\subfigure[The contour of h and g with $b = 4$, $r = -1$ and $c_{0} = 20$.]{
\includegraphics[width=1.5 in]{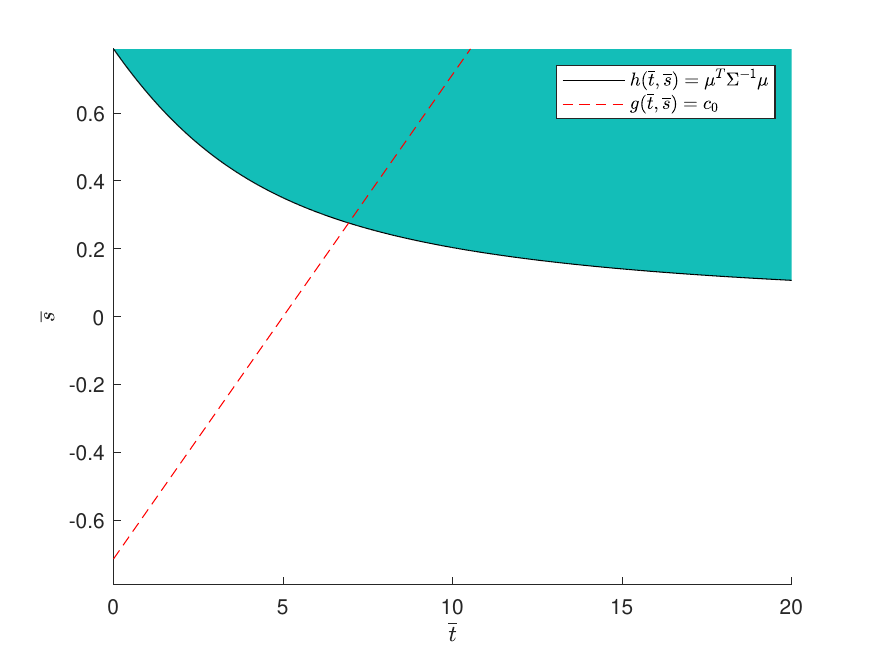} \label{Fig_r=-1}
}
\quad
\subfigure[The contour of h and g with $b = 4$, $r = 1$  and $c_{0} = 20$.]{
\includegraphics[width=1.5 in]{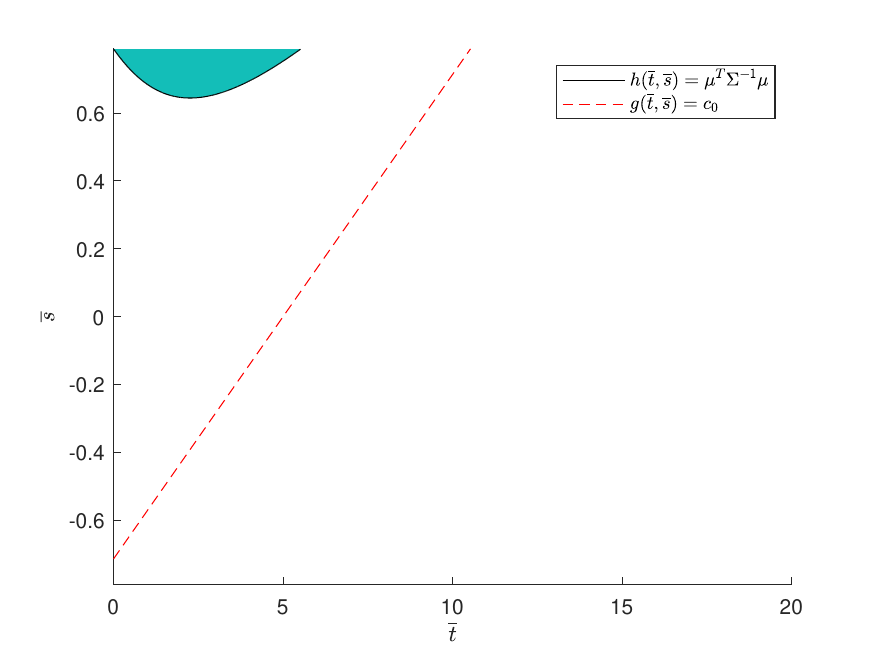} \label{Fig_r<-1}
}
\end{center}
\caption{The graphs about non-existence of $\sqrt{\theta^{*}}$.}
\end{figure}
In the Figures \ref{Fig_r=-1} and \ref{Fig_r<-1}, the shadow areas correspond to the set $Q$ defined by \eqref{tangent_Q set}, i.e.
\begin{equation*}
Q := \left\lbrace (\overline{t}, \overline{s}) \in (0, +\infty) \times [-1/\sqrt{\lambda_{\mu, \text{min}}}, 1 / \sqrt{\lambda_{\mu, \text{min}}}]: h(\overline{t}, \overline{s}) \geq \mu^{T} \Sigma^{-1} \mu \right\rbrace,
\end{equation*}
and the dotted line means \ZH{the line} $g(\overline{t}, \overline{s}) = c_{0}$ with $c_{0}=20$, and the right hand side of the dotted line corresponds to the set $G(c_{0})$ defined by \eqref{tangent_G set}, i.e.
\begin{equation*}
G(c_{0}) := \left\lbrace (\overline{t}, \overline{s}) \in (0, +\infty) \times [-1/\sqrt{\lambda_{\mu, \text{min}}}, 1 / \sqrt{\lambda_{\mu, \text{min}}}]: g(\overline{t}, \overline{s}) \geq c_{0} \right\rbrace.
\end{equation*}
Due to Theorem \ref{Corollary of Hi-1}, the threshod $\sqrt{\theta^{*}}$ can be induced by
\begin{equation*}
\sqrt{\theta^{*}} = \inf \left\lbrace c_{0} \geq 0: G(c_{0}) \subset Q \right\rbrace,
\end{equation*}
which means that there is no $\sqrt{\theta^{*}}$ when the right hand side of the dotted line can not be completely contained at the shadow ares.
\begin{figure}[ht]
\begin{center}
\includegraphics[width=2.0 in]{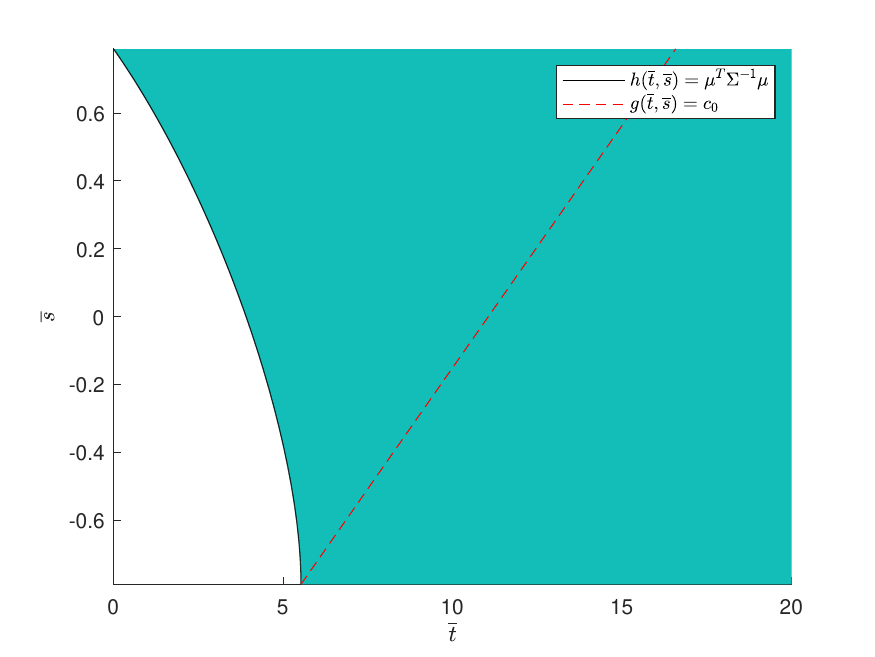}
\end{center}
\caption{The graphs about existence of $\sqrt{\theta^{*}}$.}
\label{Fig_r_less_-1_non-tangent_2D}
\end{figure}
An example corresponding to the Case (2) in Lemma \ref{tangent relationship between s1 and s2} is shown in Figure \ref{Fig_r_less_-1_non-tangent_2D}, where we take $b = 4, r = -3$. In Figure \ref{Fig_r_less_-1_non-tangent_2D}, we see \ZHnew{that} the right hand side of the dotted line is contained into the shadow area so that there exists a threshold $\sqrt{\theta^{*}}$. Finally, we provide a graph about the threshod $\sqrt{\theta^{*}}$ in Theorem \ref{Corollary of Hi-1}, where $b, \mu, \Sigma$ are defined by \eqref{choice of the mean and variance for numerical}. The graph of $\sqrt{\theta^{*}}$ is shown in Figure \ref{The graph of sqrt_theta with r_less_-1}. \ZHnew{We see} that $\sqrt{\theta^{*}}$ \ZHnew{goes} to infinity as $r$ \ZH{is} close to $-1$.
\begin{figure}[ht]
\begin{center}
\includegraphics[width=2.0 in]{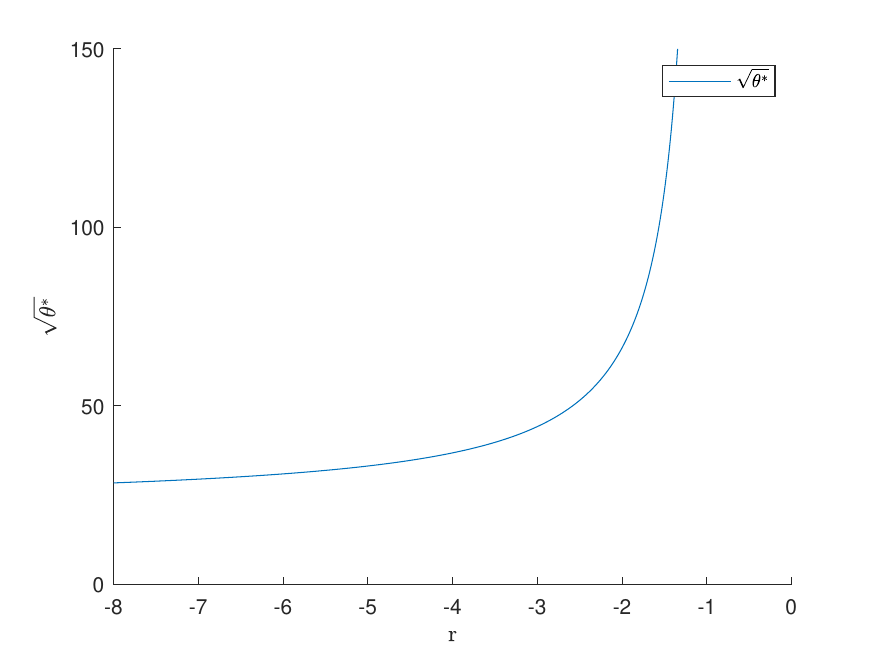}
\end{center}
\caption{The graph of $\sqrt{\theta^{*}}$ with $b > 0$}
\label{The graph of sqrt_theta with r_less_-1}
\end{figure}

\section{Conclusion}
In this paper, \ZH{we proved} the convexity of chance constraints with elliptical symmetric and skewed distributions. We first consider a \ZHnew{row} dependent elliptical random matrix via a Gumbel-Hougaard copula, where the copula is dependent on a decision vector. The eventual convexity of chance constraints is separated as the concavity and convexity of two auxiliary functions: the first one is a probability function, and a copula's generator defines the second one. For the probability function, we obtain several necessary and sufficient conditions of $r$-concavity and use a geometric method to compute \ZHnew{the} best thresholds of the $r$-concavity. Further, we obtain \ZHnew{a better} probability threshold of eventual convexity when the distribution is Gaussian. Considering the function defined by a copula's generator, we extend some assumptions about the copula and \ZHnew{$\Omega$ with the origin}. We overcome the difficulty that arose from the singularity of Gumbel-Hougaard copulas around the origin. As extensions, we \ZHnew{give} the eventual convexity with some skewed distributions. We first consider the situation when the skewness parameter of GH distributions is orthogonal to the domain, as an extension of our previous results in elliptical cases. To this end, we prove the GH density is $\alpha$-decreasing by using a series of approximations about the modified Bessel functions. Next, we consider the eventual conexity under more general skewed distributions, i.e., NMVM distributions, with a bounded mean mixture vector. At the end, we consider a situation under a GH distribution with stronger influence from the skewness parameter. Based on a radial decomposition method, we solve this problem by means of the former two skewed eventual convexity results. About numerical experiments, we provide an eventual convexity simulation under elliptical distributions by means of a special copula function. And we present some examples to show the existence and non-existence of the threshold.

\section*{Acknowledgements}

This research was supported by French government under the France 2030 program, reference ANR-11-IDEX-0003 within the OI H-Code.





\bibliography{reference}

\begin{appendices}

\section{Archimedean copulas}\label{list of Archimedean copulas}
\ 
\begin{table*}[!h]
			\centering
			\caption{Generators of a set of copulas.}
\def\arraystretch{1.3}      \begin{tabular}{rc|cc}
   \cmidrule{2-4}
   & Copula family & Parameter $\theta$ & Generator $\Phi_{\theta}$  \\ \cmidrule{2-4}
   & Independent & - & -$\ln t$ \\
   & Gumbel-Hougaard & $\theta \geq 1$ & $\left( -\ln t \right)^{\theta}$  \\
   & Clayton & $\theta > 0$ & $\theta^{-1}\left( t^{-\theta} - 1 \right)$  \\
   & Frank & $\theta > 0$ & $-\ln \left( \frac{e^{-\theta t} - 1}{e^{-\theta} - 1} \right)$ \\
   & Joe & $\theta \geq 1$ & $-\ln \left( 1 - \left( 1 - t \right)^{\theta} \right)$ \\ \cmidrule{2-4}
\end{tabular}
\label{Generators of a set of copulas}
\end{table*}

\section{Generalized hyperbolic distributions}\label{list of GH distributions}
{
\ 
\begin{table*}[!h]
			\centering
			\caption{Special cases of Generalized hyperbolic (GH) distributions \cite{mcneil2005quantitative, jiang2023review, fischer2023variance}.}
\def\arraystretch{1.3}			
\begin{tabular} {rc|cccccc}
   \cmidrule{2-8}
   & Distribution family &  $\lambda$ & $\chi$ & $\psi$ & $\mu$ & $\Sigma$ & $\gamma$  \\ \cmidrule{2-8}
   & Hyperbolic & $\left( N + 1 \right) / 2$ & $\chi$ & $\psi$ & $\mu$ & $\Sigma$ & $\gamma$ \\
   & Normal inverse Gaussian & $-1/2$ & $\chi$ & $\psi$ & $\mu$ & $\Sigma$ & $\gamma$  \\
   & Laplace & $\lambda$ & $0$ & $\psi$ & $\mu$ & $\Sigma$ & $0$  \\
   & Variance Gamma & $\lambda$ & $0$ & $\psi$ & $\mu$ & $\Sigma$ & $\gamma$  \\
   & Student's t & $-\nu/2$ & $\nu$ & $0$ & $\mu$ & $\Sigma$ & $0$  \\
   & Skewed Student's t & $-\nu/2$ & $\nu$ & $0$ & $\mu$ & $\Sigma$ & $\gamma$  \\
   & GH Skewed Student's t & $-\nu/2$ & $\chi$ & $0$ & $\mu$ & $\Sigma$ & $\gamma$  \\
   & Gaussian & $+\infty$ & $+\infty$ & $0$ & $\mu$ & $\Sigma$ & $0$  \\ \cmidrule{2-8}
\end{tabular}
\label{special cases of GH distributions}
\end{table*}

}

\section{Generalized inverse Gaussian (GIG) distributions}\label{definition of Generalized inverse Gaussian distributions}
\begin{definition}
The random variable $Y$ has a \textbf{generalized inverse Gaussian (GIG)} \ZH{distribution $Y \sim N^{-}(\lambda, \chi, \psi)$ if its density is}
\begin{equation*}
f_{Y}(t) = \frac{\chi^{-\lambda}\left( \sqrt{ \chi \psi} \right)^{\lambda}}{2 K_{\lambda} \left( \sqrt{\chi \psi} \right)} t^{\lambda - 1} exp \left( -\frac{1}{2}\left( \chi t^{-1} + \psi t \right) \right), \ \ t > 0,
\end{equation*}
where $K_{\lambda}$ denotes a modified Bessel function of the third kind with index $\lambda$ and the parameters satisfy
\begin{equation*}
\left\{
        \begin{array}{ll}
        \chi > 0, \psi \geq 0, & \text{if} \ \ \lambda < 0,\\         
        \chi > 0, \psi > 0, & \text{if} \ \ \lambda = 0,\\
        \chi \geq 0, \psi > 0, & \text{if} \ \ \lambda > 0.
        \end{array}
\right.
\end{equation*}
In particular, when $\psi = 0$, GIG density contains the \textbf{inverse gamma (IG)} density as a special limiting case. The random variable $Y$ has \ZH{an} inverse gamma distribution, written by $Y \sim Ig(\alpha, \beta)$, if its density is
\begin{equation*}
f_{Y}(t) = \frac{\beta^{\alpha}}{\Gamma(\alpha)} t^{-\left( \alpha + 1 \right)} \cdot e^{-\frac{\beta}{t}}, \ \ t>0, \ \alpha > 0, \ \beta > 0.
\end{equation*}
Assume $\lambda < 0$ and $\psi = 0$. Then, the random variable $Y \sim N^{-}(\lambda, \chi, 0)$ is equivalent to $Y \sim Ig(-\lambda, \frac{1}{2} \chi)$, which means the density of $Y$ can be written \ZH{as}
\begin{equation*}
f_{Y}(t) = \frac{\left( \frac{1}{2}\chi \right)^{-\lambda}}{\Gamma \left( - \lambda \right)} \cdot t^{\lambda - 1} \cdot e^{-\frac{1}{2} \chi \cdot t^{-1}}.
\end{equation*}
\end{definition}

\section{Properties of the modified Bessel functions of the third kind}\label{Properties of the modified Bessel functions of the third kind}
\begin{lemma}\cite{spanier1987atlas, gil2002evaluation}\label{The properties of modified Bessel functions of the third kind}
The properties of the modified Bessel functions of the third kind, denoted it by $K_{\nu}$:

(1) $K_{\nu}(x) > 0$, for any $x>0$, $\nu \in \mathbb{C}$.

(2) Given $\nu \in \mathbb{C}$, the $K_{\nu}(x)$ is decreasing w.r.t. $x$ on $(0, +\infty)$.

(3) Given $\nu \in \mathbb{C}$, the $K_{\nu}(x)$ satisfies
\begin{equation*}
\lim_{x \to +\infty} K_{\nu}(x) = 0.
\end{equation*}

(4) Take $x > 0$. $K_{\nu}(x) = K_{-\nu}(x)$ for any $\nu \in \mathbb{C}$.

(5) Take $x > 0$. The $K_{\nu}(x)$ increases as $| \nu |$ increasing.

(6) For large $x$, the approximation of $K_{\nu}(x)$ with given $\nu$ can be written as
\begin{equation*}
K_{\nu}(x) \simeq \sqrt{\frac{\pi}{2x}} \cdot exp(-x) \cdot \left( 1 + \frac{1}{x} \right)^{\mu}, \ \ x \gg \mu := \frac{\nu^{2}}{2} - \frac{1}{8}.
\end{equation*}

(7) Given $\nu \in \mathbb{R}$ and $x > 0$, then $K_{\nu}$ and $K_{\nu + 1}$ have the following relationship
\begin{equation*}
K^{'}_{\nu}(x) = \frac{\nu}{x} \cdot K_{\nu}(x) - K_{\nu + 1}(x).
\end{equation*}

(8) For small $x \to 0^{+}$, the approximation of $K_{\nu}(x)$ can be written as
\begin{equation*}
\left\{
        \begin{array}{ll}
        K_{\nu}(x) \simeq \Gamma(\nu) \cdot 2^{\nu - 1} \cdot x^{-\nu}, & \text{if} \ \ \nu > 0,\\         
        K_{\nu}(x) \simeq -\ln (x) + \ln 2 - \hat{\gamma}, & \text{if} \ \ \nu = 0,\\         
        K_{\nu}(x) \simeq \Gamma(-\nu) \cdot 2^{-\nu - 1} \cdot x^{\nu}, & \text{if} \ \ \nu < 0,
        \end{array}
\right.
\end{equation*}
where \ZH{$\hat{\gamma} \approx 0.5772$ is the Euler's constant.}
\end{lemma}

\section{Proof of Section \ref{Main results about eventual convexity with elliptical distributions} \ZHnew{results}}\label{proof of Lemma about Hi-0}
\begin{lemma}\label{Lemma about Hi-0}
Let $Y, Z \in \mathbb{R}^{N}$. Then the matrices composed of $Y$ and $Z$ have the following eigenvalues :
\begin{itemize}
\item[(1)] The matrix $Z \cdot Z^{T} + Y \cdot Y^{T}$ has at most two positive eigenvalues $\lambda_{1}, \lambda_{2} (\lambda_{1} \geq \lambda_{2})$ defined by
\begin{equation}\label{eigenvalue 1-1}
\lambda_{1,2} = \frac{\| Y \|^{2} + \| Z \|^{2} \pm \sqrt{\left( \| Y \|^{2} - \| Z \|^{2} \right)^{2} + 4\left( Y^{T}Z \right)^{2} }}{2}.
\end{equation}
\item[(2)] The matrix $Z \cdot Z^{T}$ has at most one positive eigenvalue
\begin{equation}\label{eigenvalue 2-1}
\lambda_{1} = \| Z \|^{2}.
\end{equation}
\item[(3)] The matrix $Z \cdot Z^{T} - Y \cdot Y^{T}$ only has two non-zero eigenvalue defined by
\begin{equation}\label{eigenvalue 3-1}
\lambda_{1} = \frac{\| Z \|^{2} - \| Y \|^{2} + \sqrt{\left( \| Z \|^{2} + \| Y \|^{2}\right)^{2} - 4\left( Y^{T}Z \right)^{2} }}{2} > 0
\end{equation}
and
\begin{equation}\label{eigenvalue 3-2}
\lambda_{2} = \frac{\| Z \|^{2} - \| Y \|^{2} - \sqrt{\left( \| Z \|^{2} + \| Y \|^{2}\right)^{2} - 4\left( Y^{T}Z \right)^{2} }}{2} < 0,
\end{equation}
or only has zero eigenvalues.
\end{itemize}
\end{lemma}
\begin{proof}
First, we consider the matrix $Z \cdot Z^{T} +Y \cdot Y^{T}$. Observe that $Z \cdot Z^{T}$ and $Y \cdot Y^{T}$ both are positive semi-definite matrices with rank 1. Then, we can see $Z \cdot Z^{T} +Y \cdot Y^{T}$ is a real symmetric matrix and $\text{rank}(Z \cdot Z^{T} +Y \cdot Y^{T}) \leq 2$. It follows that there exists an orthogonal matrix $T$ such that $T^{-1} (Z \cdot Z^{T} + Y \cdot Y^{T}) T$ is a diagonal matrix, which implies the number of non-zero eigenvalues equal to the rank of $Z \cdot Z^{T} + Y \cdot Y^{T}$. Then, there are at most two non-zero eigenvalues of $Z \cdot Z^{T} +Y \cdot Y^{T}$. Set $Z = (z_{1}, \cdots ,z_{N})^{T}$, $Y = (y_{1}, \cdots , y_{N})^{T}$, then we have
\begin{equation*}
Z \cdot Z^{T} + Y \cdot Y^{T} = 
\begin{pmatrix}
y^{2}_{1} + z^{2}_{1} & y_{1}y_{2} + z_{1}z_{2} & \cdots & y_{1}y_{N} + z_{1}z_{N}\\
y_{1}y_{2} + z_{1}z_{2} & y^{2}_{2} + z^{2}_{2} & \cdots & y_{2}y_{N} + z_{2}z_{N}\\
\vdots & \vdots & \ddots & \vdots\\
y_{1}y_{N} + z_{1}z_{N} & y_{2}y_{N} + z_{2}z_{N} & \cdots & y^{2}_{N} + z^{2}_{N}
\end{pmatrix}.
\end{equation*}
Since the matrix $Z \cdot Z^{T} +Y \cdot Y^{T}$ has at most two non-zero eigenvalues, the characteristic polynomial of $Z \cdot Z^{T} +Y \cdot Y^{T}$ can be written as
\begin{equation}\label{characteristic polynomial of ZZ + YY}
\begin{aligned}
&\left| \lambda \cdot I - \left( Z \cdot Z^{T} + Y \cdot Y^{T} \right) \right| \\
& =  \lambda^{N} -  \lambda^{N-1} \left[ \left( y^{2}_{1} + \cdots + y^{2}_{N} \right) + \left( z^{2}_{1} + \cdots + z^{2}_{N} \right) \right]\\
& + \lambda^{N-2} \left[ \underset{i}{\Sigma} \underset{j > i}{\Sigma} \left( y_{i} z_{j} - y_{j} z_{i} \right)^{2} \right],
\end{aligned}
\end{equation}
where $I$ is the identity matrix. Set
\begin{equation*}
\begin{aligned}
&a := 1,\\
&b := -\left( \| Y \|^{2} + \| Z \|^{2} \right),\\
&c := \underset{i}{\Sigma} \underset{j > i}{\Sigma} \left( y_{i} z_{j} - y_{j} z_{i} \right)^{2},\\
&\Delta := b^{2} - 4ac.
\end{aligned}
\end{equation*}
Since the eigenvalues of symmetric matrices are all real numbers, we have $\Delta \geq 0$. If $\Delta = 0$, then $\lambda_{1} := -b/(2a)= (\| Y \|^{2} + \| Z \|^{2}) /2$ is the only one root of \eqref{characteristic polynomial of ZZ + YY}. If $\Delta > 0$, then there are two characteristic roots of \eqref{characteristic polynomial of ZZ + YY}, $\lambda_{1}, \lambda_{2} (\lambda_{1} \geq \lambda_{2})$, which can be written as
\begin{equation*}
\begin{aligned}
\lambda_{1,2} & = \frac{\| Y \|^{2} + \| Z \|^{2} \pm \sqrt{\left( \| Y \|^{2} + \| Z \|^{2} \right)^{2} - 4 \left( \underset{i}{\Sigma} \underset{j > i}{\Sigma} \left( y_{i} z_{j} - y_{j} z_{i} \right)^{2} \right)}}{2}\\
& = \frac{\| Y \|^{2} + \| Z \|^{2} \pm \sqrt{\left( \| Y \|^{2} - \| Z \|^{2} \right)^{2} + 4 \left( Y^{T} Z \right)^{2}}}{2}.
\end{aligned}
\end{equation*}
The proof of the first result is complete. A similar proof works for the second and third results. We only highlight the inequalities of \eqref{eigenvalue 3-1} and \eqref{eigenvalue 3-2}. Similar to the first result, the characteristic polynomial of $Z \cdot Z^{T} - Y \cdot Y^{T}$ can be written as
\begin{equation}\label{characteristic polynomial of ZZ - YY}
\begin{aligned}
&\left| \lambda \cdot I - \left( Z \cdot Z^{T} - Y \cdot Y^{T} \right) \right| \\
& =  \lambda^{N} -  \lambda^{N-1} \left[ -\left( y^{2}_{1} + \cdots + y^{2}_{N} \right) + \left( z^{2}_{1} + \cdots + z^{2}_{N} \right) \right]\\
& + \lambda^{N-2} \left[ - \underset{i}{\Sigma} \underset{j > i}{\Sigma} \left( y_{i} z_{j} - y_{j} z_{i} \right)^{2} \right],
\end{aligned}
\end{equation}
Set
\begin{equation*}
\begin{aligned}
&a := 1,\\
&b := -\left( \| Z \|^{2} - \| Y \|^{2} \right),\\
&c := - \underset{i}{\Sigma} \underset{j > i}{\Sigma} \left( z_{i} y_{j} - z_{j} y_{i} \right)^{2},\\
&\Delta := b^{2} - 4ac.
\end{aligned}
\end{equation*}
Since $c \leq 0$, we have
\begin{equation}\label{Delta of ZZ - YY}
\Delta = \left( \| Z \|^{2} - \| Y \|^{2} \right)^{2} - 4 \left( - \underset{i}{\Sigma} \underset{j > i}{\Sigma} \left( z_{i} y_{j} - z_{j} y_{i} \right)^{2} \right) \geq \left( \| Z \|^{2} - \| Y \|^{2} \right)^{2} \geq 0.
\end{equation}
Thus, \eqref{characteristic polynomial of ZZ - YY} has at most two characteristic roots $\lambda_{1}, \lambda_{2} (\lambda_{1} \geq \lambda_{2})$ defined by
\begin{equation*}
\begin{aligned}
\lambda_{1,2} &= \frac{\| Z \|^{2} - \| Y \|^{2} \pm \sqrt{\left( \| Z \|^{2} - \| Y \|^{2}\right)^{2} - 4\left( - \underset{i}{\Sigma} \underset{j > i}{\Sigma} \left( z_{i} y_{j} - z_{j} y_{i} \right)^{2} \right) }}{2}\\
&= \frac{\| Z \|^{2} - \| Y \|^{2} \pm \sqrt{\left( \| Z \|^{2} + \| Y \|^{2}\right)^{2} - 4\left( Y^{T}Z \right)^{2} }}{2}.
\end{aligned}
\end{equation*} 
By using \eqref{Delta of ZZ - YY} again, we can see $\lambda_{1} \geq 0$, $\lambda_{2} \leq 0$, and $\lambda_{1} = \lambda_{2}$ if and only if $\| Z \| = \| Y \|$, which means $\lambda_{1} = 0 = \lambda_{2}$.
\end{proof}

\begin{proof}[Proof of Lemma \ref{Lemma about Hi-1}]
It is obvious that $f$ is locally concave on $E$ with $r = 0$. Assume $r \neq 0$, then we have
\begin{equation}\label{gradient of g}
\nabla_{x}\left( \frac{b - \mu^{T}x}{\sqrt{x^{T} \Sigma x}} \right) = \frac{-1}{\sqrt{x^{T} \Sigma x}}\left( \mu + \frac{\left( b - \mu^{T}x \right)}{x^{T} \Sigma x} \cdot \Sigma x \right)
\end{equation}
and
\begin{equation}\label{Hessian of g}
\begin{aligned}
\nabla^{2}_{x}\left( \frac{b - \mu^{T}x}{\sqrt{x^{T} \Sigma x}} \right) &= \frac{\left( \Sigma x \right) \cdot \mu^{T} + \mu \cdot \left( \Sigma x \right)^{T}}{\left( x^{T} \Sigma x \right)^{\frac{3}{2}}} + \frac{3\left( b - \mu^{T}x \right)}{\left( x^{T} \Sigma x \right)^{\frac{5}{2}}} \cdot \left( \Sigma x \right) \cdot \left( \Sigma x \right)^{T} -\frac{b - \mu^{T}x}{\left( x^{T} \Sigma x \right)^{\frac{3}{2}}} \cdot \Sigma.
\end{aligned}
\end{equation}
Thus, we have
\begin{equation*}
\nabla_{x}f(x) = r \cdot \left( \frac{b - \mu^{T}x}{\sqrt{x^{T} \Sigma x}} \right)^{r-1} \cdot \nabla_{x}\left( \frac{b - \mu^{T}x}{\sqrt{x^{T} \Sigma x}} \right)
\end{equation*}
and
\begin{equation*}
\nabla^{2}_{x} f(x) = r \cdot \frac{\left( b - \mu^{T}x \right)^{r}}{\left( \sqrt{x^{T} \Sigma x} \right)^{r + 2}} \cdot A,
\end{equation*}
where $A$ is defined as follows
\begin{equation}\label{A matrix}
\begin{aligned}
A &:= \frac{1}{b - \mu^{T}x} \cdot \left[ \frac{\left( b - \mu^{T}x \right)}{x^{T} \Sigma x} \cdot \left( \Sigma x \right) \cdot \left( \Sigma x \right)^{T} \cdot \left( r + 2 \right) \right. \\
&+ r \cdot \left( \mu \cdot \left( \Sigma x \right)^{T} + \left( \Sigma x \right)\cdot \mu^{T} \right) - \left( b - \mu^{T}x \right) \cdot \Sigma \\
&\left. + \left( r - 1 \right) \cdot \frac{\left( x^{T} \Sigma x \right)}{\left( b - \mu^{T}x \right)} \cdot \mu \cdot \mu^{T} \right].
\end{aligned}
\end{equation}
Since $\Sigma^{-\frac{1}{2}}$ is a symmetric and \ZHnew{positive definite} matrix, the definiteness of $\text{sign}(r) \cdot \nabla^{2}_{x}f$ is equivalent to the definiteness of $H := \Sigma^{-\frac{1}{2}} \cdot A \cdot \Sigma^{-\frac{1}{2}}$. Set $V := \Sigma^{\frac{1}{2}} \cdot x$ and $W := \Sigma^{-\frac{1}{2}}\cdot \mu$. Observe that $\theta := \left( b - \mu^{T}x \right)^{2} / \left( x^{T} \Sigma x \right)$, then we can see
\begin{equation*}
x^{T} \Sigma x = V^{T} \cdot V = \| V \|^{2}, \ \ b - \mu^{T} x = \sqrt{\theta} \cdot \sqrt{x^{T} \Sigma x} = \sqrt{\theta} \cdot \| V \|.
\end{equation*}
Thus, the matrix $H$ can be rewritten as
\begin{equation}\label{equivalent matrix to the Hessian of f}
\begin{aligned}
H &:= \frac{1}{b - \mu^{T}x} \cdot \left[ \left( r + 2 \right) \frac{b - \mu^{T}x}{x^{T} \Sigma x} \cdot \left( \Sigma^{\frac{1}{2}}x \right) \cdot \left( \Sigma^{\frac{1}{2}} x \right)^{T} \right.\\
& + r \cdot \left( \left( \Sigma^{-\frac{1}{2} }\mu \right) \cdot \left( \Sigma^{\frac{1}{2}}x \right)^{T} + \left( \Sigma^{\frac{1}{2}}x \right) \cdot \left( \Sigma^{-\frac{1}{2}}\mu \right)^{T} \right) \\
&\left. -\left( b - \mu^{T}x \right) \cdot I + \left( r -1 \right) \frac{x^{T} \Sigma x}{b - \mu^{T}x} \cdot \left( \Sigma^{-\frac{1}{2}} \mu \right) \cdot \left( \Sigma^{-\frac{1}{2}}\mu \right)^{T} \right]\\
&= \left( r + 2 \right) \cdot \frac{1}{\| V \|^{2}} \cdot V \cdot V^{T} + \frac{r}{\| V \| \cdot \sqrt{\theta}} \left[ W \cdot V^{T} + V \cdot W^{T} \right] - I\\
&+ \left( r - 1 \right) \cdot W \cdot W^{T} \cdot \frac{1}{\theta}.
\end{aligned}
\end{equation}
The assertion follows from Lemma \ref{proof of part of Lemma about Hi-1} in Appendix \ref{proof of Lemma about Hi-0}.

Finally, we consider the concavity of $\ln ((b - \mu^{T}x) / \sqrt{x^{T} \Sigma x})$. Set $g(x) := (b - \mu^{T}x) / \sqrt{x^{T} \Sigma x}$, then we have
\begin{equation}\label{gradient and Hessian of ln g}
\begin{aligned}
&\nabla_{x} \ln g(x) = \frac{1}{g(x)} \cdot \nabla_{x} g(x),\\
&\nabla^{2}_{x} \ln g(x) = -\frac{1}{g^{2}(x)} \cdot \nabla_{x}g(x) \cdot \left( \nabla_{x} g(x) \right)^{T} + \frac{1}{g(x)} \cdot \nabla^{2}_{x} g(x).
\end{aligned}
\end{equation}
Taking \eqref{gradient of g} and \eqref{Hessian of g} into \eqref{gradient and Hessian of ln g}, we can yield
\begin{equation*}
\nabla^{2}_{x} \ln g(x) = \frac{1}{x^{T} \Sigma x} \cdot A,
\end{equation*}
where the matrix $A$ is simply defined by \eqref{A matrix}. Consequently, in order to get \eqref{sufficient and necessary condition ln}, it suffices to take $r = 0$ and repeat the process of the case with $r \in (-2, 2) \setminus \lbrace 0 \rbrace$ literally.
\end{proof}

\begin{lemma}\label{proof of part of Lemma about Hi-1}
Let $r \in \mathbb{R}\setminus \lbrace 0 \rbrace$, $b \in \mathbb{R}$, $\mu \in \mathbb{R}^{N}$. Assume $\Sigma$ is a positive definite covariance matrix. Set $E := \lbrace x \in \mathbb{R}^{N} \setminus \lbrace o \rbrace : b - \mu^{T}x > 0 \rbrace$. Define $\theta := \left( b - \mu^{T}x \right)^{2} / \left( {x^{T} \Sigma x} \right)$. Let $V := \Sigma^{\frac{1}{2}} \cdot x$ and $W := \Sigma^{-\frac{1}{2}}\cdot \mu$. Define a matrix $H$ as follows
\begin{equation}\label{equivalent matrix to the Hessian of f Appendix}
\begin{aligned}
H &:=  \frac{\left( r + 2 \right)}{\| V \|^{2}} \cdot V \cdot V^{T} + \frac{r}{\| V \| \cdot \sqrt{\theta}} \left[ W \cdot V^{T} + V \cdot W^{T} \right] - I + \frac{\left( r - 1 \right)}{\theta} \cdot W \cdot W^{T} .
\end{aligned}
\end{equation}
Then, the matrix $H$ is negative semi-definite on $E$ if and only if :
\begin{equation}\label{sufficient and necessary condition r<2 and r neq 0 Appendix}
\mu^{T}\Sigma^{-1}\mu \leq (2-r)\frac{(\mu^{T}x)^{2}}{x^{T}\Sigma x} - 2r\sqrt{\theta} \frac{\mu^{T}x}{\sqrt{x^{T}\Sigma x}} - (r + 1)\theta.
\end{equation}
\end{lemma}
\begin{proof}
Suppose $r > 2$ and define
\begin{equation}\label{transformation 1 of Lemma about Hi-1}
Z := \frac{r}{\sqrt{r + 2}} \cdot \frac{1}{\sqrt{\theta}} \cdot W + \frac{\sqrt{r + 2}}{\| V \| } \cdot V, \ \ Y := \frac{1}{\sqrt{\theta}} \cdot \frac{\sqrt{r - 2}}{\sqrt{r + 2}} \cdot W.
\end{equation}
Then, we can reformulate \eqref{equivalent matrix to the Hessian of f Appendix} as follows
\begin{equation*}
H = Z \cdot Z^{T} + Y \cdot Y^{T} - I.
\end{equation*}
Using Lemma \ref{Lemma about Hi-0}, we know that the matrix $Z \cdot Z^{T} + Y \cdot Y^{T}$ has at most two positive eigenvalues $\lambda_{1}$, $\lambda_{2}(\lambda_{1} \geq \lambda_{2})$ defined by \eqref{eigenvalue 1-1}. As a result, the matrix $H$ is negative semi-definite if and only if $\lambda_{1} \leq 1$, which is equivalent to
\begin{equation*}
\| Y \|^{2} + \| Z \|^{2} + \sqrt{\left( \| Y \|^{2} - \| Z \|^{2} \right)^{2} + 4\left( Y^{T} Z \right)^{2}} \leq 2,
\end{equation*}
that is
\begin{equation}\label{equivalence with negative semi-definite H 1}
\| Y \|^{2} + \| Z \|^{2} - \| Y \|^{2} \cdot \| Z \|^{2} + \left( Y^{T} Z \right)^{2} \leq 1.
\end{equation}
According to \eqref{transformation 1 of Lemma about Hi-1}, we have
\begin{equation}\label{|Z| of the transformation 1 of Lemma about Hi-1}
\begin{aligned}
\| Z \|^{2} = Z^{T} \cdot Z & = \frac{r^{2}}{r + 2}\cdot \frac{1}{\theta} \cdot W^{T} \cdot W + (r + 2) + \frac{r}{\sqrt{\theta}} \cdot \frac{1}{\| V \|} \cdot \left[ W^{T} \cdot V + V^{T} \cdot W \right]\\
& = \frac{r^{2}}{r + 2} \cdot \frac{1}{\theta} \cdot \mu^{T} \Sigma^{-1} \mu + (r + 2) + \frac{r}{\sqrt{\theta}} \cdot \frac{1}{\sqrt{x^{T} \Sigma x}} \cdot 2\mu^{T}x,
\end{aligned}
\end{equation}
\begin{equation}\label{|Y| of the transformation 1 of Lemma about Hi-1}
\| Y \|^{2} = Y^{T} \cdot Y = \frac{1}{\theta} \cdot \frac{r - 2}{r + 2} \cdot W^{T} \cdot W = \frac{1}{\theta} \cdot \frac{r - 2}{r + 2}\cdot \mu^{T}\Sigma^{-1}\mu,
\end{equation}
\begin{equation}\label{YZ of the transformation 1 of Lemma about Hi-1}
\begin{aligned}
Y^{T}Z &= \frac{1}{\theta} \cdot \frac{r\sqrt{r - 2}}{r + 2} \cdot W^{T} \cdot W + \frac{\sqrt{r -2}}{\sqrt{\theta}} \cdot \frac{1}{\| V \|} \cdot W^{T} \cdot V\\
&= \frac{1}{\theta} \cdot \frac{r\sqrt{r - 2}}{r + 2} \cdot \mu^{T}\Sigma^{-1}\mu + \frac{\sqrt{r - 2}}{\sqrt{\theta}} \cdot \frac{1}{\sqrt{x^{T} \Sigma x}} \cdot \mu^{T}x.
\end{aligned}
\end{equation}
By taking \eqref{|Z| of the transformation 1 of Lemma about Hi-1}-\eqref{YZ of the transformation 1 of Lemma about Hi-1} into \eqref{equivalence with negative semi-definite H 1}, we get
\begin{equation*}
\mu^{T} \Sigma^{-1} \mu \leq (2 - r) \cdot \frac{\left( \mu^{T}x \right)^{2}}{x^{T} \Sigma x} - 2r \cdot \sqrt{\theta} \frac{\mu^{T}x}{\sqrt{x^{T}\Sigma x}} -(r + 1)\cdot \theta.
\end{equation*}
If $r = 2$, then we can use \eqref{transformation 1 of Lemma about Hi-1} to rewrite \eqref{equivalent matrix to the Hessian of f Appendix} as follows
\begin{equation*}
H = Z \cdot Z^{T} - I.
\end{equation*}
The second result of Lemma \ref{Lemma about Hi-0} tell us that $Z \cdot Z^{T}$ has at most one positive eigenvalue $\lambda_{1} := \| Z \|^{2}$. Thus, the matrix $H$ is negative semi-definite if and only if $\| Z \|^{2} \leq 1$, which means
\begin{equation*}
\frac{1}{\theta} \cdot \mu^{T} \Sigma^{-1} \mu + 4 + \frac{4}{\sqrt{\theta}} \cdot \frac{1}{\sqrt{x^{T} \Sigma x}} \cdot \mu^{T}x \leq 1,
\end{equation*}
it is simply \eqref{sufficient and necessary condition r<2 and r neq 0 Appendix} with $r = 2$. In conclusion, we have proved that $\text{sign}(-r) \cdot f$ is locally concave on $E$ with $r \geq 2$.

Next, we consider the case with $r < 2$ and $r \neq 0$. If $-2 < r$, then \eqref{equivalent matrix to the Hessian of f Appendix} can be reformulated as follows
\begin{equation*}
H = Z \cdot Z^{T} - Y \cdot Y^{T} - I,
\end{equation*}
where
\begin{equation*}
Z := \frac{r}{\sqrt{r + 2}} \cdot \frac{1}{\sqrt{\theta}}W + \frac{\sqrt{r + 2}}{\| V \|} \cdot V, \ \ Y := \frac{1}{\sqrt{\theta}} \cdot \frac{\sqrt{2 - r}}{\sqrt{r + 2}}W.
\end{equation*}
Using the third result of Lemma \ref{Lemma about Hi-0}, we know that $Z \cdot Z^{T} - Y \cdot Y^{T}$ has at most one positive eigenvalue $\lambda_{1}$ defined by \eqref{eigenvalue 3-1}. As a result, the negative semi-definiteness of the $H$ is equivalent to $\lambda_{1} \leq 1$, which can be simplified as follows
\begin{equation}\label{equivalence with negative semi-definite H 3}
\| Z \|^{2} - \| Y \|^{2} + \| Z \|^{2} \cdot \| Y \|^{2} - \left( Y^{T} Z \right)^{2} \leq 1.
\end{equation}
Since we have
\begin{equation*}
\begin{aligned}
\| Z \|^{2} = Z^{T} \cdot Z &= \frac{r^{2}}{r + 2} \cdot \frac{1}{\theta} \cdot W^{T}W + (r + 2) + \frac{r}{\sqrt{\theta}} \cdot \frac{1}{\| V \|} \cdot \left[ W^{T}V + V^{T}W \right]\\
& = \frac{r^{2}}{r + 2} \cdot \frac{1}{\theta} \cdot \mu^{T} \Sigma^{-1} \mu + (r + 2) + \frac{r}{\sqrt{\theta}} \cdot \frac{1}{\sqrt{x^{T} \Sigma x}} \cdot 2\mu^{T}x,
\end{aligned}
\end{equation*}
\begin{equation*}
\| Y \|^{2} = \frac{1}{\theta} \cdot \frac{2 - r}{r + 2} \cdot W^{T}W = \frac{1}{\theta} \cdot \frac{2 - r}{r + 2} \cdot \mu^{T}\Sigma^{-1}\mu,
\end{equation*}
\begin{equation*}
\begin{aligned}
Y^{T}Z &= \frac{1}{\theta} \cdot \frac{r\sqrt{2 - r}}{r + 2} \cdot W^{T}W + \frac{\sqrt{2 - r}}{\sqrt{\theta}} \cdot \frac{1}{\| V \|} \cdot W^{T}V\\
&= \frac{1}{\theta} \cdot \frac{r\sqrt{2 - r}}{r + 2} \cdot  \mu^{T} \Sigma^{-1} \mu + \frac{\sqrt{2 - r}}{\sqrt{\theta}} \cdot \frac{1}{\sqrt{x^{T} \Sigma x}} \cdot \mu^{T}x,
\end{aligned}
\end{equation*}
then \eqref{equivalence with negative semi-definite H 3} can be rewritten as
\begin{equation*}
\mu^{T} \Sigma^{-1} \mu \leq \left( 2 - r \right)\cdot \frac{\left( \mu^{T}x \right)^{2} }{x^{T} \Sigma x} - 2r \cdot \sqrt{\theta} \cdot \frac{\mu^{T}x}{\sqrt{x^{T} \Sigma x}} - \left( r + 1 \right) \cdot \theta.
\end{equation*}
Similarly, if $r < -2$, then \eqref{equivalent matrix to the Hessian of f Appendix} can be reformulated as follows
\begin{equation*}
H = Y \cdot Y^{T} - Z \cdot Z^{T} - I,
\end{equation*}
where
\begin{equation*}
Z := \frac{\sqrt{-r - 2}}{\| V \|} \cdot V - \frac{r}{\sqrt{-r - 2}} \cdot \frac{1}{\sqrt{\theta}} \cdot W,\ \ Y := \frac{\sqrt{2 - r}}{\sqrt{-r - 2}} \cdot \frac{1}{\sqrt{\theta}} \cdot W.
\end{equation*}
Using the third result of Lemma \ref{Lemma about Hi-0} again, we can also see that $H$ is a negative semi-definite matrix if and only if \eqref{sufficient and necessary condition r<2 and r neq 0 Appendix} is satisfied. If $r = -2$, then we can rewrite \eqref{equivalent matrix to the Hessian of f Appendix} as follows
\begin{equation*}
H = Y \cdot Y^{T} - Z \cdot Z^{T} - I,
\end{equation*}
where
\begin{equation*}
Z := \frac{\sqrt{3}}{\sqrt{\theta}} \cdot W + \frac{2}{\sqrt{3}} \cdot \frac{1}{\| V \|} \cdot V, \ \ Y := \frac{2}{\sqrt{3}} \cdot \frac{1}{\| V \|} \cdot V.
\end{equation*}
Since
\begin{equation*}
\| Z \|^{2} = Z^{T} Z = \frac{3}{\theta} \cdot \mu^{T} \Sigma^{-1} \mu + \frac{4}{3} + \frac{4}{\sqrt{\theta}} \cdot \frac{\mu^{T}x}{\sqrt{x^{T} \Sigma x}},
\end{equation*}
\begin{equation*}
\| Y \|^{2} = Y^{T} Y = \frac{4}{3},
\end{equation*}
\begin{equation*}
Y^{T}Z = \frac{2}{\sqrt{\theta}} \cdot \frac{\mu^{T}x}{\sqrt{x^{T} \Sigma x}} + \frac{4}{3}.
\end{equation*}
Then, applying \eqref{eigenvalue 3-1}, we can see that $H$ is negative semi-definite equivalent to
\begin{equation*}
\mu^{T} \Sigma^{-1} \mu \leq 4 \cdot \frac{\left( \mu^{T}x \right)^{2}}{x^{T} \Sigma x} + 4\sqrt{\theta} \cdot \frac{\mu^{T}x}{\sqrt{x^{T} \Sigma} x} + \theta = \left( 2\cdot \frac{\mu^{T}x}{\sqrt{x^{T} \Sigma x}} + \sqrt{\theta} \right)^{2},
\end{equation*}
which is \eqref{sufficient and necessary condition r<2 and r neq 0 Appendix} with $r = -2$.
\end{proof}

\begin{proof}[Proof of Lemma \ref{reformualtion of the necessary and sufficient conditions again}]
Observe that
\begin{equation*}
\frac{b}{\sqrt{x^{T} \Sigma x}} - \sqrt{\theta} = \frac{\mu^{T} x}{\sqrt{x^{T} \Sigma x}}.
\end{equation*}
Then, we can reformulate \eqref{sufficient and necessary condition r<2 and r neq 0} as follows
\begin{equation}\label{1 transformation of key sufficient and necessary conndition inequality Appendix}
\mu^{T} \Sigma^{-1} \mu \leq (2 - r)\cdot \frac{b^{2}}{x^{T} \Sigma x} - 4\sqrt{\theta} \cdot \frac{b}{\sqrt{x^{T} \Sigma x}} + \theta.
\end{equation}
Let $x \in E\cap X \setminus \lbrace o \rbrace$ and $\beta$ be the angle between vectors $\mu$ and $x$. Then, we have $\mu^{T}x = \| \mu \| \cdot \| x \| \cdot \text{cos}\beta$. Set $\lambda_{\text{min}}$, $\lambda_{\text{max}}$ as the minimum and maximum eigenvalues of positive definite matrix $\Sigma$, then we can yield
\begin{equation*}
0 < \lambda_{\text{min}} \| x \|^{2} \leq x^{T} \Sigma x \leq \lambda_{\text{max}} \| x \|^{2}.
\end{equation*}
Applying the Mean Value Theorem, we can know that \ZH{there exists} a $\lambda(x) \in [\lambda_{\text{min}}, \lambda_{\text{max}}]$ such that $x^{T} \Sigma x = \lambda(x) \cdot \| x \|^{2}$. Thus, \eqref{1 transformation of key sufficient and necessary conndition inequality Appendix} can be rewritten as follows
\begin{equation*}
\begin{aligned}
\mu^{T} \Sigma^{-1} \mu &\leq (2 - r) \cdot \frac{b^{2}}{x^{T} \Sigma x} - 4b \cdot \frac{b - \mu^{T}x}{x^{T} \Sigma x} + \frac{\left( b - \mu^{T}x \right)^{2}}{x^{T} \Sigma x}\\
&= \frac{\left( \mu^{T}x \right)^{2}}{x^{T} \Sigma x} + \frac{2b\mu^{T}x}{x^{T} \Sigma x} + \frac{\left( -1 - r \right)b^{2}}{x^{T} \Sigma x}\\
&= \frac{\| \mu \|^{2} \cdot \| x \|^{2} \cdot \text{cos}^{2} \beta}{\lambda(x) \cdot \| x \|^{2}} + \frac{2b \cdot \| \mu \| \cdot \| x \| \cdot \text{cos}\beta}{\lambda(x) \cdot \| x \|^{2}} + \frac{\left( -1 - r \right)b^{2}}{\lambda(x) \cdot \| x \|^{2}},
\end{aligned}
\end{equation*}
that is
\begin{equation}\label{2 transformation of key sufficient and necessary conndition inequality Appendix}
\mu^{T} \Sigma^{-1} \mu \leq \frac{\| \mu \|^{2} \cdot \text{cos}^{2} \beta}{\lambda(x)} + \frac{2b \cdot \| \mu \| \cdot \text{cos}\beta}{\lambda(x)} \cdot \frac{1}{\| x \|} + \frac{\left( -1 - r \right)b^{2}}{\lambda(x)} \cdot \frac{1}{\| x \|^{2}}.
\end{equation}

Set $t := 1/ \| x \|$ and $s := \text{cos}\beta$. It follows that the right hand side of \eqref{2 transformation of key sufficient and necessary conndition inequality Appendix} can be considered as a function $h(t, s) : (0, +\infty) \times [-1 , 1] \rightarrow \mathbb{R}$ defined by
\begin{equation}\label{h function's transformation with t and s Appendix}
h(t, s) := \frac{\| \mu \|^{2}}{\lambda} \cdot s^{2} + 2 \cdot \frac{b \cdot \| \mu \|}{\lambda} \cdot t \cdot s + \frac{(-1 - r) \cdot b^{2}}{\lambda} \cdot t^{2},
\end{equation}
where $\lambda \in [\lambda_{\text{min}}, \lambda_{\text{max}}]$ depending on $s$ and $t$; $\lambda_{\text{min}}$, $\lambda_{\text{max}}$ are the minimum and maximum eigenvalues of the matrix $\Sigma$ respectively. In fact, the $\lambda$ is defined by $\lambda(x) := x^{T} \Sigma x / \| x \|^{2}$. We reformulate the $g$ function as follows
\begin{equation*}
g(\frac{1}{\| x \|}, \text{cos} \beta) = \frac{b}{\sqrt{\lambda} \cdot \| x \|} - \frac{\| \mu \|}{\sqrt{\lambda}} \cdot \text{cos} \beta,
\end{equation*} 
where $\lambda$ depends on $(1/\| x \|, \text{cos}\beta)$. And then $g$ can be redefined as follows
\begin{equation}\label{g function's transformation with t and s Appendix}
g(t, s) = \frac{b}{\sqrt{\lambda}} \cdot t - \frac{\| \mu \|}{\sqrt{\lambda}} \cdot s.
\end{equation}
It can be seen that a pair $(t, s) \in (0, +\infty) \times [-1, 1]$ corresponds to vectors $x(t, s) \in \mathbb{R}^{N}$. We will use $h$ and $g$ to get the best threshold $\theta^{*}$. To this end, we next reformulate $h$ and $g$ into more concise forms. Define $\rho: \mathbb{R}^{N}\setminus \lbrace 0 \rbrace \to \mathbb{R}^{N}$ as follows
\begin{equation}\label{the definition of rho}
\rho(x) := \frac{x}{\| x \|} \cdot \frac{\|x \|}{\sqrt{x^{T} \Sigma x}} = \frac{x}{\| x \|} \cdot \frac{1}{\sqrt{\lambda(x)}}.
\end{equation}
It can be noticed that $\rho$ defined on $\mathbb{R}^{N} \setminus \lbrace 0 \rbrace$ is positively homogeneous of degree zero. Thus, it suffices to consider $\rho$ on $\mathbb{S}^{N-1}$. Further, we define $M$ by
\begin{equation*}
M := \left\lbrace \frac{x}{\sqrt{\lambda(x)}} : x \in \mathbb{S}^{N-1} \right\rbrace,
\end{equation*}
then, the map $\rho : \mathbb{S}^{N-1} \to M$ is the reverse Gauss map of $M$ and $M$ is a star-shape w.r.t. the origin center. Given $\mu \neq 0$, we define $\phi_{\mu} : \mathbb{R}^{N}\setminus \lbrace 0 \rbrace \to \mathbb{R}$ by
\begin{equation*}
\phi_{\mu}(x) := \rho(x) \cdot \mu / \| \mu \|.
\end{equation*} 
The function $\phi_{\mu}$ is an odd function on $\mathbb{R}^{N}\setminus \lbrace 0 \rbrace$ and is also positively homogeneous of degree zero. Thus, we will consider $\phi_{\mu}$ on $\mathbb{S}^{N-1}$. From the definition of $\phi_{\mu}$, we can consider $\phi_{\mu}$ as the projection of $\rho$ onto the given unit vector $\mu/ \| \mu \|$. Since $\phi_{\mu}$ is continuous on $\mathbb{S}^{N-1}$, and $\mathbb{S}^{N-1}$ is a bounded and closed set. Thus, \ZH{there exists} $x_{1} \in \mathbb{S}^{N-1}$ such that
\begin{equation*}
\phi_{\mu}(x_{1}) = \max_{\mathbb{S}^{N-1}} \phi_{\mu}(x) =: \frac{1}{\sqrt{\lambda_{\mu, \text{min}}}}
\end{equation*}
Because $\phi_{\mu}$ is odd on $\mathbb{S}^{N-1}$. Then, we can yield
\begin{equation*}
\min_{\mathbb{S}^{N-1}} \phi_{\mu}(x) = -\max_{\mathbb{S}^{N-1}} \phi_{\mu}(x) = -\frac{1}{\sqrt{\lambda_{\mu, \text{min}}}} = \phi_{\mu}(-x_{1}).
\end{equation*}
Thus, the range of $\phi_{\mu}$ is $\left[ -1/\sqrt{\lambda_{\mu, \text{min}}}, 1/\sqrt{\lambda_{\mu, \text{min}}} \right]$. Let $x_{\text{min}} \in \mathbb{S}^{N-1}$ be the eigenvector of $\Sigma$ corresponding to the minimum eigenvalue $\lambda_{\text{min}}$. Using \eqref{the definition of rho}, we can see
\begin{equation*}
\| \rho(x) \| = \frac{1}{\sqrt{\lambda(x)}} \leq \frac{1}{\sqrt{\lambda_{\text{min}}}}, \forall \  x\in \mathbb{S}^{N-1}.
\end{equation*}
Because $\phi_{\mu}$ is the projection of $\rho$ onto $\mu / \| \mu \|$. Thus, we can yield
\begin{equation*}
\frac{1}{\sqrt{\lambda_{\mu, \text{min}}}} = \max_{\mathbb{S}^{N-1}} \phi_{\mu}(x) = \phi_{\mu}(x_{\text{min}}) = \frac{1}{\sqrt{\lambda_{\text{min}}}}, \ \text{if} \ \mu / \| \mu \| = \pm x_{\text{min}},
\end{equation*}
\begin{equation*}
\frac{1}{\sqrt{\lambda_{\mu, \text{min}}}} < \frac{1}{\sqrt{\lambda_{\text{min}}}}, \ \text{if} \ \mu / \| \mu \| \neq \pm x_{\text{min}},
\end{equation*}
\begin{equation*}
\frac{1}{\sqrt{\lambda_{\text{max}}}} \leq \min_{\mu \in \mathbb{R}^{N} \setminus \lbrace 0 \rbrace} \max_{x \in \mathbb{S}^{N-1}} \frac{\mu \cdot x}{\| \mu \| \cdot \| x \|} \cdot \frac{\| x \|}{\sqrt{x^{T} \Sigma x}} = \min_{\mu \in \mathbb{R}^{N} \setminus \lbrace 0 \rbrace} \frac{1}{\sqrt{\lambda_{\mu, \text{min}}}},
\end{equation*}
where the first inequality in the above last formula is yielded by observing $\max_{x} \phi_{\mu}(x)$ is the distance between the origin and the support hyperplane of $M$ with $\mu/ \| \mu \|$ as its normal.

As a result, we have proved
\begin{equation*}
\lambda_{\text{min}} \leq \lambda_{\mu, \text{min}} \leq \lambda_{\text{max}}.
\end{equation*}
On the one hand, we claim that for any $c_{1} \in (0, +\infty)$ and $c_{2} \in [-1/ \sqrt{\lambda_{\mu, \text{min}}}, 1/\sqrt{\lambda_{\mu, \text{min}}}]$, there exist $t_{0} \in (0, +\infty)$, $s_{0} \in [-1, 1]$, $\lambda_{0} := \lambda(x(t_{0}, s_{0})) \in [\lambda_{\text{min}}, \lambda_{\text{max}}]$ such that
\begin{equation}\label{transform of t and s Appendix}
\frac{t_{0}}{\sqrt{\lambda_{0}}} = c_{1}, \ \ \frac{s_{0}}{\sqrt{\lambda_{0}}} = c_{2}.
\end{equation} 
Since $\phi_{\mu}$ is continuous on $\mathbb{S}^{N-1}$ and the range of $\phi_{\mu}$ is $\left[ -1/\sqrt{\lambda_{\mu, \text{min}}}, 1/\sqrt{\lambda_{\mu, \text{min}}} \right]$. Thus, for each $c_{2} \in \left[ -1/\sqrt{\lambda_{\mu, \text{min}}}, 1/\sqrt{\lambda_{\mu, \text{min}}} \right]$, there exists $x_{0} \in \mathbb{S}^{N-1}$ such that $\phi_{\mu}(x_{0}) = c_{2}$, i.e.
\begin{equation}\label{homogeneous 1}
\frac{\mu \cdot x_{0}}{\| \mu \| \cdot \| x_{0} \|} \cdot \frac{\| x_{0} \|}{\sqrt{x^{T}_{0} \Sigma x_{0}}} = c_{2}.
\end{equation}
For each $c_{1} \in (0, +\infty)$, we define $\hat{x_{0}}$ by
\begin{equation}\label{homogeneous 2}
\hat{x_{0}} := \frac{x_{0}}{\| x_{0} \|} \cdot \frac{1}{c_{1} \cdot \sqrt{\lambda(x_{0})}}.
\end{equation}
Since $\phi_{\mu}$ is positively homogeneous of degree zero. Combining \eqref{homogeneous 1} and \eqref{homogeneous 2}, we can see
\begin{equation*}
\phi_{\mu}(\hat{x_{0}}) = \phi_{\mu}(x_{0}) = c_{2}.
\end{equation*}
We take
\begin{equation*}
(t_{0}, s_{0}, \lambda_{0}) := (\frac{1}{\| \hat{x_{0}} \|}, \frac{\mu \cdot \hat{x_{0}}}{\| \mu \| \cdot \| \hat{x_{0}}\|}, \frac{\hat{x_{0}}^{T} \Sigma \hat{x_{0}}}{\| \hat{x_{0}} \|^{2}}).
\end{equation*}
Since $\lambda(x(t, s)) = x^{T} \Sigma x / \| x \|^{2}$, we can see that $(t_{0}, s_{0}, \lambda_{0})$ satisfies \eqref{transform of t and s Appendix}. On the other hand, the definitions of $t$, $s$, $\lambda$ imply $t / \sqrt{\lambda} \in (0, +\infty)$ and $s / \sqrt{\lambda} \in [-1/ \sqrt{\lambda_{\mu, \text{min}}}, 1/ \sqrt{\lambda_{\mu, \text{min}}}]$, where we have used
\begin{equation*}
\frac{s}{\sqrt{\lambda(x(t,s))}} = \frac{\mu \cdot x}{\| \mu \| \cdot \| x \|} \cdot \frac{\| x \|}{\sqrt{x^{T}\Sigma x}} = \phi_{\mu}(x) \in \left[ -\frac{1}{\sqrt{\lambda_{\mu, \text{min}}}}, \frac{1}{\sqrt{\lambda_{\mu, \text{min}}}} \right].
\end{equation*}
As a result, by taking $\overline{t} := t / \sqrt{\lambda}$ and $\overline{s} := s / \sqrt{\lambda}$, we can reformulate \eqref{h function's transformation with t and s Appendix} and \eqref{g function's transformation with t and s Appendix} as follows
\begin{equation*}
h(\overline{t}, \overline{s}) = \| \mu \|^{2} \cdot \overline{s}^{2} + 2b \cdot \| \mu \| \cdot \overline{t} \cdot \overline{s} + (-1 - r) \cdot b^{2} \cdot \overline{t}^{2},
\end{equation*}
\begin{equation*}
g(\overline{t}, \overline{s}) = b \cdot \overline{t} - \| \mu \| \cdot \overline{s},
\end{equation*}
where $\overline{t} \in (0, +\infty)$ and $\overline{s} \in [-1/ \sqrt{\lambda_{\mu, \text{min}}}, 1 / \sqrt{\lambda_{\mu, \text{min}}}]$. 
\end{proof}

\begin{proof}[Proof of Lemma \ref{tangent relationship between s1 and s2}]
Because $h(\overline{t}, \overline{s})$ is an elliptic paraboloid, at least one of the Case (1), Case (2), and Case (3) can be satisfied. Next, we compute $\sqrt{c^{*}}$ in the three cases. For Case (2) and Case (3), taking $\overline{s} = -1/ \sqrt{\lambda_{\mu, \text{min}}}$ into $h(\overline{t}, \overline{s}) = \mu^{T} \Sigma^{-1} \mu$, where $h(\overline{t}, \overline{s})$ is defined by \eqref{h function's transformation with overline t and overline s} and $\lambda_{\mu, \text{min}}$ is defined by \eqref{the definition of lamda_mu_min}, we can yield
\begin{equation}\label{case 2 and case 3 equation 1 Appendix}
\mu^{T} \Sigma^{-1} \mu = \| \mu \|^{2} \cdot \frac{1}{\lambda_{\mu, \text{min}}} - 2b \cdot \| \mu \| \cdot \overline{t} \cdot \frac{1}{\sqrt{\lambda_{\mu, \text{min}}}} + \left( -1 -r \right) \cdot b^{2} \cdot \overline{t}^{2}.
\end{equation}
Set $\Delta$ as follows:
\begin{equation}\label{case 2 and case 3 equation 2 Appendix}
\Delta := \left( 2 + r \right) \| \mu \|^{2} \frac{1}{\lambda_{\mu, \text{min}}} + \left( -1 - r \right)\mu^{T} \Sigma^{-1} \mu.
\end{equation}
If $\Delta < 0$, then there does not exist solutions for \eqref{case 2 and case 3 equation 1 Appendix}, which corresponds to the Case (3). Since $r < -1$ and $\| \mu \| > 0$, based on \eqref{case 2 and case 3 equation 2 Appendix} we can see that $\Delta < 0$ is equivalent to
\begin{equation}\label{case 2 and case 3 equation 2 Appendix-1}
\frac{1}{1+r} > \frac{\mu^{T} \Sigma^{-1} \mu}{\| \mu \|^{2}} \cdot \lambda_{\mu, \text{min}} - 1,
\end{equation}
which implies $r^{o}(\mu, \Sigma) := \| \mu \|^{2} / \left( \mu^{T} \Sigma^{-1} \mu \cdot \lambda_{\mu, \text{min}} \right) > 1$. Thus, when $r^{o}(\mu, \Sigma) \leq 1$, the Case(3) can not happen. Further, we can rewrite \eqref{case 2 and case 3 equation 2 Appendix-1} as follows
\begin{equation}\label{case 2 and case 3 equation 3  Appendix}
r < \frac{1}{ \frac{1}{r^{o}(\mu, \Sigma)} - 1 } - 1 =: r^{**}(\mu, \Sigma) \ \ \text{and} \ \ r^{**}(\mu, \Sigma) < + \infty.
\end{equation}
As a result, once \eqref{case 2 and case 3 equation 3 Appendix} is satisfied, we can take $\sqrt{\theta^{*}} = \|\mu \| / \sqrt{\lambda_{\mu, \text{min}}}$. It can be noticed that if $\Sigma = a \cdot I$ with $a > 0$, then \eqref{case 2 and case 3 equation 3 Appendix} will not be satisfied. If $\Delta \geq 0$, then there exist solutions for \eqref{case 2 and case 3 equation 1 Appendix} and the greater one is defined by
\begin{equation*}
\overline{t}^{*} := \frac{\| \mu \| \cdot \frac{1}{\sqrt{\lambda_{\mu, \text{min}}}} + \sqrt{\left( 2 + r \right) \| \mu \|^{2} \cdot \frac{1}{\lambda_{\mu, \text{min}}} + \left( -1 - r \right) \cdot \mu^{T} \Sigma^{-1} \mu }}{\left( -1 - r \right) \cdot b}.
\end{equation*}
By substituting $(\overline{t}, \overline{s}) := (\overline{t}^{*}, -1/\sqrt{\lambda_{\mu, \text{min}}})$ into $g(\overline{t}, \overline{s}) = c$, we can obtain
\begin{equation}\label{case 2 and case 3 equation 4 Appendix}
\begin{aligned}
c^{*} &:= g(\overline{t}^{*}, -1/\sqrt{\lambda_{\mu,\text{min}}})\\
&= \frac{\left( -r \right) \frac{\| \mu \|}{\sqrt{\lambda_{\mu, \text{min}}}} + \sqrt{\left( 2 + r \right) \frac{\| \mu \|^{2}}{\lambda_{\mu, \text{min}}} + \left( -1 - r \right) \mu^{T} \Sigma^{-1} \mu}}{-1 - r}.
\end{aligned}
\end{equation}
As a result, when the Case (2) takes place, we can take $\sqrt{\theta^{*}} = c^{*}$, where $c^{*}$ is defined by \eqref{case 2 and case 3 equation 4 Appendix}. It remains to consider the Case (1). Since the slope of $\overline{s}_{1}$ is $b / \| \mu \| > 0$, we can apply the Implicit Function theorem and use \eqref{h function's transformation with overline t and overline s} to define $\overline{s}^{'}_{2}(\overline{t})$ near to the tangent point as follows:
\begin{equation*}
\overline{s}^{'}_{2}(\overline{t}) = - \frac{h^{'}_{\overline{t}}\left( \overline{t}, \overline{s} \right)}{h^{'}_{\overline{s}}\left( \overline{t}, \overline{s} \right)} = -\frac{b \| \mu \| \cdot \overline{s} + \left( -1 - r \right) b^{2} \cdot  \overline{t}}{\| \mu \|^{2} \cdot \overline{s} + b \| \mu \| \cdot \overline{t}}.
\end{equation*}
Set the tangent point is $(\overline{t}^{*}, \overline{s}^{*})$, then we have $\overline{s}^{'}_{1}(\overline{t}^{*}) = \overline{s}^{'}_{2}(\overline{t}^{*})$, which means
\begin{equation*}
 -\frac{b \| \mu \| \cdot \overline{s} + \left( -1 - r \right) b^{2} \cdot  \overline{t}}{\| \mu \|^{2} \cdot \overline{s} + b \| \mu \| \cdot \overline{t}} = \frac{b}{\| \mu \|},
\end{equation*}
that is
\begin{equation}\label{case 2 and case 3 equation 5 Appendix}
\overline{s}^{*} = \frac{r}{2} \cdot \frac{b}{\| \mu \|} \cdot \overline{t}^{*},
\end{equation}
where we have used the fact $h^{'}_{\overline{s}}\left( \overline{t}^{*}, \overline{s}^{*} \right) \neq 0$ since $\overline{s}_{1}$ is not vertical. Substituting  \eqref{case 2 and case 3 equation 5 Appendix} into $h(\overline{t}, \overline{s}) = \mu^{T} \Sigma^{-1} \mu$ we get
\begin{equation}\label{case 2 and case 3 equation 6 Appendix}
\left( \frac{1}{4} r^{2} - 1 \right) \cdot b^{2} \cdot \overline{t}^{2} = \mu^{T} \Sigma^{-1} \mu.
\end{equation}
It can be seen from \eqref{case 2 and case 3 equation 6 Appendix} that when $r \in [-2, -1)$, there does not exist the tangent point so that the Case (1) degenerates into the Case (2) or Case (3). Thus, when $r < -2$, using \eqref{case 2 and case 3 equation 6 Appendix} we can yield $\overline{t}^{*}$ defined as follows:
\begin{equation}\label{case 2 and case 3 equation 7 Appendix}
\overline{t}^{*} := \frac{1}{b} \cdot \sqrt{\mu^{T} \Sigma^{-1} \mu} \cdot \frac{1}{\sqrt{\frac{1}{4} r^{2} - 1}}.
\end{equation}
Substituting \eqref{case 2 and case 3 equation 7 Appendix} into \eqref{case 2 and case 3 equation 5 Appendix} we yield
\begin{equation}\label{case 2 and case 3 equation 8 Appendix}
\overline{s}^{*} = \frac{r}{2} \cdot \frac{1}{\| \mu \|} \cdot \sqrt{\mu^{T} \Sigma^{-1} \mu} \cdot \frac{1}{\sqrt{\frac{1}{4} r^{2} - 1}}.
\end{equation}
Thus, by using \eqref{case 2 and case 3 equation 7 Appendix} and \eqref{case 2 and case 3 equation 8 Appendix} we can obtain
\begin{equation}\label{case 2 and case 3 equation 9 Appendix}
c^{*} = g\left( \overline{t}^{*} , \overline{s}^{*} \right) = b \cdot \overline{t}^{*} - \| \mu \| \cdot \overline{s}^{*} = \sqrt{\frac{-r + 2}{-r - 2}} \cdot \sqrt{\mu^{T} \Sigma^{-1} \mu}.
\end{equation}
Next, we check whether $\overline{s}^{*}$ defined by \eqref{case 2 and case 3 equation 8 Appendix} is contained in $[-1/\sqrt{\lambda_{\mu, \text{min}}}, 1/ \sqrt{\lambda_{\mu, \text{min}}}]$. 
Since $r < -2$ and \eqref{case 2 and case 3 equation 8 Appendix} is satisfied, then we can see that $\overline{s}^{*} < 0 < 1/ \sqrt{\lambda_{\mu, \text{min}}}$, and $\overline{s}^{*} \geq -1/ \sqrt{\lambda_{\mu, \text{min}}}$ is equivalent to
\begin{equation}\label{case 2 and case 3 equation 10 Appendix}
\sqrt{1 + \frac{1}{\frac{1}{4}r^{2} - 1}} \cdot \frac{\sqrt{\mu^{T} \Sigma^{-1} \mu}}{\| \mu \|} \leq \frac{1}{\sqrt{\lambda_{\mu, \text{min}}}}.
\end{equation}
Since $r < -2$, we can see \eqref{case 2 and case 3 equation 10 Appendix} implies $r^{o}(\mu, \Sigma) := \| \mu \|^{2} / \left( \mu^{T} \Sigma^{-1} \mu \cdot \lambda_{\mu, \text{min}} \right) > 1$. As a result, we have proved that given $r<-1$, if $r^{o}(\mu, \Sigma) \leq 1$, then only the Case (2) happens. From Lemma \ref{reformualtion of the necessary and sufficient conditions again}, we know $\lambda_{\text{min}} \leq \lambda_{\mu, \text{min}} \leq \lambda_{\text{max}}$. When $\Sigma = a \cdot I$ with $a > 0$, we can see $r^{o}(\mu, \Sigma) = 1$, which corresponds to the Case (2). In fact, suppose $\Sigma = a \cdot I$ with $a > 0$ and $\overline{s}_{1}$ is tangent with $\overline{s}_{2}$, then the point of tangency must locate under the horizontal line $\overline{s} = -1/\lambda_{\mu, \text{min}}$. Thus, we can get better $\sqrt{\theta^{*}}$ following the Case (2).
Further, we can reformulate \eqref{case 2 and case 3 equation 10 Appendix} as follows
\begin{equation}\label{case 2 and case 3 equation 11 Appendix}
r \leq -2 \cdot \sqrt{\frac{1}{r^{o}(\mu, \Sigma) - 1 } + 1} =: r^{*}(\mu, \Sigma).
\end{equation}
Due to \eqref{case 2 and case 3 equation 10 Appendix}, we know $r^{o}(\mu, \Sigma) > 1$. Thus, we have $r^{*}(\mu, \Sigma) < -2$. As a result, when $r < r^{*}(\mu, \Sigma)$ and $r^{o}(\mu, \Sigma) > 1$, we can take $\sqrt{\theta^{*}} = c^{*}$, where $c^{*}$ is defined by \eqref{case 2 and case 3 equation 9 Appendix}. Assume $\| \mu \| \neq 0$ and $r^{o}(\mu, \Sigma) > 1$. Then, by combining \eqref{case 2 and case 3 equation 3 Appendix} and \eqref{case 2 and case 3 equation 11 Appendix}, we can see $r^{**} < r^{*}$, which means the Case (3) can not take place. In fact, since for any $z > 1$, we have $(2z - 1)^{2} > 4z \cdot (z - 1)$, which implies
\begin{equation}\label{case 2 and case 3 equation 12 Appendix}
-2 \sqrt{\frac{1}{z - 1} + 1} > \frac{1}{\frac{1}{z} - 1} - 1.
\end{equation}
Since $r^{o}(\mu, \Sigma) > 1$. Then, we can see $r^{**} < r^{*}$ by substituting $z = r^{o}(\mu, \Sigma)$ into \eqref{case 2 and case 3 equation 12 Appendix}. 
\end{proof}

\begin{lemma}\label{proof of part of Th. 3.4}
Assume $o \notin X$. Under the assumptions of Lemma \ref{Lemma about Hi-1}, define $g(x) := (b - \mu^{T}x)/\sqrt{x^{T}\Sigma x}$. Define $G(\theta^{**})$ by \eqref{"locally convex" set}. Assume $\| \mu \| \neq 0$ . If $b = 0$, then there exists an \textbf{best} threshold $\theta^{*}=\mu^{T} \Sigma^{-1} \mu$ for any $r \neq 0$, which means $\text{sign(-r)}\cdot g^{r}$ is locally convex on $G(\theta^{**})$ if and only if $\theta^{**} \geq \theta^{*} := \mu^{T} \Sigma^{-1} \mu$.
If $b < 0$ and $r \leq -2$, then we can take
\begin{equation*}
\theta^{*} = \frac{1}{\left( -1 - r \right)} \cdot \left[ \mu^{T} \Sigma^{-1} \mu - \left( 2 + r \right) \cdot \frac{\| \mu \|^{2}}{\lambda_{\text{min}}} \right]
\end{equation*}
or a greater one $\theta^{*} = \| \mu \|^{2} \cdot \lambda^{-1}_{\text{min}}$. If $b < 0$ and $-2 < r < -1$, then we can take $\theta^{*} = \mu^{T} \Sigma^{-1} \mu / (-1 - r)$ or simply $\theta^{*} = \| \mu \|^{2}\cdot \lambda^{-1}_{\text{min}} / (-1-r) $. If $b < 0$ and $r \geq -1$ and $r \neq 0$, then we can set
\begin{equation*}
\theta^{*} = \mu^{T} \Sigma^{-1} \mu + (2 + r) \cdot \frac{\| \mu \|^{2}}{\lambda_{\text{min}}}
\end{equation*}
or simply $\theta^{*} = (3 + r) \cdot \| \mu \|^{2} \cdot \lambda^{-1}_{\text{min}}$.
\end{lemma}
\begin{proof}
First, we consider the case with $b = 0$. Since $x \in E$, we know that
\begin{equation*}
0 = b > \mu^{T}x = \| \mu \| \cdot \| x \| \cdot \text{cos}\beta,
\end{equation*}
which implies $\text{cos} \beta < 0$. Thus, we can rewrite \eqref{2 transformation of key sufficient and necessary conndition inequality} as follows 
\begin{equation*}
\sqrt{\mu^{T} \Sigma^{-1} \mu} \leq \frac{- \text{cos} \beta \cdot \| \mu \|}{\sqrt{\lambda(x)}} = \frac{b - \| \mu \| \cdot \| x \| \cdot \text{cos}\beta}{\sqrt{\lambda(x)} \cdot \| x \|} = \frac{b - \mu^{T} x}{\sqrt{x^{T} \Sigma x}},
\end{equation*}
which means $\sqrt{\theta^{*}} = \sqrt{\mu^{T} \Sigma^{-1} \mu}$ is best for $\text{sign}(-r) \cdot g^{r}$ being locally convex on the set defined by \eqref{"locally convex" set}.

If $b < 0$, then we can yield
\begin{equation*}
0 > b > \mu^{T} x = \| \mu \| \cdot \| x \| \cdot \text{cos}\beta \geq -\| \mu \| \cdot \| x \|, \ \ \forall x \in E,
\end{equation*}
which implies
\begin{equation}\label{constraints about variables with b<0}
\text{cos}\beta < 0, \ \ 0 < \frac{1}{\| x \|} \leq \frac{\| \mu \|}{-b}.
\end{equation}
When $r \leq -2$, we can rewrite \eqref{2 transformation of key sufficient and necessary conndition inequality} as follows
\begin{equation}\label{3.1 transformation of key sufficient and necessary conndition inequality}
\begin{aligned}
&(-1 -r)\cdot \theta = \left( -1 - r \right)\cdot \left( \frac{b}{\sqrt{\lambda}} \cdot \frac{1}{\| x \|} - \frac{\| \mu \| \cdot \text{cos}\beta}{\sqrt{\lambda}} \right)^{2}\\
\geq & \mu^{T} \Sigma^{-1} \mu - \left( 2 + r \right) \cdot \frac{\| \mu \|^{2} \cdot \text{cos}^{2}\beta}{\lambda} - \left( -2r \right) \cdot \frac{b \cdot \| \mu \| \cdot \text{cos}\beta}{\lambda \cdot \| x \|} =: \overline{\theta}.
\end{aligned}
\end{equation}
Thus, if we take
\begin{equation*}
\sqrt{\theta^{*}} := \frac{1}{\sqrt{-1 - r}} \cdot \sqrt{\mu^{T} \Sigma^{-1} \mu - (2 + r) \cdot \frac{\| \mu \|^{2}}{\lambda_{\text{min}}}},
\end{equation*}
then $\theta^{*}$ is the required threshold since $\theta^{*} \geq \overline{\theta}$. Furthermore, we can see that
\begin{equation*}
\begin{aligned}
&\frac{1}{\sqrt{-1 - r}} \cdot \sqrt{\mu^{T} \Sigma^{-1} \mu - (2 + r) \cdot \frac{\| \mu \|^{2}}{\lambda_{\text{min}}}}\\
& \leq \frac{1}{\sqrt{-1 - r}} \cdot \sqrt{\| \mu \|^{2} \cdot  \lambda^{-1}_{\text{min}} - (2 + r) \cdot \frac{\| \mu \|^{2}}{\lambda_{\text{min}}}} = \| \mu \| \cdot \lambda^{-\frac{1}{2}}_{\text{min}}.
\end{aligned}
\end{equation*}
Then, we can take a greater threshold with a simpler form $\sqrt{\theta^{*}} := \| \mu \| \lambda^{-\frac{1}{2}}_{\text{min}}$. When $-2 < r < -1$, we have
\begin{equation*}
\overline{\theta} \leq \mu^{T} \Sigma^{-1} \mu \leq \| \mu \|^{2} \cdot \lambda^{-1}_{\text{min}},
\end{equation*}
where $\overline{\theta}$ is defined by \eqref{3.1 transformation of key sufficient and necessary conndition inequality}. As a result, we can take $\theta^{*} = \mu^{T} \Sigma^{-1} \mu / (-1-r)$ or a greater value $\theta^{*} = \| \mu \| \cdot \lambda^{-1}_{\text{min}}/ (-1-r)$. When $r \geq -1$, \eqref{2 transformation of key sufficient and necessary conndition inequality} can be reformulated as follows
\begin{equation*}
\begin{aligned}
&\theta = \left( \frac{b}{\sqrt{\lambda} \cdot \| x \|} - \frac{\| \mu \| \cdot \text{cos} \beta}{\sqrt{\lambda} \cdot \| x \|} \right)^{2}\\
&\geq \mu^{T} \Sigma^{-1} \mu + \left( 2 + r \right) \cdot \frac{b^{2}}{\lambda \cdot \| x \|^{2}} - 4\cdot \frac{b \cdot \| \mu \| \cdot  \text{cos}\beta}{\lambda \cdot \| x \|} =: \overline{\theta}.
\end{aligned}
\end{equation*}
And then, by using \eqref{constraints about variables with b<0} we yield
\begin{equation*}
\overline{\theta} \leq \mu^{T} \Sigma^{-1} \mu + (2 + r) \cdot \frac{b^{2}}{\lambda} \cdot \frac{\| \mu \|^{2}}{b^{2}} \leq \| \mu \|^{2} \cdot \lambda^{-1}_{\text{min}} \cdot (3 + r).
\end{equation*}
Thus, we can take $\theta^{*} = \mu^{T} \Sigma^{-1} \mu + (2 + r)\lambda^{-1}_{\text{min}} \| \mu \|^{2}$ or a greater one $\theta^{*} = (3 + r)\| \mu \|^{2} \lambda^{-1}_{\text{min}}$.
\end{proof}

\begin{proof}[Proof of Corollary \ref{Corollary of Hi-1-1}]
Applying \ZH{Theorem} 1 in \cite{minoux2016convexity}, we can yield a similar necessary and sufficient condition to \eqref{sufficient and necessary condition r<2 and r neq 0} as follows
\begin{equation}\label{Gaussian case 1}
\mu^{T} \Sigma^{-1} \mu \leq \theta^{2} - 2 \theta + 2 \sqrt{\theta}(\theta - 1) \frac{\mu^{T} x}{\sqrt{x^{T} \Sigma x}} + (\theta + 1) \frac{\left( \mu^{T} x \right)^{2}}{x^{T} \Sigma x}.
\end{equation}
Set $\lambda(x) := x^{T} \Sigma x / \| x \|^{2}$ and $\beta$ being the angle between vectors $x$ and $\mu$. By setting $\overline{t} := t / (\sqrt{\lambda} \cdot \| x \| )$ and $\overline{s} := cos \beta / \sqrt{\lambda}$, we can rewrite \eqref{Gaussian case 1} as follows
\begin{equation}\label{Gaussian case 2}
\begin{aligned}
\mu^{T} \Sigma^{-1} \mu &\leq \| \mu \|^{2} \cdot \left( \overline{s} \right)^{2} + 2b \cdot \| \mu \| \cdot \overline{t} \cdot \overline{s} - 2b^{2} \cdot \left( \overline{t} \right)^{2}\\
&+ b^{4} \cdot \left( \overline{t} \right)^{4} - 2b^{3} \cdot \| \mu \| \cdot \overline{s} \cdot \left( \overline{t} \right)^{3} + b^{2} \cdot \| \mu \|^{2} \cdot \left( \overline{s} \right)^{2} \cdot \left( \overline{t} \right)^{2}.
\end{aligned}
\end{equation}
The definitions of $\overline{t}$ and $\overline{s}$ imply $\overline{t} \in (0, +\infty)$, $\overline{s} \in [-1/ \sqrt{\lambda_{\mu, \text{min}}}, 1/ \sqrt{\lambda_{\mu, \text{min}}}]$. Suppose $b > 0$. Observe that if $\overline{t}$ is large enough, then for any $\overline{s} \in [-1/ \sqrt{\lambda_{\mu, \text{min}}}, 1/ \sqrt{\lambda_{\mu, \text{min}}}]$ the right hand side of \eqref{Gaussian case 2} will be uniformly greater than the given $\mu^{T} \Sigma^{-1} \mu$. As a result, following the discussion about the case with $b > 0$ in Theorem \ref{Corollary of Hi-1}, we can prove the existence of best probability threshold $p^{*}:= F(\sqrt{\theta^{*}})$. The best threshold $\theta^{*}$ can be yielded by computing the maximum and minimum values of $g(\overline{t}, \overline{s}) := b \cdot \overline{t} - \| \mu \| \cdot \overline{s}$ on the points $(\overline{t}, \overline{s})$ which do not satisfy \eqref{Gaussian case 2}. Suppose $b = 0$. Then, we can see it is the same case with $b = 0$ in Theorem \ref{Corollary of Hi-1}, which implies that the best probability threshold is $p^{*}:= F(\sqrt{\mu^{T} \Sigma^{-1} \mu})$. Consider the case with $b < 0$. The non-existence of threshold $\theta^{**}$ is the result from the decreasing property of \eqref{g function's transformation with overline t and overline s} and the increasing property of the right-hand side of \eqref{Gaussian case 2} concerning $\overline{t}$.
\end{proof}

\begin{proof}[Proof of Lemma \ref{Lemma about Hi-2}]
The proof of the first part refers to \ZH{Lemma} 3.1 in \cite{henrion2008convexity}. It suffices to prove the second part. Let $r > 0$ and set $h(z) := F(z^{\frac{1}{r}})$. Then $h(z)$ can be rewritten as
\begin{equation}\label{Lemma about Hi-2 equation 1}
h(z) = F(0) + \int^{z^{\frac{1}{r}}}_{0} f(t) dt.
\end{equation}
Taking $t = u^{\frac{1}{r}}$ into \eqref{Lemma about Hi-2 equation 1} we have
\begin{equation*}
h(z) = F(0) - r^{-1}\int^{+\infty}_{z} u^{\frac{1}{r} - 1} f(u^{\frac{1}{r}}) du,
\end{equation*}
thus
\begin{equation*}
h'(z) = r^{-1} \cdot z^{\frac{1}{r} -1} \cdot f(z^{\frac{1}{r}}).
\end{equation*}
The definition of $(-r + 1)$-decreasing function tell us that the mapping $t \rightarrow t^{-r + 1}f(t)$ is strictly decreasing on $(t^{*}, +\infty)$. It follows that the mapping $z \rightarrow z^{\frac{1}{r} - 1} f(z^{\frac{1}{r}})$ is strictly decreasing on $((t^{*})^{r}, +\infty)$, so does $h'(z)$. Therefore, we have proved that $h(z)$ is concave on $((t^{*})^{r}, +\infty)$ with $r > 0$.
\end{proof}

\begin{proof}[Proof of Proposition \ref{Lemma end 2}]
First, we prove $\varphi_{1} > 0$. Since $y_{i} \in (0, 1)$, we have
\begin{equation}\label{Lemma end 2_equation 1}
0 < -\ln y_{i}, \ \ i = 1, ... , K.
\end{equation}
Due to the (2) of Assumption \ref{Assumption 3}, we know
\begin{equation}\label{Lemma end 2_equation 2}
0 < \kappa(x) \leq 1, \ \ \forall x \in X.
\end{equation}
It follows that
\begin{equation}\label{Lemma end 2_equation 3}
0 < y_{i} \leq y^{\kappa(x)}_{i} \leq 1.
\end{equation}
Combining the (1) of Assumption \ref{Assumption 3} and $0<p<1$, we can get
\begin{equation}\label{Lemma end 2_equation 4}
0 < -\ln p \leq 1.
\end{equation}
Then, by combining \eqref{Lemma end 2_equation 3} and \eqref{Lemma end 2_equation 4} we yield
\begin{equation}\label{Lemma end 2_equation 5}
0 < -\ln p \cdot y_{i} \leq -\ln p \cdot y^{\kappa(x)}_{i} \leq 1
\end{equation}
Next, using \eqref{Lemma end 2_equation 2} and \eqref{Lemma end 2_equation 5}, for any $(x, y_{i})\in X \times (0, 1)$ we have
\begin{equation}\label{Lemma end 2_equation 6}
\begin{array}{ll}
& 0 < -\ln p \cdot y_{i} \leq -\ln p \cdot y^{\kappa(x)}_{i} = 1 - (1 + \ln p \cdot y^{\kappa(x)}_{i})\\
& \leq 1 - \kappa(x) \cdot (1 + \ln p \cdot y^{\kappa(x)}_{i}) = 1 - \kappa(x) - \kappa(x) \cdot \ln p \cdot y^{\kappa(x)}_{i}.
\end{array}
\end{equation}
As a result, by multiplying \eqref{Lemma end 2_equation 1}, \eqref{Lemma end 2_equation 2}, and \eqref{Lemma end 2_equation 6}, we can yield
\begin{equation*}
\begin{array}{ll}
\varphi_{1} &:= \kappa(x)\cdot\ln y_{i} \cdot [ \kappa(x) - 1 + \kappa(x) \cdot  \ln p \cdot y^{\kappa(x)}_{i} ]\\
&\geq \kappa(x) \cdot \ln y_{i} \cdot \ln p \cdot y_{i} > 0,
\end{array}
\end{equation*}
which completes the proof.

Next, we prove $\varphi_{2} > 0$. Using \eqref{Lemma end 2_equation 3} and \eqref{Lemma end 2_equation 4}, we have
\begin{equation}\label{Lemma end 2_equation 8}
0 < 1 + \ln p \cdot y^{\kappa(x)}_{i} < 1,
\end{equation}
then we combine \eqref{Lemma end 2_equation 8} and \eqref{Lemma end 2_equation 2} to get
\begin{equation}\label{Lemma end 2_equation 9}
0 < \kappa(x)\cdot \left( 1 + \ln p \cdot y^{\kappa(x)}_{i} \right) < 1.
\end{equation}
Thus, combining \eqref{Lemma end 2_equation 2}, \eqref{Lemma end 2_equation 8} and \eqref{Lemma end 2_equation 9} we can yield
\begin{equation*}
\begin{array}{ll}
\varphi_{2} &:= \kappa(x) \cdot \left( \ln y_{i} \right)^{2} \cdot \left( 1 + \ln p \cdot y^{\kappa(x)}_{i} \right) \cdot [ 1 - \kappa(x) - \kappa(x) \cdot  \ln p \cdot y^{\kappa(x)}_{i} ]\\
        & + \left( 1 + \kappa(x)\cdot \ln y_{i} + \ln p \cdot \ln y_{i} \cdot y^{\kappa(x)}_{i} \cdot \kappa(x) \right)^{2}\\
        & \geq \kappa(x) \cdot \left( \ln y_{i} \right)^{2} \cdot \left( 1 + \ln p \cdot y^{\kappa(x)}_{i} \right) \cdot [ 1 - \kappa(x) - \kappa(x) \cdot  \ln p \cdot y^{\kappa(x)}_{i} ]>0.
\end{array}
\end{equation*}

It remains to prove that there exists $d > 0$ such that $1/\omega(x, y_{i}) \geq d \cdot \kappa(x)$ for any $(x, y_{i}) \in X \times (0, 1)$. Before proving the result, we rewrite $\varphi_{2}$ as follows
\begin{equation*}
\begin{array}{ll}
\varphi_{2} &= \kappa(x) \cdot \left( \ln y_{i} \right)^{2} \cdot \left( 1 + \ln p \cdot y^{\kappa(x)}_{i} \right) \cdot [ 1 - \kappa(x) - \kappa(x) \cdot  \ln p \cdot y^{\kappa(x)}_{i} ]\\
            &+ \left( \kappa(x) \right)^{2} \cdot \left( \ln y_{i} \right)^{2} \cdot \left( 1 + \ln p \cdot y^{\kappa(x)}_{i} \right)^{2}  + 2\kappa(x) \cdot \ln y_{i} \cdot \left[ 1 + \ln p \cdot y^{\kappa(x)}_{i} \right] + 1.
\end{array}
\end{equation*}
We set
\begin{equation*}
\begin{array}{ll}
\textcircled{1} &:= \frac{\kappa(x) \cdot \left( \ln y_{i} \right)^{2} \cdot \left( 1 + \ln p \cdot y^{\kappa(x)}_{i} \right) \cdot [ 1 - \kappa(x) - \kappa(x) \cdot  \ln p \cdot y^{\kappa(x)}_{i} ]}{\kappa(x)\cdot\ln y_{i} \cdot [ \kappa(x) - 1 + \kappa(x) \cdot  \ln p \cdot y^{\kappa(x)}_{i} ]}\\
& = -\ln y_{i} \cdot \left( 1 + \ln p \cdot y^{\kappa(x)}_{i} \right),
\end{array}
\end{equation*}
and
\begin{equation*}
\begin{array}{ll}
\textcircled{2} &:= \frac{ \kappa(x) \cdot \left( -\ln y_{i} \right) \cdot \left( 1 + \ln p \cdot y^{\kappa(x)}_{i} \right)^{2} }{\left[ 1 - \kappa(x) \cdot \left( 1 + \ln p \cdot y^{\kappa(x)}_{i} \right) \right]} + \frac{-2\cdot \left[ 1 + \ln p \cdot y^{\kappa(x)}_{i} \right]}{\left[ 1 - \kappa(x) \cdot \left( 1 + \ln p \cdot y^{\kappa(x)}_{i} \right) \right]}\\
& + \frac{\left( \kappa(x) \cdot \left( -\ln y_{i} \right) \right)^{-1}}{\left[ 1 - \kappa(x) \cdot \left( 1 + \ln p \cdot y^{\kappa(x)}_{i} \right) \right]},
\end{array}
\end{equation*}
then $\omega := \varphi_{2}/ \varphi_{1} = \textcircled{2} + \textcircled{1}$.
Then, by using \eqref{Lemma end 2_equation 2} and \eqref{Lemma end 2_equation 5} we can yield
\begin{equation}\label{Lemma end 2_equation 10}
\textcircled{1} = -\ln y_{i} \cdot \left( 1 + \ln p \cdot y^{\kappa(x)}_{i} \right) \leq -\ln y_{i} \cdot \left( 1 + \ln p \cdot y_{i} \right)
\end{equation}
and
\begin{equation*}
1 - \kappa(x) \cdot \left( 1 + \ln p \cdot y^{\kappa(x)}_{i} \right) \geq 1 - 1 - \ln p \cdot y_{i} = -\ln p \cdot y_{i} > 0, 
\end{equation*}
which implies $\textcircled{2}$ is well-defined. Due to Assumption \ref{Assumption 3} and $0 < y_{i} < 1$, we have
\begin{equation*}
\kappa(x) \cdot \left( -\ln y_{i} \right) \cdot \left( 1 + \ln p \cdot y^{\kappa(x)}_{i} \right)^{2} > 0,
\end{equation*}
\begin{equation*}
-2\cdot \left[ 1 + \ln p \cdot y^{\kappa(x)}_{i} \right] < 0, \ \ \left( \kappa(x) \cdot \left( -\ln y_{i} \right) \right)^{-1} > 0.
\end{equation*}
It follows that
\begin{equation}\label{Lemma end 2_equation 12}
\begin{array}{ll}
\textcircled{2} &\leq \frac{ \kappa(x) \cdot \left( -\ln y_{i} \right) \cdot \left( 1 + \ln p \cdot y^{\kappa(x)}_{i} \right)^{2} }{\left[ 1 - \kappa(x) \cdot \left( 1 + \ln p \cdot y^{\kappa(x)}_{i} \right) \right]} + \frac{\left( \kappa(x) \cdot \left( -\ln y_{i} \right) \right)^{-1}}{\left[ 1 - \kappa(x) \cdot \left( 1 + \ln p \cdot y^{\kappa(x)}_{i} \right) \right]}\\
& \ \\
& \leq \kappa(x) \cdot \left[ \frac{\left( -\ln y_{i} \right) \cdot \left( 1 + \ln p \cdot y^{\kappa(x)}_{i} \right)^{2}}{-\ln p \cdot y_{i}} \right] + \frac{1}{\kappa(x)} \cdot \left[ \frac{\left( -\ln y_{i} \right)^{-1}}{-\ln p \cdot y_{i}} \right]\\
& \ \\
& \leq \kappa(x) \cdot \left[ \frac{\left( -\ln y_{i} \right) \cdot \left( 1 + \ln p \cdot y_{i} \right)^{2}}{-\ln p \cdot y_{i}} \right] + \frac{1}{\kappa(x)} \cdot \left[ \frac{\left( -\ln y_{i} \right)^{-1}}{-\ln p \cdot y_{i}} \right]
\end{array}
\end{equation}
Combing \eqref{Lemma end 2_equation 10} and \eqref{Lemma end 2_equation 12}, we yield
\begin{equation}\label{Lemma end 2_equation 13}
\begin{array}{ll}
\omega = \textcircled{1} + \textcircled{2} &\leq -\ln y_{i} \cdot \left( 1 + \ln p \cdot y_{i} \right)\\
& + \kappa(x) \cdot \left[ \frac{\left( -\ln y_{i} \right) \cdot \left( 1 + \ln p \cdot y_{i} \right)^{2}}{-\ln p \cdot y_{i}} \right] + \frac{1}{\kappa(x)} \cdot \left[ \frac{\left( -\ln y_{i} \right)^{-1}}{-\ln p \cdot y_{i}} \right]\\
& =: a_{2} + a_{1} \cdot \kappa(x) + a_{3} \cdot \frac{1}{\kappa(x)},
\end{array}
\end{equation}
where
\begin{equation*}
a_{1} := \frac{\left( -\ln y_{i} \right) \cdot \left( 1 + \ln p \cdot y_{i} \right)^{2}}{-\ln p \cdot y_{i}} , \ \ a_{2} := -\ln y_{i} \cdot \left( 1 + \ln p \cdot y_{i} \right), \ \ a_{3} := \frac{\left( -\ln y_{i} \right)^{-1}}{-\ln p \cdot y_{i}}.
\end{equation*}
It can be seen that $a_{1}$, $a_{2}$, and $a_{3}$ are positive. Applying the (2) of Assumption \ref{Assumption 3}, we can see
\begin{equation}\label{Lemma end 2_equation 14}
0 < \kappa(x) \leq 1 \leq \frac{1}{\kappa(x)}.
\end{equation}
Thus, combing \eqref{Lemma end 2_equation 13} and \eqref{Lemma end 2_equation 14} we can get
\begin{equation*}
\omega(x, y_{i}) \leq a_{1} \kappa(x) + a_{2} + a_{3} \cdot \frac{1}{\kappa(x)} \leq \left( a_{1} + a_{2} + a_{3} \right) \cdot \frac{1}{\kappa(x)},
\end{equation*}
that is
\begin{equation*}
\frac{1}{\omega(x, y_{i})} \geq \kappa(x) \cdot \frac{1}{\left( a_{1} + a_{2} + a_{3} \right)} =: \kappa(x) \cdot d.
\end{equation*}
The proof is now complete.
\end{proof}

\begin{proof}[Proof of Proposition \ref{The detail of the second part of N-0--1}]
Define
\begin{equation}\label{Definition of matrix N}
N(x, y_{i}) := \frac{\partial^{2}}{\partial y^{2}_{i}} U(x, y_{i})\cdot H_{x}U(x, y_{i}) - \left( \nabla_{x} \frac{\partial}{\partial y_{i}} U(x, y_{i}) \right) \cdot \left( \nabla_{x} \frac{\partial}{\partial y_{i}} U(x, y_{i}) \right)^{T}.
\end{equation}
First, we claim that the matrix $H_{(x, y_{i})} U$ is positive semi-definite if and only if the matrix $N$ is positive semi-definite for any $(x, y_{i}) \in X \times (0, 1)$. In fact, we can see that if $\frac{\partial^{2}}{\partial y^{2}_{i}} U(x, y_{i}) \neq 0$ then $N(x, y_{i}) / \frac{\partial^{2}}{\partial y^{2}_{i}} U(x, y_{i})$ is the Schur complement of $\frac{\partial^{2}}{\partial y^{2}_{i}} U(x, y_{i})$ in the matrix $H_{(x, y_{i})} U(x, y_{i})$. As a result, by applying the (2) of Proposition \ref{Property of Schur complement}, it suffices to show $\frac{\partial^{2}}{\partial y^{2}_{i}} U(x, y_{i}) > 0$ on $X \times (0, 1)$. Since $U(x, y_{i}) := \psi^{(-1)}_{x} (y_{i} \psi_{x}(p))$ and $\psi_{x}(t) := \left( -\ln t \right)^{\frac{1}{\kappa(x)}}$, then we can yield
\begin{equation*}
\frac{\partial^{2}}{\partial y^{2}_{i}} U(x, y_{i}) = \left( \psi_{x}(p) \right)^{2} \cdot \left( \psi^{(-1)}_{x} \right)'' \cdot \left( y_{i} \psi_{x}(p) \right),
\end{equation*}
\begin{equation*}
\psi^{(-1)}_{x}(t) = e^{- t^{\kappa(x)}} > 0,
\end{equation*}
\begin{equation*}
\left( \psi^{(-1)}_{x} \right)''(t) = e^{-t^{\kappa(x)}} \cdot t^{\kappa(x) - 2} \cdot \kappa(x) \cdot \left( \kappa(x)\cdot t^{\kappa(x)} - \kappa(x) + 1 \right) > 0.
\end{equation*}
It follows that $\frac{\partial^{2}}{\partial y^{2}_{i}} U(x, y_{i}) > 0$ for any $(x, y_{i}) \in X \times (0, 1)$.

Next, we show that the positive semi-definiteness of $N(x, y_{i})$ is equivalent to the positive semi-definiteness of the matrix $M(x, y_{i})$ defined by
\begin{equation}\label{Definition of matrix M}
M(x, y_{i}) := \varphi_{1}(x, y_{i}) \cdot H_{x}\kappa(x) - \varphi_{2}(x, y_{i}) \cdot \left( \nabla_{x} \frac{\partial}{\partial y_{i}} U(x, y_{i}) \right) \cdot \left( \nabla_{x} \frac{\partial}{\partial y_{i}} U(x, y_{i}) \right)^{T},
\end{equation}
where $H_{x}\kappa$ is the Hessian matrix of $\kappa$.
Taking $\psi_{x}(t) := \left( -\ln t \right)^{\frac{1}{\kappa(x)}}$ and $\psi^{(-1)}_{x}(t) := e^{- t^{\kappa(x)}}$ into $U(x, y_{i}) := \psi^{(-1)}_{x} (y_{i} \psi_{x}(p))$, we can have
\begin{equation}\label{The detail of the first part of N}
\begin{array}{ll}\
&\frac{\partial^{2}}{\partial y^{2}_{i}}U(x, y_{i}) \cdot H_{x} U(x, y_{i})\\
=& \left[ \left( \ln p \right)^{2} y^{2\kappa(x) - 2}_{i} p^{2y^{\kappa(x)}_{i}} \right] \cdot \left[ \kappa(x) \left( \ln y_{i} \right) \right] \cdot \left[ \kappa(x) - 1 + \kappa(x) \left( \ln p \right) y^{\kappa(x)}_{i} \right]\\
& \cdot \left[ H_{x}\kappa(x) + \left( \nabla_{x} \kappa(x) \right) \cdot \left( \nabla_{x} \kappa(x) \right)^{T} \cdot \left( \ln y_{i} + \left(\ln y_{i} \right)\left( \ln p \right) y^{\kappa(x)}_{i} \right) \right]\\
\end{array},
\end{equation}
\begin{equation}\label{The detail of the second part of N}
\begin{array}{ll}
&\left( \nabla_{x} \frac{\partial}{\partial y_{i}} U(x, y_{i}) \right) \cdot \left( \nabla_{x} \frac{\partial}{\partial y_{i}} U(x, y_{i}) \right)^{T}\\
=& \left[ \left( \ln p \right)^{2} y^{2\kappa(x) - 2}_{i} p^{2y^{\kappa(x)}_{i}} \right] \cdot \left[ \left( \nabla_{x} \kappa(x) \right) \cdot \left( \nabla_{x} \kappa(x) \right)^{T} \right]\\
& \cdot \left[ 1 + \kappa(x) \left( \ln y_{i} \right) + \left( \ln p \right) \left( \ln y_{i} \right) y^{\kappa(x)}_{i} \kappa(x) \right]^{2}
\end{array}.
\end{equation}
Thus, combing \eqref{Definition of matrix N}-\eqref{The detail of the second part of N}, we can get
\begin{equation*}
N(x, y_{i}) = \left[ \left( \ln p \right)^{2} y^{2\kappa(x) - 2}_{i} p^{2y^{\kappa(x)}_{i}} \right] \cdot M(x, y_{i}).*
\end{equation*}
It can be noticed that $\left( \ln p \right)^{2} y^{2\kappa(x) - 2}_{i} p^{2y^{\kappa(x)}_{i}} > 0$. Therefore, the positive semi-definiteness of $N(x, y_{i})$ is equivalent to the positive semi-definiteness of the matrix $M(x, y_{i})$, which completes the proof.
\end{proof}

\begin{proof}[Proof of Lemma \ref{Lemma end 0}]
Since $o \in S(p)$ and $p \in (0, 1]$, by observing \eqref{Fea-1} we can have $D \geq 0$ (componentwise). Take arbitrary $\lambda 
\in [0, 1]$, $x \in S(p)$. We need to prove that $\lambda x \in S(p)$. If $\lambda = 0$, then
\begin{equation*}
\lambda x = o \in S(p).
\end{equation*}
If $\lambda \in (0, 1]$, then
\begin{equation*}
\mathbb{P}\lbrace V \cdot \left( \lambda x \right) \leq D \rbrace = \mathbb{P}\lbrace V \cdot x \leq \lambda^{-1} \cdot
 D \rbrace \geq \mathbb{P}\lbrace V \cdot x \leq D \rbrace \geq p,
\end{equation*}
which implies $\lambda x \in S(p)$. Thus, the \ZH{feasible set} $S(p)$ is star-shaped with respect to o.
\end{proof}

\section{Proof of Section \ref{Eventual convexity with skewed distributions} \ZHnew{results}}\label{Appendix H}

\

\begin{proof}[Proof of Corollary \ref{density of special GH_1}]
Since a random vector $Z \sim N^{-}(\lambda, \chi, 0)$ is equivalent to $Z \sim Ig(-\lambda, \frac{1}{2} \chi)$. Thus, by observing the GIG density defined in Appendix \ref{definition of Generalized inverse Gaussian distributions}, we can yield
\begin{equation}\label{limitation of coefficient}
\lim_{s \to 0^{+}} \frac{ s^{\lambda}}{K_{\lambda} \left( s \right)} = \frac{2^{\lambda + 1}}{\Gamma \left( - \lambda \right)}.
\end{equation}
Suppose $\xi(x) \sim GH_{1} \left( \lambda, \chi, \psi, 0, 1, \frac{x^{T}\gamma}{\sqrt{x^{T} \Sigma x}} \right)$. Follow the notations defined by \eqref{auxiliary values 1} and \eqref{auxiliary values 2}. We can see $\Lambda = 0$ is equivalent to $\psi = 0$ and $\phi = 0$. Thus, by using \eqref{1-dimensional non-integral form _ density of GH_d random vector with W density form} and \eqref{limitation of coefficient} we can yield
\begin{equation*}
\begin{aligned}
& \lim_{\Lambda \to 0^{+}} f_{\xi(x)}(t) = \lim_{\Lambda \to 0^{+}} \frac{\left( \sqrt{\chi \cdot \psi} \right)^{\lambda}}{K_{\lambda} \left( \sqrt{ \chi \cdot \psi} \right)} \cdot \frac{\left( \sqrt{ \chi \cdot \psi } \right)^{-2\lambda} \cdot \psi^{\lambda} \cdot \Lambda^{1-2\lambda}}{\sqrt{ 2 \pi}} \cdot \frac{K_{\lambda-\left( 1/2 \right)}\left( \Lambda \cdot \eta(t) \right)}{\left( \Lambda \cdot \eta(t) \right)^{\lambda - \left( 1/2 \right)}} \cdot \frac{\sigma(t)}{\left( \Lambda \cdot \eta(t) \right)^{1 - 2\lambda}}\\
& = \frac{2^{\lambda + 1}}{\Gamma \left( -\lambda \right)} \cdot \frac{\Gamma \left( - \lambda + \frac{1}{2} \right)}{2^{\lambda + \frac{1}{2}}} \cdot \lim_{\Lambda \to 0^{+}} \frac{\left( \sqrt{ \chi \cdot \psi } \right)^{-2\lambda} \cdot \psi^{\lambda} \cdot \Lambda^{1 - 2\lambda}}{\sqrt{2 \pi}} \cdot \frac{\sigma(t)}{\left( \Lambda \cdot \eta(t) \right)^{1 - 2\lambda}}\\
& = \frac{\Gamma \left( -\lambda + \frac{1}{2} \right)}{\Gamma \left( - \lambda \right)} \cdot \frac{1}{\sqrt{ \pi }} \cdot \lim_{\Lambda \to 0^{+}} \frac{\chi^{-\lambda} \cdot \sigma(t)}{\left( \eta(t) \right)^{1 - 2\lambda}} = \frac{\Gamma \left( -\lambda + \frac{1}{2} \right)}{\Gamma \left( - \lambda \right)} \cdot \frac{1}{\sqrt{ \pi }} \cdot \chi^{-\lambda} \cdot \left( \chi + t^{2} \right)^{\lambda - \left( 1/2 \right)},
\end{aligned}
\end{equation*}
which completed the proof.
\end{proof}

Suppose $x \in \mathbb{R}^{N} \setminus \left\lbrace 0 \right\rbrace$ and a random vector $\xi(x) \sim GH_{1}(\lambda, \chi, \psi, 0, 1, \frac{x^{T}\gamma}{\sqrt{x^{T}\Sigma x}})$. For simplicity, we make the following definitions throughout this paper: 
\begin{equation}\label{auxiliary values 1}
\eta(t) := \sqrt{ \chi + t^{2}}, \ \ \phi := \frac{x^{T}\gamma}{\sqrt{x^{T} \Sigma x}}, \ \ \sigma(t) := e^{t \cdot \phi}, \ \ \Lambda := \sqrt{ \psi + \phi^{2}}.
\end{equation}
A simple computation gives
\begin{equation}\label{auxiliary values 2}
\eta^{'}(t) = \frac{t}{\sqrt{\chi + t^{2}}}, \ \ \sigma^{'}(t) = \frac{x^{T}\gamma}{\sqrt{x^{T} \Sigma x}} \sigma(t).
\end{equation}

\begin{lemma}\label{preparation result of alpha-decreasing 1}
Assume $\xi(x) \sim GH_{1}(\lambda, \chi, \psi, 0, 1, \frac{x^{T}\gamma}{\sqrt{x^{T}\Sigma x}})$. Follow the notations defined by \eqref{auxiliary values 1} and \eqref{auxiliary values 2}. Let $J$ be defined on $(0, +\infty)$ by
\begin{equation}\label{the division about Bessel functions}
J(t) := \frac{K_{\lambda - \left( 1/2 \right)} \left( \Lambda \cdot \eta \left( t \right) \right)}{K_{\lambda + \left( 1/2 \right)}\left( \Lambda \cdot \eta \left( t \right) \right)}.
\end{equation}
Then, there exists $t_{0} > 0$ such that
\begin{equation}\label{boundary of J}
0 < J(t) \leq \sup_{t > t_{0}} J(t) =: c_{0}(t_{0}) < +\infty, \ \ \forall t > t_{0},
\end{equation}
Further, for any $t>0$, the function $J(t)$ satisfies
\begin{equation}\label{the boundary of c_0}
\left\{
        \begin{array}{ll}
        0 < J(t) < 1, & \text{if} \ \ \lambda > 0,\\            
        J(t) \equiv 1, & \text{if} \ \ \lambda = 0,\\
        J(t) > 1, & \text{if} \ \ \lambda < 0,
        \end{array}
\right.
\end{equation}
and
\begin{equation}\label{limitation of J}
J(t) \simeq \left( \frac{1}{1 + \frac{1}{\Lambda \cdot \eta(t) }} \right)^{\lambda} \to 1 \ \ \text{as} \ \ t \to +\infty.
\end{equation}
\end{lemma}
\begin{proof}
It is obvious to yield \eqref{auxiliary values 2} by simple computation. From the Definition \ref{definition of Generalized hyperbolic distributions}, we know that \eqref{auxiliary values 1} and \eqref{auxiliary values 2} are well-defined for $t > 0$. Applying the property (1) and (6) of Lemma \ref{The properties of modified Bessel functions of the third kind}, we can see
\begin{equation}\label{the property of the division about Bessel functions}
0 < \frac{K_{\lambda - \left( 1/2 \right)}(t)}{K_{\lambda + \left( 1/2 \right)}(t)}\simeq \left( 1 + \frac{1}{t} \right)^{ \frac{1}{2} \left[ \left( \lambda - \left( 1/2 \right) \right)^{2} - \left( \lambda + \left( 1/2 \right) \right)^{2} \right] } = \left( \frac{1}{1 + \frac{1}{t}} \right)^{\lambda} \to 1 \ \ \text{as} \ \ t \to +\infty.
\end{equation}
From \eqref{auxiliary values 1}, we can see $\Lambda\cdot \eta(t)$ is increasing w.r.t. $t>0$ and $\lim_{t \to +\infty} \Lambda \cdot \eta(t) = +\infty$. Thus, we can yield \eqref{boundary of J} and \eqref{limitation of J} by using \eqref{the property of the division about Bessel functions}. Using the property (4) and (5) of Lemma \ref{The properties of modified Bessel functions of the third kind}, we know that $K_{\lambda - \left( 1/2 \right)} < K_{\lambda + \left( 1/2 \right)}$ if and only if $\left| \lambda - \left( 1/2 \right) \right| < \left| \lambda + \left( 1/2 \right) \right|$. Thus, we can obtain
\begin{equation*} 
\left\{
        \begin{array}{ll}
        0 < J(t) < 1, & \text{if} \ \ \left|\lambda - \frac{1}{2} \right| < \left| \lambda + \frac{1}{2} \right|,\\      
        \ &\ \\        
        J(t) \equiv 1, & \text{if} \ \ \left|\lambda - \frac{1}{2} \right| = \left| \lambda + \frac{1}{2} \right|,\\
        \ &\ \\
        J(t) > 1, & \text{if} \ \ \left|\lambda - \frac{1}{2} \right| > \left| \lambda + \frac{1}{2} \right|,
        \end{array}
\right.
\end{equation*}
i.e.
\begin{equation*}
\left\{
        \begin{array}{ll}
        0 < J(t) < 1, & \text{if} \ \ \lambda > 0,\\            
        J(t) \equiv 1, & \text{if} \ \ \lambda = 0,\\
        J(t) > 1, & \text{if} \ \ \lambda < 0.
        \end{array}
\right.
\end{equation*}
\end{proof}

\begin{lemma}\label{preparation result of alpha-decreasing 2}
Under the assumptions of Lemma \ref{preparation result of alpha-decreasing 1}, assume $\Lambda > 0$. Then, the function $J$ defined by  \eqref{the division about Bessel functions} satisfies
\begin{equation*}
\lim_{t \to +\infty} t \cdot \left( 1 - J^{2}(t) \right) = \frac{2\lambda}{\Lambda}.
\end{equation*}
\end{lemma}
\begin{proof}
If $\lambda = 0$, the result is trivial since $J(t) \equiv 1$ is obtained from the property (4) of Lemma \ref{The properties of modified Bessel functions of the third kind}. Suppose $\lambda \neq 0$. Using \eqref{limitation of J} and letting $t$ be large enough, we can see
\begin{equation*}
t \cdot \left( 1 - \left( J(t) \right)^{2} \right) \simeq t \cdot \left[ 1 - \left( \frac{1}{1 + \frac{1}{\Lambda \cdot \eta(t) }} \right)^{2\lambda} \right] = t \cdot \left[ 1 - \left( 1 + \frac{1}{\Lambda \cdot \eta(t)} \right)^{-2\lambda} \right] =: A(t).
\end{equation*}
Applying L'Hospital's Rule and using \eqref{auxiliary values 2}, we can yield
\begin{equation*}
\begin{aligned}
\lim_{t \to +\infty} A(t) &= \lim_{t \to +\infty} \frac{1 - \left( 1 + \frac{1}{\Lambda \cdot \eta(t)} \right)^{-2\lambda}}{t^{-1}} = \lim_{t \to +\infty} \frac{2\lambda \cdot \left( 1 + \frac{1}{\Lambda \cdot \eta(t)} \right)^{-2\lambda - 1} \cdot \frac{\eta^{'}(t)}{\Lambda \cdot \eta^{2}(t)}}{t^{-2}}\\
&= \lim_{t \to +\infty} \frac{2\lambda}{\Lambda} \cdot \left( 1 + \frac{1}{\Lambda \cdot \eta(t)} \right)^{-2\lambda - 1} \cdot \frac{t^{3}}{\left( \chi + t^{2} \right)^{\frac{3}{2}}} = \frac{2\lambda}{\Lambda},
\end{aligned}
\end{equation*}
which finishes the proof of this Lemma.
\end{proof}

\begin{proof}[Proof of Lemma \ref{alpha-decreasing of density function}]
In order to prove the $\alpha$-decreasing property of $f_{\xi}$, we only need to prove there exists $t^{*}(\alpha) > 0$ such that for any $t > t^{*}(\alpha)$ the following inequality is satisfied
\begin{equation}\label{alpha_decreasing proof}
t f^{'}_{\xi(x)} (t) + \alpha f_{\xi(x)}(t) < 0.
\end{equation}
At the beginning, we consider the case with $\Lambda := \sqrt{\psi + \phi^{2}} > 0$. Substituting \eqref{auxiliary values 1} and \eqref{auxiliary values 2} into \eqref{1-dimensional non-integral form _ density of GH_d random vector with W density form}, we have
\begin{equation*}
f_{\xi(x)}(t) = c \frac{K_{\lambda - (1/2)} \left( \Lambda \cdot \eta(t) \right) \cdot \sigma(t)}{\left( \Lambda \cdot \eta(t) \right)^{\left( \frac{1}{2} - \lambda \right)}},
\end{equation*}
\begin{equation*}
\begin{aligned}
f^{'}_{\xi(x)}(t) = c \frac{1}{\left( \Lambda \cdot \eta(t) \right)^{1 - 2 \lambda}}& \left\lbrace \left[ K^{'}_{\lambda - \left( 1/2 \right)} \cdot \Lambda \cdot \eta^{'} \cdot \sigma + K_{\lambda - \left( 1/2 \right)} \cdot \sigma^{'} \right] \cdot \left( \Lambda \cdot \eta \right)^{\left( \frac{1}{2} - \lambda \right)} \right.\\
& \left. - \left[ \left( \frac{1}{2} - \lambda \right) \cdot \left( \Lambda \cdot \eta(t) \right)^{\left( -\frac{1}{2} - \lambda \right)} \cdot \Lambda \cdot \eta^{'}(t) \cdot K_{\lambda- \left( 1/2 \right)} \cdot \sigma \right] \right\rbrace.
\end{aligned}
\end{equation*}
Thus, we can reformulate the left hand side of \eqref{alpha_decreasing proof} as follows
\begin{equation}\label{alpha_decreasing proof-simplification-1}
\begin{aligned}
t f^{'}_{\xi(x)} (t) + \alpha f_{\xi(x)}(t) = & c \left( \Lambda \cdot \eta \right)^{\lambda - \left(1/2\right)} \cdot \left\lbrace \left[ - t \cdot \left( \Lambda \cdot \eta \right)^{-1} \cdot \left( \frac{1}{2} - \lambda \right) \cdot \Lambda \cdot \eta^{'} \cdot K_{\lambda-\left(1/2\right)} \cdot \sigma \right] \right.\\
& + \left. \left[ t \cdot K^{'}_{\lambda - \left(1/2\right)} \cdot \Lambda \cdot \eta^{'} \cdot \sigma + t \cdot K_{\lambda - \left(1/2\right)} \cdot \sigma^{'} \right] + \left[ \alpha \cdot K_{\lambda - \left(1/2\right)} \cdot \sigma \right] \right\rbrace\\
= & c \cdot \sigma \cdot \left( \Lambda \cdot \eta \right)^{\lambda - \left(1/2\right)} \cdot \left\lbrace \left[ -\frac{t^{2}}{\chi + t^{2}} \cdot \left( 
\frac{1}{2} - \lambda \right) \cdot K_{\lambda - \left(1/2\right)} \right] \right.\\
& + \left. \left[ K^{'}_{\lambda - \left(1/2\right)} \cdot \Lambda \cdot \frac{t^{2}}{\sqrt{\chi + t^{2}}} + t \cdot K_{\lambda - \left(1/2\right)} \cdot \frac{x^{T} \gamma}{\sqrt{ x^{T} \Sigma x}} \right] + \left[ \alpha \cdot K_{\lambda - \left(1/2\right)} \right] \right\rbrace.
\end{aligned}
\end{equation}
Applying the property (7) of Lemma \ref{The properties of modified Bessel functions of the third kind}, we can yield
\begin{equation}\label{the link between Bessel function and its differentional function for proof}
K^{'}_{\lambda - \left(1/2 \right)}\left( \Lambda \cdot \eta(t) \right) = \frac{\lambda - \left( 1/2 \right)}{\Lambda \cdot \eta(t)} \cdot K_{\lambda - \left(1/2 \right)}\left( \Lambda \cdot \eta(t) \right) - K_{\lambda + \left(1/2 \right)}\left( \Lambda \cdot \eta(t) \right).
\end{equation}
As a result, by substituting \eqref{the link between Bessel function and its differentional function for proof} into \eqref{alpha_decreasing proof-simplification-1}, we can obtain
\begin{equation*}
\begin{aligned}
t f^{'}_{\xi(x)} (t) + \alpha f_{\xi(x)}(t) = & c \cdot \sigma \cdot \left( \Lambda \cdot \eta \right)^{\lambda - \left(1/2\right)} \cdot \left\lbrace \left[ -\frac{t^{2}}{\chi + t^{2}} \cdot \left( 
\frac{1}{2} - \lambda \right) \cdot K_{\lambda - \left(1/2\right)} \right] + \left[ \alpha \cdot K_{\lambda - \left(1/2\right)} \right] \right.\\
& \left. + \left[ t \cdot K_{\lambda - \left(1/2\right)} \cdot \frac{x^{T} \gamma}{\sqrt{ x^{T} \Sigma x}} \right] +\frac{ \Lambda \cdot t^{2}}{\sqrt{\chi + t^{2}}} \cdot \left[ \frac{\lambda - \left( 1/2 \right)}{\Lambda \cdot \sqrt{\chi + t^{2}}} K_{\lambda - \left( 1/2 \right)} - K_{\lambda + \left( 1/2 \right)} \right]   \right\rbrace\\
= & c \cdot \sigma \cdot \left( \Lambda \cdot \eta \right)^{\lambda - \left(1/2\right)} \cdot \left\lbrace - \frac{ \Lambda \cdot t^{2}}{\sqrt{\chi + t^{2}}} \cdot K_{\lambda + \left( 1/2 \right)} \right.\\
& \left. + \left[ \frac{t^{2}}{\chi + t^{2}} \cdot \left( 2 \lambda -1 \right) + t \cdot \frac{x^{T}\gamma}{\sqrt{x^{T}\Sigma x}} + \alpha    \right] \cdot K_{\lambda - \left(1/2\right)} \right\rbrace =: A(t) \cdot B,
\end{aligned}
\end{equation*}
where $A(t)$ and $B$ are defined by
\begin{equation*}
A(t) := - \frac{ \Lambda \cdot t^{2}}{\sqrt{\chi + t^{2}}} \cdot K_{\lambda + \left( 1/2 \right)} + \left[ \frac{t^{2}}{\chi + t^{2}} \cdot \left( 2 \lambda -1 \right) + t \cdot \frac{x^{T}\gamma}{\sqrt{x^{T}\Sigma x}} + \alpha    \right] \cdot K_{\lambda - \left(1/2\right)},
\end{equation*}
\begin{equation*}
B := c \cdot \sigma \cdot \left( \Lambda \cdot \eta \right)^{\lambda - \left(1/2\right)}.
\end{equation*}
Combining the Definition \ref{definition of Generalized hyperbolic distributions}, \eqref{auxiliary values 1}, \eqref{auxiliary values 2} and $\Lambda > 0$,  we can see $B > 0$. Thus, the inequality \eqref{alpha_decreasing proof} is equivalent to $A(t) < 0$, which can be reformulated as follows
\begin{equation}\label{alpha_decreasing proof-simplification-2}
\left[ \left( 2 \lambda -1 \right) - \frac{\chi \cdot \left( 2 \lambda - 1 \right)}{\chi + t^{2}} + t \cdot \frac{x^{T}\gamma}{\sqrt{x^{T}\Sigma x}} + \alpha    \right] \cdot K_{\lambda - \left(1/2\right)} < \left[  \Lambda \cdot \sqrt{\chi + t^{2}} - \frac{\Lambda \cdot \chi}{\sqrt{\chi + t^{2}}} \right] \cdot K_{\lambda + \left( 1/2 \right)}.
\end{equation}
Using the property (1) of Lemma \ref{The properties of modified Bessel functions of the third kind}, we can rewrite \eqref{alpha_decreasing proof-simplification-2} as follows
\begin{equation}\label{alpha_decreasing proof-simplification-3}
\left[ \left( 2 \lambda -1 \right) - \frac{\chi \cdot \left( 2 \lambda - 1 \right)}{\chi + t^{2}} + t \cdot \frac{x^{T}\gamma}{\sqrt{x^{T}\Sigma x}} + \alpha    \right] \cdot J(t) < \left[  \Lambda \cdot \sqrt{\chi + t^{2}} - \frac{\Lambda \cdot \chi}{\sqrt{\chi + t^{2}}} \right],
\end{equation}
where $J$ is defined by \eqref{the division about Bessel functions}. Using \eqref{auxiliary values 1} and \eqref{auxiliary values 2}, we can reformulate \eqref{alpha_decreasing proof-simplification-3} as follows
\begin{equation*}
\left[ \left( 2 \lambda -1 \right) - \frac{\chi \cdot \left( 2 \lambda - 1 \right)}{\chi + t^{2}} + t \cdot \phi + \alpha  \right] \cdot J(t) < \sqrt{\psi + \phi^{2}} \cdot \left[  \sqrt{\chi + t^{2}} - \frac{\chi}{\sqrt{\chi + t^{2}}} \right],
\end{equation*}
and the equivalent form
\begin{equation}\label{alpha_decreasing proof-simplification-5}
\begin{aligned}
& J(t) \cdot \left( 2 \lambda - 1 \right) + J(t) \cdot \alpha - J(t)\cdot \frac{\chi \cdot \left( 2 \lambda -1 \right)}{\chi + t^{2}} + \sqrt{\psi + \phi^{2}} \cdot \frac{\chi}{\sqrt{ \chi + t^{2}}}\\
& < \sqrt{\psi + \phi^{2}} \cdot \sqrt{\chi + t^{2}} - J(t) \cdot t \cdot \phi\\
& = \sqrt{\psi + \phi^{2}} \cdot \left( \sqrt{\chi + t^{2}} - t \right) + t \left( \sqrt{\psi + \phi^{2}} - J(t) \cdot \phi \right).
\end{aligned}
\end{equation}
Using \eqref{limitation of J}, we can see the left-hand side of \eqref{alpha_decreasing proof-simplification-5} tends to $2 \lambda - 1 + \alpha$ as $t \to +\infty$. Suppose $\psi > 0$. We claim that the right-hand side of \eqref{alpha_decreasing proof-simplification-5} approaches positive infinity. In fact, by applying Lemma \ref{preparation result of alpha-decreasing 2} and \eqref{limitation of J}, we can see
\begin{equation*}
\lim_{t \to +\infty} \sqrt{\psi + \phi^{2}} - J(t) \cdot \phi \geq \frac{\psi}{\sqrt{\psi + \phi^{2}} + \left| \phi \right| }> 0,
\end{equation*}
and then we can have
\begin{equation*}
\begin{aligned}
&\lim_{t \to +\infty} \sqrt{\psi + \phi^{2}} \cdot \left( \sqrt{\chi + t^{2}} - t \right) + t \left( \sqrt{\psi + \phi^{2}} - J(t) \cdot \phi \right)\\
&=\lim_{t \to +\infty} \sqrt{\psi + \phi^{2}} \cdot \frac{\chi}{\sqrt{ \chi + t^{2}} + t} + t \cdot \frac{\psi + \left( 1 - J^{2}(t) \right)\cdot \phi^{2}}{\sqrt{\psi + \phi^{2}} + J(t) \cdot \phi} = +\infty.
\end{aligned}
\end{equation*}
Thus, \ZH{there exists} $t_{0} > 0$ such that \eqref{alpha_decreasing proof-simplification-5} is satisfied for any $t > t_{0}$. As a result, we have proved that, given any $\alpha \in \mathbb{R}$, the density $f_{\xi(x)}$ is $\alpha$-decreasing with some $t^{*}(\alpha) > 0$. Our next goal is to consider the case with $\psi = 0$. Due to $\Lambda > 0$, we only need to consider the cases with $\phi \neq 0$ under the assumption $\psi = 0$. We can reformulate \eqref{alpha_decreasing proof-simplification-5} as follows
\begin{equation}\label{alpha_decreasing proof-simplification-6}
\begin{aligned}
&J(t) \cdot \left( 2\lambda - 1 \right) + J(t) \cdot \alpha - J(t) \cdot \frac{\chi \cdot \left( 2\lambda - 1 \right)}{\chi + t^{2}} + \left| \phi \right| \cdot \frac{\chi}{\sqrt{\chi + t^{2}}}\\
&< \left| \phi \right| \cdot \left( \sqrt{\chi + t^{2}} - t \right) + t\cdot \left( \left| \phi \right| - J(t) \cdot \phi \right) =: Q(t).
\end{aligned}
\end{equation}
Suppose $\phi > 0$. Then, by using Lemma \ref{preparation result of alpha-decreasing 2} and \eqref{limitation of J}, we can see
\begin{equation*}
\lim_{t \to +\infty} Q(t) = \lim_{t \to +\infty} \left| \phi \right| \cdot \frac{\chi}{\sqrt{\chi + t^{2}} + t} + t \cdot \frac{\left( 1 - J^{2}(t) \right) \cdot \phi^{2}}{\left| \phi \right| + J(t) \cdot \phi} = \frac{2\lambda}{\left| \phi \right|} \cdot \frac{\phi^{2}}{2 \left| \phi \right|} = \lambda.
\end{equation*}
Notice that the left-hand side of \eqref{alpha_decreasing proof-simplification-6} approaches to $2\lambda - 1 + \alpha$. Thus, by using \eqref{alpha_decreasing proof-simplification-6} we have proved that given $\phi > 0$, for any $\alpha < 1$, the density function $f_{\xi(x)}$ is $\alpha$-decreasing with some $t^{*}(\alpha) > 0$. Suppose $\phi < 0$. Assume $\lambda \neq 0$. We can reformulate the function $Q$ as follows
\begin{equation*}
Q(t) = \left| \phi \right| \cdot \frac{\chi}{\sqrt{\chi + t^{2}} + t} + t \cdot \frac{\left( 1 - J^{2}(t) \right) \cdot \phi^{2}}{\left| \phi \right| - J(t) \cdot \left| \phi \right|}.
\end{equation*}
Using \eqref{the boundary of c_0} and \eqref{limitation of J}, we can see
\begin{equation*}
\lim_{t \to +\infty} \frac{\lambda}{1 - J(t)} = +\infty.
\end{equation*}
Then, by using Lemma \ref{preparation result of alpha-decreasing 2} again, we can yield
\begin{equation*}
\lim_{t \to +\infty} Q(t) = +\infty,
\end{equation*}
which implies that there exists $t_{0}>0$ such that \eqref{alpha_decreasing proof-simplification-6} is satisfied for any $t > t_{0}$. Assume $\lambda = 0$. Then, the function $Q$ defined by \eqref{alpha_decreasing proof-simplification-6} can be rewritten as
\begin{equation*}
Q(t) = \left| \phi \right| \cdot \frac{\chi}{\sqrt{ \chi + t^{2}} + t} + 2t \cdot \left| \phi \right|,
\end{equation*}
which implies $Q(t)$ tends to positive infinity as $t$ being large enough. As a result, we have proved that given $\phi < 0$, the density function $f_{\xi(x)}$ is $\alpha$-decreasing with some $t^{*}(\alpha)>0$ for any $\alpha \in \mathbb{R}$.

It remains to consider the case with $\Lambda = 0$, in which $\Lambda = 0$ means $\xi(x) \sim GH_{1}\left( \lambda, \chi, 0, 0, 1, 0 \right)$. Using Corollary \ref{density of special GH_1}, we can see the density of $\xi(x)$ can be written as follows
\begin{equation*}
f_{\xi(x)}(t) = \hat{c} \cdot \left( \chi + t^{2} \right)^{\lambda - \frac{1}{2}},
\end{equation*}
where $\hat{c}$ is defined by
\begin{equation*}
\hat{c}: = \frac{\Gamma \left( -\lambda + \frac{1}{2} \right)}{\Gamma \left( -\lambda \right)} \frac{1}{\sqrt{\pi}} \cdot \chi^{-\lambda}.
\end{equation*}
Then, we can yield
\begin{equation*}
\begin{aligned}
&t \cdot f^{'}_{\xi(x)} \left( t \right) + \alpha \cdot f_{\xi(x)} \left( t \right)\\
&= \hat{c} \cdot \left( \lambda - \frac{1}{2} \right) \cdot 2t^{2} \cdot \left( \chi + t^{2} \right)^{\lambda - \frac{3}{2}} + \alpha \cdot \hat{c} \cdot \left( \chi + t^{2} \right)^{\lambda - \frac{1}{2}}\\
&= \hat{c} \cdot \left( \chi + t^{2} \right)^{\lambda - \frac{3}{2}} \cdot \left[ \left( \lambda - \frac{1}{2} \right) \cdot 2t^{2} + \alpha \cdot \left( \chi + t^{2} \right) \right]\\
&= \hat{c} \cdot \left( \chi + t^{2} \right)^{\lambda - \frac{3}{2}} \cdot \left[ \left( 2\lambda - 1 + \alpha \right) \cdot t^{2} + \alpha \cdot \chi \right].
\end{aligned}
\end{equation*}
As a result, the inequality \eqref{alpha_decreasing proof} is equivalent to
\begin{equation}\label{alpha_decreasing proof equivalent form}
\left( 2\lambda - 1 + \alpha \right) \cdot t^{2} + \alpha \cdot \chi < 0.
\end{equation}
Since $\psi = 0$. Observing the Definition \ref{definition of Generalized hyperbolic distributions}, we know that we need to take $\chi > 0$ and $\lambda < 0$. As a result, we can see $2\lambda - 1 + \alpha = 0$ does not satisfy \eqref{alpha_decreasing proof equivalent form}. Since the left-hand side of \eqref{alpha_decreasing proof equivalent form} is a quadratic function. Thus, the density $f_{\xi(x)}$ is $\alpha$-decreasing with some $t^{*}\left( \alpha \right)$ if and only if $2\lambda - 1 + \alpha < 0$ i.e. $\alpha < 1 - 2\lambda$.
\end{proof}

\begin{proof}[Proof of Lemma \ref{lemma of general case with gamma times x greater than 0}]
Define
\begin{equation*}
\xi(x) = \frac{v^{T} x - \mu^{T} x}{\sqrt{x^{T} \Sigma x}}, \ \ g(x) = \frac{D - \mu^{T} x}{\sqrt{x^{T} \Sigma x}}.
\end{equation*}
From the Definition \ref{definition of Generalized hyperbolic distributions}, we know that
\begin{equation*}
v \overset{d}{=} \mu + W \cdot \gamma + \sqrt{W} A Z,
\end{equation*}
where $W \sim N^{-}\left( \lambda, \chi, \psi \right)$, $Z \sim N_{N}\left( 0, I_{N} \right)$, $A \in \mathbb{R}^{N \times N}$. Thus, by using Proposition \ref{1-dimensional generalized hyperbolic distributions}, we can yield
\begin{equation*}
\xi(x) \overset{d}{=} W \cdot \frac{x^{T} \gamma}{\sqrt{x^{T} \Sigma x}} + \sqrt{W} Y,
\end{equation*}
where $Y \sim N_{1}\left( 0, 1 \right)$. Since the normal distribution is a special elliptical distribution. By using \cite[Proposition 3.28.]{mcneil2005quantitative}, we can see
\begin{equation*}
Y \overset{d}{=} R \cdot S,
\end{equation*}
where $R$ is a radial random variable with the law $\mu_{R}$, and $S$ is uniformly distributed on the unit sphere $\left\lbrace s\in \mathbb{R}: \|s \| = 1 \right\rbrace$ with the law $\mu_{S}$. As a result, we can reformulate $\xi(x)$ as follows
\begin{equation*}
\xi(x) \overset{d}{=} W \cdot \frac{x^{T} \gamma}{\sqrt{x^{T} \Sigma x}} + \sqrt{W} R S,
\end{equation*}
where $\mathbb{P}\left( S = 1 \right) = \mathbb{P}\left( S = -1 \right) = 1/2$. Define $M(x)$ by
\begin{equation*}
M(x) := \left\lbrace \zeta \in \mathbb{R}: \zeta \leq g(x) \right\rbrace.
\end{equation*}
We can see that $M(x)$ is a convex set. The ray function of $M(x)$ can be written \ZH{as}
\begin{equation*}
\rho_{M(x)}(x, u) =
\left\{
        \begin{array}{l}
        \sup\limits_{t \geq 0} \ t\\            
        s.t. \ \ W \cdot \frac{x^{T} \gamma}{\sqrt{x^{T} \Sigma x}} + \sqrt{W} t u \leq g(x),
        \end{array}
\right.
\end{equation*}
which can be reformulated as follows
\begin{equation*}
\rho_{M(x)}(x, u) =
\left\{
        \begin{array}{ll}
        \frac{1}{\sqrt{W}} \cdot \frac{D - \left( \mu + W \gamma \right)^{T} x}{\sqrt{x^{T} \Sigma x}} & \text{if} \ \ u = 1 \ \ \text{and} \ \ \sqrt{W} > 0\\            
        +\infty, & \text{otherwise}.
        \end{array}
\right.
\end{equation*}
As a result, by using \cite[Lemma 2.1]{van2019eventual} and the convexity of $M(x)$, we can yield
\begin{equation*}
\begin{aligned}
\mathbb{P}\left( \xi(x) \in M(x) \right) &= E_{W}\left[ \int_{u\in S^{0}} \mu_{R} \left( \left\lbrace r \geq 0: W \cdot \frac{x^{T}\gamma}{\sqrt{x^{T}\Sigma x}} + \sqrt{W}\cdot r \cdot u \cap M \neq \emptyset \right\rbrace \right) d \mu_{S}\left( u \right)  \right]\\
&= E_{W}\left[ \frac{1}{2} \mu_{R} \left\lbrace r \geq  0: W \cdot \frac{x^{T} \gamma}{\sqrt{x^{T}\Sigma x}} + r \cdot \sqrt{W}\cap M \neq \emptyset \right\rbrace \right.\\
&+ \left. \frac{1}{2} \mu_{R} \left\lbrace r \geq  0: W \cdot \frac{x^{T} \gamma}{\sqrt{x^{T}\Sigma x}} - r \cdot \sqrt{W}\cap M \neq \emptyset \right\rbrace \right]\\
&= E_{W} \left[ \frac{1}{2} F_{R}\left( \rho\left( x, 1 \right) \right) \right] + \frac{1}{2} = \int^{+\infty}_{0} \frac{1}{2} F_{R}\left( \rho\left( x, 1 \right) \right) d\mu_{W}\left( \omega \right) + \frac{1}{2}.
\end{aligned}
\end{equation*}
\end{proof}

\section{Proof of Section \ref{Convexity of the feasible set $S(p)$} \ZHnew{results}}\label{Appendix HH}

\

\begin{proof}[Proof of Lemma \ref{the construction of copula}]
To satisfy the constraints of \eqref{Fea-6}, we restrict ourselves here to $y_{i} := 1/K$ for the sake of simplicity. Considering the value of $d$ in Lemma \ref{Lemma end 2}, we divide it into the following cases.\\
\textbf{Case 1:} $d \in (0, 1)$. To prove $d \cdot \kappa(x) H_{x}\kappa(x) - \nabla_{x} \kappa(x) (\nabla_{x}\kappa(x))^{T}$ is a positive definite matrix, due to the Schur's complement, it suffices to prove $\triangle$ is a positive definite matrix. Before giving the formula of $\kappa$, we consider the characteristics of the matrix $\triangle$. To prove $H_{x}\kappa$ is semi-positive definite, we consider variable separable functions, which means that $H_{x}\kappa$ enables to be positive diagonal matrices. To this end, we set $\kappa(x) := \Sigma^{K}_{i = 1} f(x_{i})$, where $f$ is a strictly convex, positive, and second differentiable function. We can get
\begin{equation*}
\nabla_{x}\kappa(x) = \left\lbrace f'(x_{1}), f'(x_{2}), ... , f'(x_{K}) \right\rbrace^{T},
\end{equation*}
\begin{equation*}
H_{x}\kappa(x) = 
\begin{pmatrix}
f''(x_{1}) & 0 & 0 & \cdots & 0\\
0 & f''(x_{2}) & 0 & \cdots & 0\\
0 & 0 & f''(x_{3}) & \cdots & 0\\
\vdots & \vdots & \vdots & \ddots & \vdots\\
0 & 0 & 0 & \cdots & f''(x_{K})
\end{pmatrix},
\end{equation*}
\begin{equation*}
\triangle = 
\begin{pmatrix}
f''(x_{1}) & 0 & \cdots & 0 & f'(x_{1})\\
0 & f''(x_{2}) & \cdots & 0 & f'(x_{2})\\
\vdots & \vdots & \ddots & \vdots & \vdots\\
0 & 0 & \cdots & f''(x_{K}) & f'(x_{K})\\
f'(x_{1}) & f'(x_{2}) & \cdots & f'(x_{K}) & d\cdot \left( \sum\limits^{K}_{i = 1} f(x_{i}) \right)
\end{pmatrix}.
\end{equation*}
Thus, we have
\begin{equation}\label{the det of complex matrix 1}
\begin{array}{ll}
\begin{vmatrix}
\triangle
\end{vmatrix}
&= 
\begin{vmatrix}
f''(x_{1}) & 0 & \cdots & 0 & f'(x_{1})\\
0 & f''(x_{2}) & \cdots & 0 & f'(x_{2})\\
\vdots & \vdots & \ddots & \vdots & \vdots\\
0 & 0 & \cdots & f''(x_{K}) & f'(x_{K})\\
f'(x_{1}) & f'(x_{2}) & \cdots & f'(x_{K}) & d\cdot \left( \sum\limits^{K}_{i = 1} f(x_{i}) \right)
\end{vmatrix} 
\\
& \ \\ 
&=  
\begin{vmatrix}
f''(x_{1}) & 0 & \cdots & 0 & f'(x_{1})\\
0 & f''(x_{2}) & \cdots & 0 & f'(x_{2})\\
\vdots & \vdots & \ddots & \vdots & \vdots\\
0 & 0 & \cdots & f''(x_{K}) & f'(x_{K})\\
0 & 0 & \cdots & 0 & \sum\limits^{K}_{i = 1} \left( d \cdot f(x_{i}) - \frac{f'(x_{i})^{2}}{f''(x_{i})} \right)
\end{vmatrix}
\\
& \ \\ 
&= \left[ \prod\limits^{K}_{i = 1} f''(x_{i}) \right] \cdot \left[ \sum\limits^{K}_{i = 1} \left( d \cdot f(x_{i}) - \frac{f'(x_{i})^{2}}{f''(x_{i})} \right) \right].
\end{array}
\end{equation}
Since $f$ is assumed to be strictly convex, it suffices to prove the second term of the last line in \eqref{the det of complex matrix 1} is positive. 
To this end, consider $a>1$ and a positive function $f$ satisfying $f''>0$ and
\begin{equation}\label{differential equation}
\frac{d}{a} \cdot f \cdot f'' = (f')^{2}.
\end{equation}
In fact, if $f$ satisfies \eqref{differential equation}, then we can yield
\begin{equation}\label{related to differential equation 1}
\begin{array}{ll}
\begin{vmatrix}
\triangle
\end{vmatrix}
&= \left[ \prod\limits^{K}_{i = 1} f''(x_{i}) \right] \cdot \left[ \sum\limits^{K}_{i = 1} \left( d \cdot f(x_{i}) - \frac{f'(x_{i})^{2}}{f''(x_{i})} \right) \right]\\
& \ \\
&> \left[ \prod\limits^{K}_{i = 1} f''(x_{i}) \right] \cdot \left[ \sum\limits^{K}_{i = 1} \left( \frac{d}{a} \cdot f(x_{i}) - \frac{f'(x_{i})^{2}}{f''(x_{i})} \right) \right] = 0,
\end{array}
\end{equation}
which completes the proof. Here, we can take $a = 2$ and obtain a non-trivial solution of \eqref{differential equation} as follows
\begin{equation}\label{solution of differential equaiton}
f(x) := \left( \left( \frac{2}{d} - 1 \right)\cdot \left( C_{2} + C_{1} x \right) \right)^{\frac{d}{d - 2}},
\end{equation}
where $C_{1}$, $C_{2}$ are constant coefficients to be given. Define the minimum value of the vectors' components in $X$ as
\begin{equation*}
X_{min} := \underset{ i = 1, ... , K }{\inf} \left\lbrace x_{i} \left|  \forall x = [x_{1}, x_{2}, ..., x_{K}]^{T} \in X \right\rbrace\right.
\end{equation*}
Take arbitrary $d\in(0, 1)$. Choose $C_{1}, C_{2}>0$ such that $X_{min} > d/((2-d)\cdot C_{1} \cdot K) -C_{2} / C_{1} $. Substitute \eqref{solution of differential equaiton} into \eqref{the det of complex matrix 1} and \eqref{related to differential equation 1}. Then, for any $x \in X$ we can get
$\begin{vmatrix}
\triangle
\end{vmatrix} > 0$ 
, $f > 0$, and $f''>0$. It follows that $\triangle$ and $H_{x}\kappa(x)$ are both positive definite matrices. 
\\
\textbf{Case 2:} $d \in [1, +\infty)$. Similar to the process of Case 1, by taking $a = 2d$ into \eqref{differential equation}, we can yield
\begin{equation*}
\frac{1}{2} \cdot f \cdot f'' = (f')^{2},
\end{equation*}
and its solution defined by
\begin{equation*}
f(x) := \frac{1}{C_{2} + C_{1}x}.
\end{equation*}
\end{proof}

\end{appendices}	

\end{document}